\documentclass[12pt]{amsart}
\usepackage{fullpage}
\usepackage{comment}
\usepackage{color}
\usepackage{enumerate}
\usepackage{ amssymb }
\usepackage{ amsthm}
\usepackage[normalem]{ulem}
\usepackage{thmtools}

\usepackage{amsfonts,txfonts,psfrag,color}
\usepackage[dvips]{graphicx}
\usepackage{epstopdf}
\usepackage{caption}
\usepackage[dvipsnames]{xcolor}
\usepackage[mathscr]{euscript}

\numberwithin{equation}{section}

\makeatletter
\renewcommand{\@secnumfont}{\bfseries}
\makeatother

\newtheorem{thm}{Theorem}[section]    
\newtheorem{lem}[thm]{Lemma}         
\newtheorem{prop}[thm]{Proposition}        
\newtheorem{coro}[thm]{Corollary}
\newtheorem{conj}[thm]{Conjecture}          
\theoremstyle{definition}
\newtheorem{defn}[thm]{Definition}  

\newtheorem{rmk}[thm]{Remark}

\newcommand{\N}{\mathbb{N}}
\newcommand{\Z}{\mathbb{Z}}

\newcommand{\R}{\mathbb{R}}

\newcommand{\E}{\mathbb{E}}
\newcommand{\prob}{\mathbb{P}}
\newcommand{\n}{\noindent}
\newcommand{\B}{\textbf}

\newcommand{\e}{\epsilon}
\newcommand{\al}{\alpha}
\newcommand{\vp}{\varphi}
\newcommand{\cal}{\mathcal}

\newcommand{\ov}{\overline}

\newcommand{\sig}{\sigma}
\newcommand{\pa}{\partial}

\newcommand{\om}{\omega}

\newcommand{\var}{{\rm \mathbb{V}ar}}
\newcommand{\Th}{\text{Th\'{e}ret}}

\newcommand{\wh}{\widehat}
\newcommand{\Chee}{\wh{\Phi}_{n}}
\newcommand{\1}{\B{\rm \B{1}}}

\newcommand{\un}{\underline}

\newcommand{\giant}{\text{\rm\B{C}}_n}
\newcommand{\dw}{\frak{d}}

\newcommand{\per}{\text{\rm per}}
\newcommand{\cyl}{\text{\rm \small \textsf{cyl}}}
\newcommand{\hyp}{\text{\rm \small \textsf{hyp}}}
\newcommand{\slab}{\text{\rm \small \textsf{slab}}}
\newcommand{\hemi}{\Xi_{\text{\rm \textsf{hemi}}}}
\newcommand{\face}{\Xi_{\text{\rm \textsf{face}}}}

\newcommand{\brid}{\text{\rm \small \textsf{bridge}}}
\newcommand{\str}{\text{\rm \small \textsf{string}}}
\newcommand{\web}{\text{\rm \small \textsf{web}}}
\newcommand{\hull}{\text{\rm \small \textsf{hull}}}

\newcommand{\dcyl}{\text{\rm \small \textsf{d-cyl}}}
\newcommand{\dhemi}{\text{\rm \small \textsf{d-hemi}}}
\newcommand{\dface}{\text{\rm \small \textsf{d-face}}}
\newcommand{\rmE}{\text{{\rm E}} }
\newcommand{\rmV}{\text{{\rm V}} }
\newcommand{\sfS}{{\text{\rm\textsf{S}}}}
\newcommand{\wt}{\widetilde}
\newcommand{\slice}{\text{\rm \small \textsf{slice}}}

\definecolor{azure(colorwheel)}{rgb}{0.0, 0.5, 1.0}
\definecolor{hanpurple}{rgb}{0.32, 0.09, 0.98}
\definecolor{iris}{rgb}{0.35, 0.31, 0.81}

\begin{document}
\title{Isoperimetry in supercritical bond percolation \\ in dimensions three and higher}
\author{Julian Gold}

\maketitle

\begin{abstract}

We study the isoperimetric subgraphs of the infinite cluster $\B{C}_\infty$ for supercritical bond percolation on~$\Z^d$ with $d\geq 3$. Specifically, we consider subgraphs of $\B{C}_\infty \cap [-n,n]^d$ having minimal open edge boundary to volume ratio. We prove a shape theorem for these subgraphs: when suitably rescaled, they converge almost surely to a translate of a deterministic shape. This deterministic shape is itself an isoperimetric set for a norm we construct. As a corollary, we obtain sharp asymptotics on a natural modification of the Cheeger constant for $\B{C}_\infty \cap [-n,n]^d$, settling a conjecture of Benjamini for the version of the Cheeger constant defined here.
\end{abstract}

\vspace{5mm}

\newpage

\tableofcontents 


{\large\section{\B{Introduction and results}}\label{sec:introduction}}

\subsection{Motivation}\label{sec:intro_motivation} Isoperimetric problems, namely the problem of finding a set of given size and minimal boundary measure, have been studied for millennia. In the continuum, such problems are the subject of geometric measure theory and the calculus of variations. Isoperimetric inequalities give a lower bound on the boundary measure of a set in terms of the volume measure of the set. Their applications in mathematics range from concentration of measure to PDE theory. 

Isoperimetric problems are also well-studied in the discrete setting. One can encode isoperimetric inequalities for graphs in the \emph{Cheeger constant}, or modifications thereof. Define the Cheeger constant of a graph $G$ to be
\begin{align}
\Phi_G  := \min \left\{ \frac{ |\pa_G H |}{| H|} : H \subset G, 0<  |H| \leq |G|/2 \right\} \,,
\end{align}
where $\pa_G H$ is the edge boundary of $H$ in the graph $G$ and where $|H|$ and $|G|$ respectively denote cardinalities of the vertex sets of $H$ and $G$. Introduced in the context of manifolds in Cheeger's thesis \cite{Cheeger}, the Cheeger constant was used to give a lower bound on the smallest positive eigenvalue of the negative Laplacian. Its discrete analogue, introduced by Alon \cite{Alon}, plays a similar role in spectral graph theory (see for instance Chapter 2 of \cite{FanChung}). Indeed, Cheeger's inequality and its variants are used to study mixing times of random walks and Markov chains. Ultimately, the Cheeger constant provides one of many ways to study the geometry of a~graph. 

The goal of this paper is to explore the geometry of random graphs arising from bond percolation on $\Z^d$. We view $\Z^d$ as a graph, with edge set $\rmE(\Z^d)$ determined by nearest-neighbor pairs, and we form the probability space $( \{0,1\}^{\rmE(\Z^d)}, \mathscr{F}, \prob_p)$, where $\mathscr{F}$ denotes the product $\sigma$-algebra on $\{0,1\}^{\rmE(\Z^d)}$ and where $\prob_p$ is the product Bernoulli measure associated to the \emph{percolation parameter} $p \in [0,1]$. Elements $\om = (\om_e)_{e \in \rmE(\Z^d)}$ of the probability space are \emph{percolation configurations}. An edge $e \in \rmE(\Z^d)$ is \emph{open} in the configuration $\om$ if $\om_e =1$ and is \emph{closed} otherwise. The open edges determine a random subgraph of $\Z^d$ whose connected components are called \emph{open clusters}. It is well-known (see Grimmett \cite{Grimmett} for details) that when $d\geq 2$, bond percolation exhibits a phase transition: there is a $p_c(d) \in (0,1)$ so that whenever $p > p_c(d)$, there exists a unique infinite open cluster $\prob_p$-almost surely, and whenever $p < p_c(d)$, there is no infinite open cluster $\prob_p$-almost surely. We work in the supercritical ($p > p_c(d)$) regime, and denote the unique infinite open cluster by $\B{C}_\infty$. 

We may now be more specific: our goal is to explore the geometry of $\B{C}_\infty$. There are many ways to do this, for example, one can study the asymptotic graph distance in $\B{C}_\infty$ (e.g. Antal and Pisztora  \cite{Antal_Pisztora}), the asymptotic shapes of balls in the graph distance metric of $\B{C}_\infty$ (e.g. Cox and Durrett \cite{Cox_Durrett}), or the effective resistance of $\B{C}_\infty$ within a large box (e.g. Grimmett and Kesten \cite{Grimmett_Kesten}). Our aim is to study the isoperimetry of $\B{C}_\infty$ through the Cheeger constant. 

By definition, $\Phi_G = 0$ for any amenable graph, and one can show that $\Phi_{\B{C}_\infty} = 0$ almost surely. We instead study the Cheeger constant of $\giant : =\B{C}_\infty \cap [-n,n]^d$. Let $\wt{\B{C}}_n$ be the largest connected component of $\giant$. It is known (Benjamini and Mossel \cite{BenjMo}, Mathieu and Remy \cite{Mathieu_Remy}, Rau \cite{Rau}, Berger, Biskup, Hoffman and Kozma \cite{BBHK} and Pete \cite{Pete}) that $\Phi_{\wt{\B{C}}_n} \asymp n^{-1}$ as $n \to \infty$, prompting the following conjecture of Benjamini. 

\begin{conj} For $p > p_c(d)$ and $d \geq 2$, the limit 
\begin{align}
\lim_{n \to \infty} n\, \Phi_{\wt{\B{C}}_n}
\end{align}
exists $\prob_p$-almost surely and is a positive deterministic constant.
\label{benjamini}
\end{conj}

Procaccia and Rosenthal \cite{Procaccia_Rosenthal} made progress towards resolving this conjecture: they proved upper bounds on the variance of the Cheeger constant, showing $\var(n \Phi_{\wt{\B{C}}_n} ) \leq cn^{2-d}$ for some positive $c = c(p,d)$. Recently, Biskup, Louidor, Procaccia and Rosenthal \cite{BLPR} settled this conjecture positively for a natural modification of $\Phi_{\wt{\B{C}}_n}$ in dimension two. Define the \emph{modified Cheeger constant} $\Chee$ of $\B{C}_n$ in dimensions $d \geq 2$:
\begin{align}
\Chee  := \min \left\{ \frac{ |\pa_{\B{C}_\infty} H |}{| H|} : H \subset \giant, 0 < |H| \leq |\giant|/d! \right\} \,,
\label{eq:modified_chee}
\end{align}
where $\pa_{\B{C}_\infty} H$ denotes the open edge boundary of $H$ within \emph{all of} $\B{C}_\infty$ as opposed to $\giant$. This modification is natural in the sense that subgraphs $H$ are treated as living within $\B{C}_\infty$, and the $d!$ in the volume upper bound ensures that $H$ need not touch the boundary of the box. Thanks to Proposition 1.2 of \cite{BenjMo}, the asymptotics of $\Chee$ are unchanged whether we use $\giant$ or $\wt{\B{C}}_n$ in \eqref{eq:modified_chee}.

Both $\Phi_{\wt{\B{C}}_n}$ and $\Chee$ are closely related to the so-called \emph{anchored isoperimetric profile}, defined in the context of the infinite cluster as 
\begin{align} 
\Phi_{\B{C}_\infty,0}(n)  := \inf \left\{ \frac{ |\pa_{\B{C}_\infty} H| }{|H|} : 0 \in H \subset \B{C}_\infty, H \text{ connected}, 0 < |H| \leq n \right\} \,,
\end{align}
where of course we must condition on the positive probability event $\{ 0 \in \B{C}_\infty\}$. In \cite{BLPR}, the analogue of Benjamini's conjecture for the anchored isoperimetric profile was also established in dimension two. Moreover, the subgraphs of $\B{C}_n$ and $\B{C}_\infty$ achieving each minimum were studied in both cases, and in fact were shown to scale uniformly to the same deterministic limit shape. This \emph{shape theorem} implies the existence of the limit in Conjecture \ref{benjamini} for (\ref{eq:modified_chee}). Indeed, the perimeter of this limit shape appears in the limiting value of the modified Cheeger constant.

\subsection{Results}

We extend the work of \cite{BLPR} to the setting $d \geq 3$ by settling Benjamini's conjecture for the modified Cheeger constant and by proving a shape theorem for isoperimetric subgraphs of $\giant$. As the arguments in \cite{BLPR} rely heavily on planar geometry and graph duality, a much different approach is needed. Nevertheless, we share a common starting point in the Wulff construction, described below, and there are structural similarities between both arguments. 

We state the main theorem of the paper first. For each $n$, let $\cal{G}_n$ be the (random) collection of subgraphs of $\B{C}_n$ realizing the minimum in \eqref{eq:modified_chee}. For $A \subset \R^d$, $r >0$ and $x \in \R^d$ the sets $rA$ and $x +A$ are defined as usual by
\begin{align}
rA := \Big \{ ra : a \in A \Big \}\,, \hspace{5mm} x + A := \Big  \{ x+a : a \in A \Big\} \,,
\end{align}
and we write $|| \cdot ||_{\ell^1}$ to denote the $\ell^1$-norm on $\Z^d$. Here is our main result:
\begin{thm} Let $d \geq 3$ and $p > p_c(d)$. There is a deterministic, convex set $W_{p,d} \subset [-1,1]^d$ so that
\begin{align}
\max_{G_n \in \cal{G}_n} \inf_{x \in \R^d} n^{-d} \big \| \1_{G_n} - \1_{\giant \cap (x + nW_{p,d})} \big \|_{\ell^1} \xrightarrow[n \to \infty]{}  0
\end{align}
holds $\prob_p$-almost surely.
\label{main_L1}
\end{thm}

Following \cite{BLPR}, we build the limit shape $W_{p,d}$ through what is known as the Wulff construction, a method for solving anisotropic isoperimetric problems first introduced by Wulff \cite{Wulff} in 1901. Given a norm $\tau$ on $\R^d$, one can form an associated isoperimetric problem:
\begin{align}
\text{minimize } \frac{\cal{I}_\tau(E)}{\cal{L}^d(E) }  \hspace{5mm} \text{subject to } \cal{L}^d(E) \leq 1,
\label{eq:section1.1_iso}
\end{align}
ranging over $E \subset \R^d$ with Lipschitz boundary, where $\cal{L}^d$ denotes $d$-dimensional Lebesgue measure, and where $\cal{I}_\tau(E)$ is defined as
\begin{align}
\cal{I}_\tau(E) := \int_{\pa E} \tau(v_E(x) )  \cal{H}^{d-1}(dx) \,.
\label{eq:section1.1_functional}
\end{align}
Here $\cal{H}^{d-1}$ is the $(d-1)$-dimensional Hausdorff measure on $\pa E$ and $v_E(x)$ the unit exterior normal to $E$ at the point $x \in \pa E$, which is defined for $\cal{H}^{d-1}$-almost every point of $\pa E$. Wulff's isoperimetric set is the following intersection of half-spaces:
\begin{align}
\wh{W}_\tau := \bigcap_{v \in \mathbb{S}^{d-1}} \Big\{ x \in \R^d \:: x \cdot v \leq \tau(v) \Big\}\,,
\label{eq:section1.1_shape}
\end{align}
where $\cdot$ denotes the standard dot product in $\R^d$, and where $\mathbb{S}^{d-1}$ is the unit sphere in $\R^d$. We call $\wh{W}_\tau$ the \emph{unit Wulff crystal} associated to $\tau$; this object is the unit ball in the norm $\tau'$ dual to $\tau$ (recall that $\tau'$ is defined on $y \in \R^d$ by $\tau'(y) =  \sup\{ x \cdot y : x \in\R^d, \tau(x) \leq 1 \}$). When $\wh{W}_\tau$ is scaled to have unit volume, it becomes a candidate minimizer for (\ref{eq:section1.1_iso}). Taylor \cite{T2} ultimately proved this rescaled shape is optimal within a wide class of Borel sets, and moreover (in \cite{T3}) that this rescaled shape is the unique optimizer up to translations and modifications on a null set.

The Wulff construction is relevant because a norm emerges naturally when our problem is viewed correctly. Denoted $\beta_{p,d}$, this norm is first defined on $\mathbb{S}^{d-1}$: in a given direction $v \in \mathbb{S}^{d-1}$, first rotate a large cube so that its top and bottom faces are normal to $v$, then consider the restriction of percolation to this rotated cube. The minimum size of a cutset separating the faces of the cube in this percolated graph functions as a discrete surface energy. By requiring these cutsets to be anchored near the middle of the cube, we may employ a subadditivity argument and extract a limit as the diameter of the cube tends to infinity. This homogenized surface energy is $\beta_{p,d}(v)$.

We build $\beta_{p,d}$ in Section \ref{sec:norm}, and we define the \emph{Wulff crystal} $W_{p,d}$ to be the dilate of the unit Wulff crystal $\wh{W}_{p,d} $ associated via (\ref{eq:section1.1_shape}) to $\beta_{p,d}$ so that $\cal{L}^d(W_{p,d}) = 2^d /d!$. The Wulff crystal is then the limit shape from Theorem \ref{main_L1}, and we note that the norm $\beta_{p,d}$ gives rise to a \emph{surface energy functional} of the form (\ref{eq:section1.1_functional}) denoted $\cal{I}_{p,d}$. As in \cite{BLPR}, the shape theorem we present is intimately linked with the limiting value of the Cheeger constant. Let $\theta_p(d) := \prob_p( 0 \in \B{C}_\infty)$ be the density of the infinite cluster within $\Z^d$.

\begin{thm} Let $d \geq 3$, $p > p_c(d)$ and let $\beta_{p,d}$ be the norm to be constructed in Proposition~\ref{beta}. Let $W_{p,d}$ be the Wulff crystal for this norm, that is, the ball in the dual norm $\beta_{p,d}'$ such that $\cal{L}^d(W_{p,d}) = 2^d/d!$. Then, 
\begin{align}
\lim_{n\to \infty} n\, \Chee =  \frac{ \cal{I}_{p,d}(W_{p,d}) }{ \theta_p(d) \cal{L}^d(W_{p,d} )}
\end{align}
holds $\prob_p$-almost surely.
\label{main_benj}
\end{thm}

\subsection{History and discussion}\label{sec:intro_history}  Within the last thirty years, the Wulff construction has grown into an important tool in the rigorous analysis of equilibrium crystal shapes. Such problems are concerned with understanding the macroscopic behavior of one phase of matter immersed within another. 

The present work fits into this paradigm in that we may regard each Cheeger optimizer $G_n$ as a large droplet of a \emph{crystalline} phase within $\B{C}_\infty \setminus G_n$, regarded as the \emph{ambient} phase. The value of the norm $\beta_{p,d}$ in a given direction represents the energy required to form a flat interface between the two phases in this direction, and gives rise to a surface energy functional of the form (\ref{eq:section1.1_functional}). It was Gibbs \cite{Gibbs} who postulated that, in general, the asymptotic shape of the crystalline phase should minimize this surface energy. The Wulff construction furnishes this minimal shape. 

The spirit of Theorem \ref{main_L1} can be traced back to the work of Minlos and Sinai \cite{Milnos_Sinai_1,Milnos_Sinai_2} from the 1960s, in which the geometric properties of phase separation in a material are rigorously studied. The first rigorous characterizations of phase separation via the Wulff construction are due independently to Dobrushin, Koteck\'{y} and Shlosman \cite{DKS} in the context of the two-dimensional Ising model and to Alexander, Chayes and Chayes \cite{Alexander_Chayes_Chayes} in the context of two-dimensional bond percolation. The results of \cite{DKS}, valid in the low-temperature regime, were extended up to the critical temperature thanks to the work of Ioffe \cite{Ioffe} and Ioffe and Schonmann \cite{Ioffe_Schonmann}. 

The first rigorous derivation of the Wulff construction for a genuine short-range model in three dimensions was achieved by Cerf in the context of bond percolation \cite{Cerf_3D}. Analogous results for the Ising model and in higher dimensions were achieved in several substantial works of Bodineau \cite{Bodineau1,Bodineau2} and Cerf and Pisztora \cite{Cerf_Pisztora_2,Cerf_Pisztora}. The coarse graining results of Pisztora \cite{P} played an integral role in this study of the Ising model, FK percolation and bond percolation in higher dimensions. A comprehensive survey of these results and of others can be found in Section 5.5 of Cerf's monograph \cite{stflour} and in the review article of Bodineau, Ioffe and Velenik \cite{BIV}. 

In all cases, the jump to dimensions strictly larger than two has, at least so far, necessitated a shift from the uniform topology to the $\ell^1$-topology on the space of shapes (we are intentionally vague about which space we consider). Indeed, the variational problem (\ref{eq:section1.1_iso}) is not stable in $d\geq 3$ when the space of shapes is equipped with the uniform topology: it is possible to construct a sequence of shapes bounded away from the optimal shape in the uniform topology, but whose surface energies tend to the optimal surface energy. This has implications at the microscopic level; to prove a uniform shape theorem in $d\geq 3$ for the Cheeger optimizers, one would first have to rule out the existence of long thin filaments (as in Figure \ref{fig:filament}) in these discrete objects with high probability. 

\begin{figure}[h]
\centering
\includegraphics[scale=0.75]{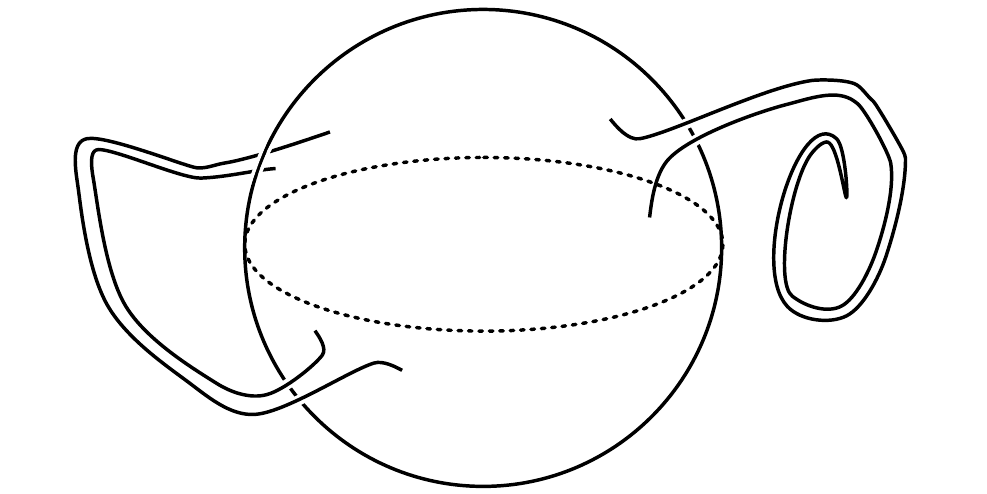}
\caption{In $d = 3$, filaments added to the optimal shape for the Euclidean isoperimetric problem produce a set which is almost optimal and and yet has large uniform distance to the sphere.} 
\label{fig:filament}
\end{figure}

This lack of regularity at the microscopic level requires that we consider the variational problem over a wider class of shapes, and it is here that geometric measure theory emerges as a valuable tool, as first realized by Alberti, Bellettini, Cassandro and Presutti \cite{ABCP}.

\subsection{Outline}

Our goals may be summarized as follows: we wish to show that the sequence of discrete, random isoperimetric problems (\ref{eq:modified_chee}) scale to a continuous, deterministic isoperimetric problem (\ref{eq:section1.1_iso}) corresponding to some norm $\beta_{p,d}$ on $\R^d$. We do not use the language of $\Gamma$-convergence, though this has been used in recent related work of Braides and Piatnitski \cite{Braides_Piatnitski_1,Braides_Piatnitski_2}. 

The first task is to construct a suitable norm $\beta_{p,d}$ on $\R^d$, done in Section \ref{sec:norm} after introducing some definitions and notation in Section \ref{sec:notation}. The key to the existence of $\beta_{p,d}$ is a spatial subadditivity argument applied to the geometric setting described briefly before Theorem \ref{main_benj}. 

The resulting norm $\beta_{p,d}$ gives rise to a surface energy $\cal{I}_{p,d}$, and the remainder of the paper is concerned with demonstrating that the unique optimizer of the isoperimetric problem associated to $\cal{I}_{p,d}$ faithfully describes the macroscopic shape of each large $G_n \in \cal{G}_n$. We must show a correspondence between discrete objects (the various subgraphs of $\giant$) and continuous objects (Borel subsets of $[-1,1]^d$ for which isoperimetric problems can be defined). This correspondence should be strong enough to link the isoperimetric ratio of subgraphs of $\giant$ to the ratio for continuous objects, as in the limiting value of Theorem~\ref{main_benj}. 

Concentration estimates proved in Section \ref{sec:concentration} allow us to pass from continuous objects to discrete objects in Section \ref{sec:consequences}, yielding high probability upper bounds on $\Chee$. This is in line with the strategy of \cite{BLPR}, and is in contrast to large deviation methods used in some of the work referenced in Section~\ref{sec:intro_history}, where the nature of these earlier problems requires working within events of small probability. All arguments presented up to this point work in the setting $d \geq 2$. 

Passing from discrete objects to continuous objects is more delicate, and requires a renormalization argument given in Section \ref{sec:coarse_original}. We base our argument on a construction from an unpublished note of Zhang \cite{Zhang}, but we must improve this construction and study it carefully in order to apply it to our situation. It is here that, for reasons which will be made clear in Section \ref{sec:coarse_applied}, we must restrict ourselves to the setting $d \geq 3$. This is no loss as the case $d=2$ is covered by results in \cite{BLPR}. 

In Section \ref{sec:contiguity}, we reap the efforts of Section \ref{sec:coarse_applied}, passing from $G_n \in \cal{G}_n$ to sets of finite perimeter (defined in Section \ref{sec:notation}). Such sets have just enough regularity that we may work locally on their boundaries. We exploit this in Section \ref{sec:final} to show whenever a $G_n$ is close to a set of finite perimeter, the surface energy of this set is roughly a lower bound on the open edge boundary of $G_n$. Our notion of closeness allows us to relate the volumes of these objects; we may then deduce that whenever $G_n$ is close to a set of finite perimeter, the isoperimetric ratio of $G_n$ (hence the Cheeger constant) is controlled from below by the isoperimetric ratio of the given continuum set. 

Invoking the results of Section \ref{sec:consequences} and the work of Taylor \cite{T2,T3}, we find that with high probability, each $G_n$ must be close to the Wulff crystal, giving Theorem \ref{main_L1} and Theorem \ref{main_benj} in quick succession. 

\subsection{Open problems}\label{sec:intro_5}

We pose several open questions, some of which were stated in~\cite{BLPR}. \\

\n\emph{(1) Boundary conditions and more general domains: } Conjecture \ref{benjamini} was recently settled for the unmodified Cheeger constant in dimension two \cite{Gold2}, though it remains to link the set of limit shapes in this case with the Wulff shape. Motivated by the Winterbottom construction (see \cite{Winterbottom,Pfister_Velenik_2,Pfister_Velenik_1,BIV_winter}), we conjecture the limit shapes in this case are rescaled quarter-Wulff crystals. 

One can generalize Benajmini's conjecture in the two-dimensional setting to domains other than boxes: given a nice bounded open set $\Omega \subset \R^2$, one can study the asymptotics of the unmodified Cheeger constant as well as the shapes of the Cheeger optimizers for the largest connected component of ${\B{C}}_\infty \cap n\Omega$. \\


\n\emph{(2) More information on the Wulff crystal: } Little is known about the geometric properties of the Wulff crystal. One recent result of Garet, Marchand, Procaccia and $\Th$ \cite{GMPT} is that, in two dimensions, the Wulff crystal varies continuously with respect to the uniform metric on compact sets as a function of the percolation parameter $p \in (p_c(2),1]$. It was conjectured in \cite{BLPR} that the two-dimensional Wulff crystal tends to a Euclidean ball as $p \downarrow p_c(2)$; this is still widely open. It is natural to ask whether the Wulff crystal has facets (open portions of the boundary with zero curvature) or corners, and how such questions depend on the percolation parameter. \\

\n\emph{(3) Uniform convergence for $d\geq 3$: } An interesting and challenging question is whether a form of Theorem~\ref{main_L1} holds in $d \geq 3$ when we replace $\ell^1$-convergence by uniform convergence. Such a result would have to overcome the challenges outlined in Section~\ref{sec:intro_history}. 

\subsection{Acknowledgements}

I thank my advisor Marek Biskup for suggesting this problem, for his guidance and his support. I am deeply indebted to Rapha\"{e}l Cerf for patiently sharing his expertise and insight during my time in Paris. I likewise thank Eviatar Procaccia for his guidance. I am grateful to Vincent Vargas, Claire Berenger and Shannon Starr for making it possible for me to attend the IHP Disordered Systems trimester. I thank Yoshihiro Abe, Ian Charlesworth, Arko Chatterjee, Hugo Duminil-Copin, Aukosh Jagannath, Ben Krause, Sangchul Lee, Tom Liggett, Jeff Lin, Peter Petersen, Jacob Rooney and Ian Zemke for helpful conversations. Finally, I wish to express my gratitude to anonymous referees for their helpful comments. This research has been partially supported by the NSF grant DMS-1407558.




{\large\section{\B{Definitions and notation}}\label{sec:notation}}

\subsection{Paths, boundaries, cutsets}\label{sec:p_b_c}

We work almost entirely within the graph $\Z^d$, whose vertex set consists of all integer $d$-tuples, and where there is an edge between two vertices if their Euclidean distance is one. Edges have no orientation, and if vertices $x$ and $y$ are adjacent, we write $x \sim y$. 

A \emph{path} between vertices $x$ and $y$ in $\Z^d$ is a finite, alternating sequence of vertices and edges $\gamma = (x \equiv x_0, e_1, x_1, \dots , e_m, x_m \equiv y)$ such that $e_i$ joins $x_{i-1}$ and $x_i$ for $i = 1, \dots, m$.  The path $\gamma$ \emph{joins} $x$ and $y$, and the \emph{length} of $\gamma$ is $m$. A subgraph $G = (\rmV(G), \rmE(G))$ of $\Z^d$ is \emph{connected} if for any vertices $x,y \in G$, there is a path using only vertices and edges of $G$ joining $x$ and $y$. For a vertex $x \in \Z^d$, a \emph{path from $x$ to $\infty$} is an infinite alternating sequence of vertices and edges $\gamma = ( x \equiv x_0, e_1, x_1, \dots)$ where no finite box contains all $e_i$. In both the finite and infinite cases, a path is \emph{simple} if it uses each vertex no more than once, and paths are often regarded as sequences of edges out of convenience.

For the rest of this subsection, let $G$ be a finite subgraph of $\Z^d$. The \emph{edge boundary} and \emph{outer edge boundary} of $G$ are respectively the following sets of edges:
\begin{align}
\pa G &:= \Big\{ e \in \rmE (\Z^d) : \text{ exactly one endpoint of $e$ lies in $G$} \Big\} \\
\pa_o G &:= \left\{e \in \pa G : \begin{matrix} \text{ the endpoint of $e$ in $G$ is connected to $\infty$} \\ \text{ via a path using no other vertices of $G$} \end{matrix} \right\} \,.
\end{align}
The \emph{vertex boundary} of $G$ is the following set of vertices:
\begin{align}
\pa_* G := \Big\{ v \in \rmV(G) : \text{$v$ is an endpoint of an edge in $\pa_oG$} \Big\} \,.
\end{align}

A \emph{cutset separating $G$ from $\infty$} is a finite collection of edges $S \subset \rmE(\Z^d)$ where any path from $G$ to $\infty$ uses an edge of $S$. If $A, B \subset \rmV(G)$ are disjoint vertex sets, a \emph{cutset separating $A$ and $B$ in $G$} is a finite collection of edges $S \subset \rmE(G)$ where any path in $G$ from $A$ to $B$ uses an edge of $S$. A cutset is \emph{minimal} if it is no longer a cutset upon removing an edge.  
 
We occasionally work with the graph $\mathbb{L}^d$, which has the same vertex set as $\Z^d$, but where vertices $x$ and $y$ are now adjacent if their $\ell^\infty$-distance is one. When $x$ and $y$ are adjacent in $\mathbb{L}^d$, write $x \sim_* y$ and say they are \emph{$*$-adjacent}. Paths in $\mathbb{L}^d$ are \emph{$*$-paths}, and $G \subset \Z^d$ is \emph{$*$-connected} if any two vertices of $G$ are joined by a $*$-path whose vertices all lie in $G$. Proposition \ref{star_conn} is standard in the literature and is useful for Peierls estimates appearing frequently in the study of lattice models.

\begin{prop} (Deuschel-Pisztora \cite{DP}, Tim\'{a}r \cite{Timar})\, Let $G \subset \Z^d$ be a finite, connected subgraph of $\Z^d$. Then $\pa_*G$ is $*$-connected, as is the set of vertices which are endpoints of edges in $\pa_o G$.
\label{star_conn}
\end{prop}

For $K \subset \R^d$ compact, write $G \cap K$ for the graph obtained by restricting $G$ to the vertex set $\rmV(G) \cap K$ in the natural way. Write $|G|$ for the cardinality of the vertex set of $G$, and in general if $F$ is any finite set, let $|F|$ denote the cardinality of $F$. 

Bond percolation gives rise to another notion of graph boundary: if $G$ is a finite subgraph of $\B{C}_\infty$, the \emph{open edge boundary} of $G$ is
\begin{align} 
\pa^\om G := \Big\{ e \in \pa G : \om(e) = 1 \Big\} \,, 
\end{align}
and the \emph{conductance} of $G$, written $\vp_G$, is the ratio $| \pa^\om G | / |G|$. A subgraph $G$ of $\giant = \B{C}_\infty \cap [-n,n]^d$ is \emph{valid} if it satisfies $0 < |G| \leq |\giant|/d!$, and a valid subgraph $G \subset \giant$ is \emph{optimal} if $\vp_G = \Chee$. 

\begin{rmk} Each $G_n \in \cal{G}_n$ is determined by its vertex set, else we could strictly reduce $|\pa^\om G_n|$.
\label{rem:barry_white}
\end{rmk}

\subsection{A metric on measures}\label{sec:notation_3}

To prove Theorem \ref{main_L1}, we first encode each optimizer $G_n$ as a measure and prove closeness to a set of limiting measures. Given $G_n \in \cal{G}_n$, the \emph{empirical measure} of $G_n$ is the following non-negative Borel measure on $[-1,1]^d$:
\begin{align}
\mu_n := \frac{1}{n^d} \sum_{ x\, \in\, \rmV(G_n)} \delta_{x/n} \,.
\label{eq:2.3_empirical} 
\end{align}
Given a Borel set $E \subset [-1,1]^d$, define $\nu_E$ as the measure on $[-1,1]^d$ having density $\theta_p(d) \1_E$ with respect to Lebesgue measure, and say that $\nu_E$ \emph{represents} $E$. The collection of finite signed Borel measures on $[-1,1]^d$ is denoted $\cal{M}([-1,1]^d)$, and the closed ball (with respect to the total variation norm) of radius $3^d$ about the zero measure in this space is written $\cal{B}_d$. For every percolation configuration $\om$, the empirical measures $\mu_n$ lie within $\cal{B}_d$, as does every $\nu_E$ for $E \subset [-1,1]^d$ Borel. 

We now equip $\cal{B}_d$ with a metric. For $k \in \{0,1,2,\dots\}$, the \emph{dyadic cubes at scale $k$} are all sets of the form $2^{-k} ( [-1,1]^d + x )$ for $x \in \Z^d$. Let $\Delta^k \equiv \Delta^{k,d}$ denote the dyadic cubes at scale $k$ contained in $[-1,1]^d$. Given $\mu, \nu \in \cal{B}_d$, define
\begin{align}
\frak{d}(\mu,\nu) := \sum_{k=0}^\infty \frac{1}{2^k} \sum_{Q\, \in\, \Delta^k} \frac{1}{| \Delta^k |} | \mu(Q) - \nu(Q) | \,.
\label{eq:2_metric}
\end{align}
The metric $\dw$ is useful for comparing discrete and continuous objects when both are encoded as measures. It figures prominently in the final section of the paper.

\subsection{Geometric measure theory, miscellaneous notation}\label{sec:gmt_misc} 

We introduce sets of finite perimeter; the following definitions are taken from Sections 13.3 and 14.1 of \cite{stflour}. Write $\cal{L}^d$ for $d$-dimensional Lebesgue measure and $\cal{H}^d$ for $d$-dimensional Hausdorff measure. Given a norm $\tau$ on $\R^d$ and a Borel subset $E$ of $\R^d$, define the \emph{surface energy of $E$ with respect to $\tau$} as 
\begin{align}
\cal{I}_\tau(E) : = \sup \left\{ \int_E \text{div} f(x)  \cal{L}^d(dx) : f \in C_c^\infty(\R^d, \wh{W}_\tau) \right\} \,,
\label{eq:new_surface_energy}
\end{align}
where $\wh{W}_\tau$ is the unit Wulff crystal defined in \eqref{eq:section1.1_shape}. Here $\text{div}(f)$ denotes the divergence of the function $f$, which is a smooth compactly supported function on $\R^d$ taking values in $\wh{W}_\tau$. By the divergence theorem, \eqref{eq:new_surface_energy} extends \eqref{eq:section1.1_functional} to Borel sets. When $\tau$ is the Euclidean norm, call $\cal{I}_\tau(E)$ the \emph{perimeter} of $E$, writing $\per(E)$. Naturally, $E$ has \emph{finite perimeter} if $\per(E) < \infty$. 

The following theorem is vital to the proof of Theorem~\ref{main_L1}. 

\begin{thm} (Taylor \cite{T2,T3,T1})\, Let $\tau$ be a norm on $\R^d$ and consider the variational problem for Borel sets $E \subset \R^d$:
\begin{align}
\text{minimize } \cal{I}_\tau(E) \hspace{5mm} \text{ subject to } \cal{L}^d(E) \leq \cal{L}^d(\wh{W}_\tau) \,,
\end{align}
A set $E$ is a minimizer of this variational problem if and only if there is $x \in \R^d$ such that the symmetric difference of $\wh{W}_\tau$ and $E + x$ has Lebesgue measure zero. 
\label{wulff_theorem}
\end{thm}

We close the section by collecting some notation. For $p \in [1,\infty]$, use $| \cdot |_p$ to denote the $\ell^p$-norm on $\R^d$. For $x \in \R^d$ and $r >0$, let $B(x,r)$ denote the closed $r$-ball centered at $x$ in the $\ell^2$-norm. For $E \subset \R^d$ and $a >0$, let $\cal{N}_a(E)$ denote the closed Euclidean $a$-neighborhood of $E$: $\cal{N}_{a}(E) := E + B(0,a)$,
and define the closed $a$-neighborhood of $E$ in the $\ell^1$-norm analogously, writing this as $\cal{N}_{a}^{(1)}(E)$. Lastly, the Hausdorff metric on compact subsets of $\R^d$ is defined via
\begin{align}
{\rm d}_H(A,B) := \max \left( \sup_{x \in A} \inf_{y \in B} \big| x- y\big|_\infty,\, \sup_{y \in B} \inf_{x \in A} \big| x-y \big|_\infty \right) \,.
\label{eq:2.5_Haus}
\end{align}


{\large\section{\B{The norm $\beta_{p,d}$ and the Wulff crystal}}\label{sec:norm}}

To motivate the construction of the norm, we regard an optimizer $G_n \in \cal{G}_n$ as a droplet in $\B{C}_\infty$, and look at a small but macroscopic (diameter on the order of $n$) box intersecting the boundary of $G_n$, as in Figure \ref{fig:droplet}.

\begin{figure}[h]
\centering
\includegraphics[scale=0.75]{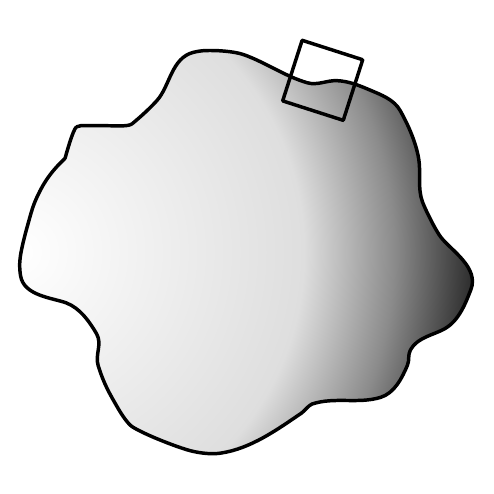}
\caption{A small box $B$ placed on the boundary of $G_n$.} 
\label{fig:droplet}
\end{figure}

The small box $B$ captures a piece of $\pa G_n$, which we imagine separates the top and bottom faces of $B$. The position of this cutset does not greatly affect the enclosed volume $|G_n|$ as $B$ is so small relative to $G_n$, so minimizing the number of open edges used by this cutset is most important to minimizing the conductance of $G_n$. This minimal number of open edges in a cutset separating the  faces of $B$ is a microscopic \emph{surface energy} in the direction normal to these faces. This energy grows like $O(n^{d-1})$ regardless of the normal direction. We then construct $\beta_{p,d}$ as a limit of these microscopic surface energies, properly normalized. 

Minimal randomly weighted cutsets in boxes ($d \geq 2$) well-studied. In $d =2$, such cutsets are dual to paths and fall within the realm of first-passage percolation. In higher dimensions, they were first examined by Kesten \cite{Kesten_Surfaces}. Variants of these objects have been studied by $\Th$ \cite{Theret}, Rossignol and $\Th$ \cite{RoTh_0,RoTh}, Zhang \cite{Zhang} and Garet \cite{Garet}. For a detailed list of these results, see Section 3.1 of \cite{CeTh}. 

Most of the work above constructs and uses the norm we are about to build. We emphasize that results presented in this section and in Section \ref{sec:concentration} are neither new nor optimal. Nevertheless, we find it important to present a relatively self-contained argument, and the notation introduced here will be used heavily throughout the paper. 

\subsection{Discrete cylinders}\label{sec:norm_1} We set up objects and notation needed to define $\beta_{p,d}$. We use notation from Cerf and $\Th$ \cite{stflour, CeTh} to build cylinders over $(d-1)$-dimensional objects; among other things these will function as boxes as in Figure~\ref{fig:droplet}.

Let $F \subset \R^d$ be the isometric image (see Remark \ref{rem:isometric}) of either a non-degenerate polytope in $\R^{d-1}$ or a Euclidean ball in $\R^{d-1}$. Polytopes are defined at the beginning of Section~\ref{sec:consequences_2}; in the present section we only ever need $F$ to be a square. 

Write $\hyp(F)$ to denote the hyperplane spanned by $F$, and let $v(F)$ denote one of the two unit vectors normal to $\hyp(F)$; the choice does not matter for our definitions. For $\rho > 0$, define $\cyl(F,\rho)$ to be the closed cylinder in $\R^d$ whose top and bottom faces are respectively $F_\rho^+ := F + \rho v(F)$ and $F_\rho^- := F - \rho v(F)$. The choice of $v(F)$ creates ambiguity over which face of the cylinder is the top, but this ambiguity is unimportant and plays no role in the definition of the norm. Define
\begin{align}
\slab(F,\rho) := \hyp(F) + B(0,\rho)\,.
\end{align}
Figure~\ref{fig:cyl_slab} depicts the geometric objects introduced so far. 

\begin{figure}[h]
\centering
\includegraphics[scale=1]{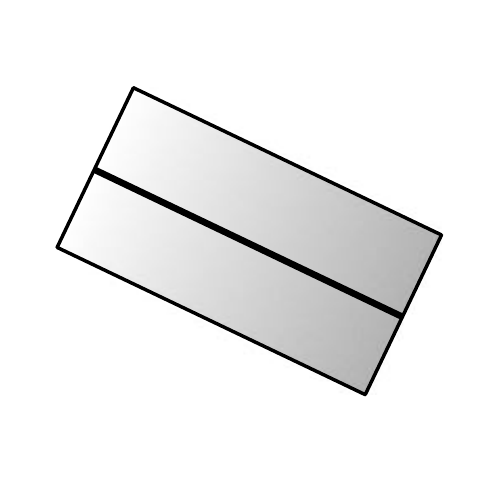} \hspace{5mm}
\includegraphics[scale=1]{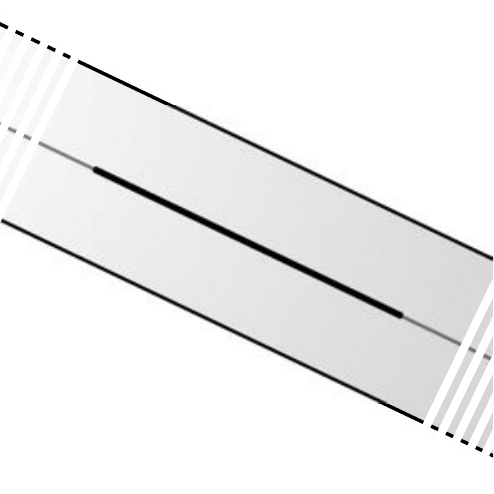} 
\caption{In both graphics, the bold line is $F$. The set $\cyl(F,\rho)$ is depicted as a box on the left. The top and bottom faces of this box are $F_\rho^+$ and $F_\rho^-$ respectively.  The set $\slab(F,\rho)$ is on the right, and the pale line running through the center is $\hyp(F)$.} 
\label{fig:cyl_slab}
\end{figure}

For $r > 0$, which we think of as large, define the \emph{discrete cylinder} $\dcyl(F,\rho,r)$ as
\begin{align}
\dcyl(F, \rho,r) := \Big\{ x \in \Z^d : x/r  \in  \cyl(F,\rho) \Big\} \,,
\end{align}
and note that $\cyl(F,\rho) \setminus \hyp(F)$ consists of two connected components. The top component, which contains $F_\rho^+$, is denoted $\cyl^+(F,\rho)$, while the bottom is denoted $\cyl^-(F,\rho)$. The following sets are the top (corresponding to ``$+$") and bottom (``$-$") \emph{hemispheres} of $\dcyl(F,\rho,r)$:
\begin{align}
\dhemi^\pm(F,\rho,r) := \Big\{ x \in \pa_* \dcyl(F,\rho, r) : x/r \in  \cyl^\pm (F,\rho) \Big\} \,.
\end{align}

\begin{figure}[h]
\centering
\includegraphics[scale=.7]{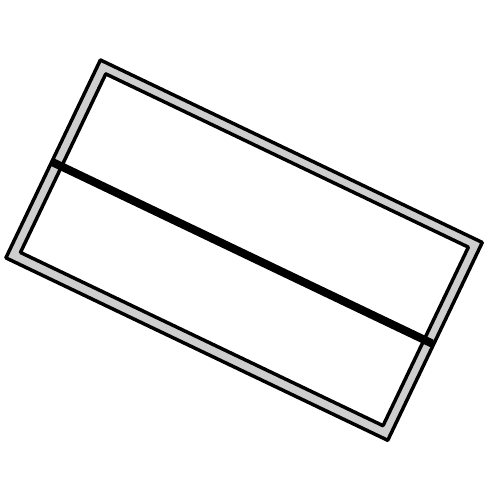} \hspace{20mm}
\includegraphics[scale=.7]{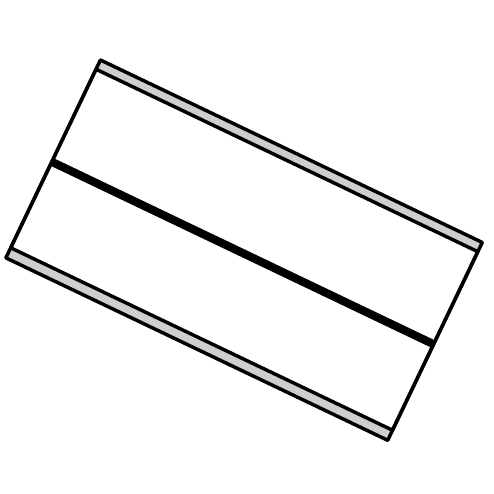} 

\caption{On the left, the vertex set $\dhemi^+(F,\rho,r)$ (respectively $\dhemi^-(F,\rho,r)$) is the shaded region above (respectively below) the bold line. On the right, the vertex sets $\dface^\pm(F,\rho,r)$ are the shaded regions above ($+$) and below ($-$) the bold line.} 
\label{fig:3_hemi}
\end{figure}

Define the top and bottom \emph{faces} of $\dcyl(F,\rho,r)$:
\begin{align} 
\dface^\pm(F,\rho,r) := \Big\{ x \in \pa_*  [ (r \cdot \slab(F,\rho)) \cap \Z^d ] : x / r \in  \cyl^\pm(F,\rho) \Big\} \,.
\end{align}
This definition of $\dface$ looks complicated, but is conceptually even simpler than $\dhemi$, it is depicted on the right side of Figure~\ref{fig:3_hemi}.

\subsection{Cutsets in discrete cylinders} The vertex sets $\dhemi^\pm (F,\rho,r)$ and $\dface^\pm(F,\rho,r)$ are contained in $\dcyl(F,\rho,r)$, which inherits a graph structure from $\Z^d$. We may then consider cutsets within $\dcyl(F,\rho,r)$ separating opposite hemispheres or faces. Bond percolation on $\Z^d$ induces bond percolation within $\dcyl(F,\rho,r)$, yielding a relevant \emph{weight} to assign to these cutsets. For any cutset $S$, let $|S|_\om$ denote the number of open edges in $S$, so that $|S|_\om$ is a random variable. Define
\begin{align}
\hemi(F,\rho,r) := \min\Big( |S|_\om : S \text{ separates } \dhemi^\pm(F,\rho,r) \text{ within } \dcyl(F,\rho,r) \Big) \,,
\label{eq:3_hemi_def}
\end{align}
and likewise define
\begin{align}
\face(F,\rho,r) := \min\Big( |S|_\om : S \text{ separates } \dface^\pm(F,\rho,r) \text{ within } \dcyl(F,\rho,r) \Big) \,.
\label{eq:3_face_def}
\end{align}
The cutsets in the definition of $\hemi(F,\rho,r)$ are \emph{anchored} at the equator of the cylinder $\cyl(F,\rho,r)$, whereas the cutsets in the definition of $\face(F,\rho,r)$ are allowed to meet the sides of $\cyl(F,\rho,r)$ at any height relative to the equator. 

\begin{rmk}
Whenever $r$ or $\rho$ are too small relative to $F$, $\hemi(F,\rho,r)$ and $\face(F,\rho,r)$ may not be well-defined. Say the parameters $r$ and $\rho$ are \emph{suitable} for $F$ if the vertex sets $\dhemi^\pm(F,\rho,r)$ and $\dface^\pm(F,\rho,r)$ are non-empty, and if the vertex sets $\dface^\pm(F,\rho,r)$ are a Euclidean distance of at least $100d$. When $\rho$ and $r$ are suitable for $F$, define $\hemi(F,\rho,r)$ and $\face(F,\rho,r)$ as in \eqref{eq:3_hemi_def} and \eqref{eq:3_face_def} respectively. Otherwise define these random variables to be zero.
\label{rmk:3_suitable}
\end{rmk}

To study cutsets within large discrete boxes, we specialize the above construction to cylinders based at squares. A \emph{square} in $\R^d$ is the isometric image of $[-1,1]^{d-1} \times \{0\}$. For $v \in \mathbb{S}^{d-1}$, consider a square in $\R^d$ centered at $0$ whose spanning hyperplane is normal to $v$. In dimensions at least three, this constraint does not uniquely determine the square, so we must assign each direction $v \in \mathbb{S}^{d-1}$ a unique square to define the norm. Let $\textsf{S}$ be such an assignment; that is for each $v \in \mathbb{S}^{d-1}$, $\textsf{S}(v)$ is a square in $\R^d$ centered at $0$ with $\hyp( \textsf{S}(v) )$ normal to $v$. Refer to $\textsf{S}$ as the \emph{chosen orientation}.

The value of $\beta_{p,d}$ in a given direction will not depend on $\textsf{S}$ (as we show in Proposition \ref{beta}). However, later proofs are simplified by building $\beta_{p,d}$ from a chosen orientation varying nicely over the sphere. Throughout this subsection and the next, treat $\textsf{S}$ as given. The random variables used to define $\beta_{p,d}$ are
\begin{align} 
\frak{X}(x,v,r) := \hemi(\textsf{S}(v) + x, 1, r) \,.
\end{align}
The final observation of this subsection is that the expected value of these random variables is not too sensitive to where the cylinder is centered. 

\begin{lem} Let $d\geq 2$. There is a positive constant $c(d)$ so that for all $p \in [0,1]$, $x \in \R^d$, $v \in \mathbb{S}^{d-1}$ and $r >0$, 
\begin{align}
\E_p \frak{X} \big(x, v,r + d^{1/2}\big) \leq \E_p \frak{X}(0,v, r) + c(d) r^{d-2} \,.
\end{align}
\label{center}
\end{lem}

\begin{proof} The idea of this proof is captured in Figure \ref{fig:integer_center}. It is intuitive that by starting with an anchored cutset for the smaller box and adding a microscopic ring of edges around the equator of this box, we produce an anchored cutset in a slightly larger box. This is indeed the case, though we are careful to show in the next lemma that patching these edge sets together truly is a cutset in the larger box. 

 Let $x \in \R^d$ and $v \in \mathbb{S}^{d-1}$. Choose $x' \in \Z^d$ so that $|x -x' |_\infty \leq 1$. For notational ease, make the following abbreviations within this proof and the next.
\begin{align}
\cyl(x') := r \cyl ( \textsf{S}(v) + x' , 1) \hspace{10mm} \cyl(x) :=  \left(r + d^{1/2} \right) \cyl ( \textsf{S}(v) + x , 1)\, 
\end{align}
\begin{align}
\hyp(x) := \hyp \left(   \left(r + d^{1/2} \right) ( \textsf{S}(v) + x ) \right) \,.
\end{align}
Thus, $\cyl(x')$ is the slightly larger box with integer center containing $\cyl(x')$. Let $A$ be the microscopic ring of edges, formally the collection of edges in $\Z^d$ having non-empty intersection with the neighborhood
\begin{align}
\cal{N}_{5d} \left( \left( \cyl(x) \setminus \cyl(x') \right) \cap \hyp(x) \right) \,.
\label{eq:3_neighborhood}
\end{align}
\begin{figure}[h]
\centering
\includegraphics[scale=1]{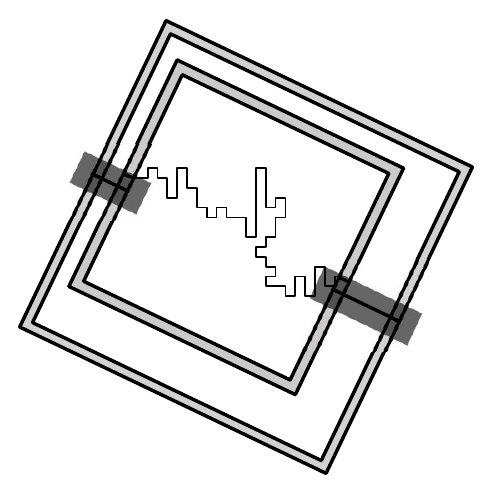} 
\caption{The inner box is $\cyl(x')$, the outer box is $\cyl(x)$, and the darker shaded region is the neighborhood \eqref{eq:3_neighborhood} used to define $A$. The thin discrete interface is the cutset $E$.} 
\label{fig:integer_center}
\end{figure}

From the construction of $A$, there is $c(d) >0$ with $|A| \leq c(d) r^{d-2}$. This is because \eqref{eq:3_neighborhood} is a microscopic thickening of a $(d-2)$-dimensional set. We now choose a cutset in the smaller cylinder. For notational clarity, make the following abbreviations.
\begin{align}
\dcyl(x') &:=  \dcyl(\textsf{S}(v) + x', 1, r) \,,\\
\dhemi^\pm(x') &:= \dhemi^\pm(\textsf{S}(v) + x',1,r) \,,  \\
\dcyl(x) &:=  \dcyl(\textsf{S}(v) +x, 1, r + d^{1/2}) \,, \\
\dhemi^\pm(x) &:= \dhemi^\pm(\textsf{S}(v) + x,1,r + d^{1/2}) \,.
\end{align}

Let $E = E(\omega)$ be a minimal cutset separating $\dhemi^\pm(x')$ within $\dcyl(x')$. We claim the edges in $A \cup E $ which lie in $\dcyl(x)$ separate $\dhemi^\pm(x)$ in $\dcyl(x)$. Assuming this, we have 
\begin{align}
\frak{X}(x,v,r + d^{1/2} ) \leq \frak{X}(x',v,r) + c(d) r^{d-2} \,,
\end{align}
and the lemma is proved upon taking expectations as $x' \in \Z^d$. \end{proof}

We now carefully show that the patching of edge sets produces a cutset in the larger box. We appeal to the following argument at several points in the future without repeating details. 

\begin{lem} In the proof of the preceding lemma, the edges of $A \cup E$ contained in $\dcyl(x)$ separate $\dhemi^\pm(x)$ in $\dcyl(x)$. 
\label{lem:cutset_reference}
\end{lem}

\begin{proof} Adopt the notation from the proof of Lemma \ref{center}. It suffices to show that any $\Z^d$-path joining $\dhemi^\pm(x)$ within $\dcyl(x)$ uses an edge of $A \cup E(\om)$. Let $y^\pm \in \dhemi^\pm(x)$, and let $\gamma$ be a simple path from $y^-$ to $y^+$ using only edges of $\dcyl(x)$. If $\gamma$ does not pass through a vertex of $\pa_* \dcyl(x')$, $\gamma$ lies entirely within $\cyl(x) \setminus \cyl(x')$, in which case $\gamma$ must use an edge of $A$. We may then suppose that $\gamma$ passes through a vertex of $\pa_* \dcyl(x')$ and consider several cases.

\emph{Case (i):} Suppose that the last vertex $z^+$ of $\dcyl(x')$ used by $\gamma$ lies within $\dhemi^-(x')$. Let $\gamma'$ denote the subpath of $\gamma$ connecting $z^+$ to $y^+$, and observe that $\gamma'$ is contained within $\cyl(x) \setminus \cyl(x')$.  As $\gamma'$ starts either in the bottom half of $\cyl(x)$ or in the neighborhood defined in \eqref{eq:3_neighborhood}, $\gamma'$ must use an edge in $A$. 

\emph{Case (ii):} Suppose that the first vertex $z^-$ of $\dcyl(x')$ used by $\gamma$ lies in $\dhemi^+(x')$. Using the same reasoning as in Case (i), we see that $\gamma$ must use an edge in $A$ between $y^-$ and $z^-$. 

\emph{Case (iii):} We may now suppose $z^\pm \in \dhemi^\pm(x')$. Let $z$ be the vertex of $\dhemi^-(x')$ used last by $\gamma$, and consider the subpath $\gamma'$ of $\gamma$ joining $z$ to $z^+$. If $\gamma'$ is contained completely within $\dcyl(x')$, then $\gamma'$ uses an edge of $E(\om)$. On the other hand, if $\gamma'$ is not contained in $\dcyl(x')$, we may assume the vertex following $z$ in the path $\gamma'$ lies outside of $\dcyl(x')$, else $\gamma'$ would either use an edge of $E(\om)$, or would not use $z$ last among all vertices of $\dhemi^-(x')$. Under this assumption, $\gamma'$ leaves $\dcyl(x')$ at the vertex $z$, and only returns to $\dcyl(x')$ at some vertex $z' \in \dhemi^+(x')$. Along the subpath $\gamma''$ of $\gamma'$ joining $z$ with $z'$, all intermediate vertices lie in $\cyl(x) \setminus \cyl(x')$, and $\gamma''$ uses an edge of $A$. 
\end{proof}

\subsection{Defining the norm}

To define $\beta_{p,d}$ as quickly as possible, we use a subadditivity argument taken from Rossignol and $\Th$ (see Section 4.3 of \cite{RoTh}).

\begin{prop} Let $d \geq 2$. For all $v \in \mathbb{S}^{d-1}$, the limit
\begin{align}
\beta_{p,d}(v) : = \lim_{n \to \infty} \frac{ \E_p \frak{X}(0,v,n) }{ (2 n)^{d-1} } 
\end{align}
exists and is finite. Moreover, this limit is independent of the chosen orientation ${\rm \sfS}$. 
\label{beta}
\end{prop}

\begin{proof} Let $n,m \in \N$ with $n$ much larger than $m$ and both numbers larger than $d$. Write $n = km +r$ for $k,r \in \N \cup\{0\}$ and $r < m$. Let $\sfS$ be the chosen orientation, and let $\wt{\sfS}$ be another assignment of unit vectors $v \in \mathbb{S}^{d-1}$ to squares $\wt{\sfS}(v)$ so that $v$ is normal to $\hyp( \wt{\sf{S}}(v))$. Define $\wt{\frak{X}}(x,v,r)$ using $\wt{\sfS}$ in place of $\sfS$:
\begin{align}
\wt{\frak{X}}(x,v,r) := \hemi\left(\wt{\sfS}(v) + x,1,r\right) \,.
\end{align}
Choose a finite collection $\{ \wt{S}_i\}_{i=1}^\ell$ of translates of $(m+d^{1/2}) \wt{\sfS}(v)$, each contained in $n\sfS(v)$, so that:

\emph{(i)} The translates $\{ \wt{S}_i \}_{i=1}^\ell$ are disjoint.

\emph{(ii)} There is a positive constant $c(d)$ so that $\cal{H}^{d-1}\left( n \sfS(v) \setminus \bigcup_{i=1}^\ell \wt{S}_i  \right) \leq c(d) m n^{d-2}$. 

\emph{(iii)} $\ell \leq (k+1)^{d-1}$. 

\begin{figure}[h]
\centering
\includegraphics[scale=1]{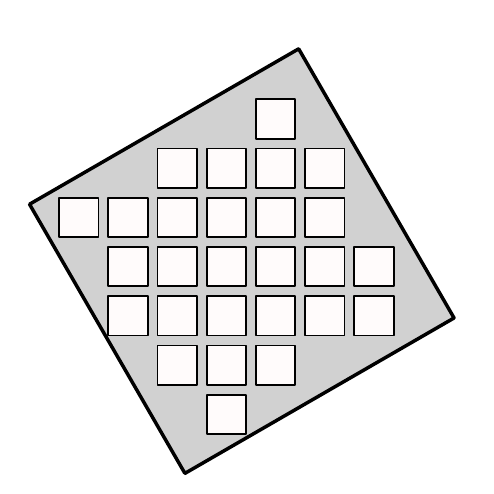} 
\caption{The small white squares are the collection $\{ \wt{S}_i \}_{i=1}^\ell$, which are disjoint and nearly exhaust the large square $n \sfS(v)$. In this diagram, we draw squares as two-dimensional objects, whereas in all previous diagrams they were drawn as one-dimensional objects.} 
\label{fig:square_tile}
\end{figure}

Make the abbreviations
\begin{align}
\dcyl(i) &:= \dcyl (\wt{S}_i, m+d^{1/2},1) \,,\\
\dhemi^\pm(i) &:= \dhemi^\pm (\wt{S}_i, m+d^{1/2},1) \,, \\
\dcyl &:= \dcyl( \sfS(v) , 1, n) \,,\\
\dhemi^\pm &:= \dhemi^\pm ( \sfS(v), 1, n) \,.
\end{align}

For each $\wt{S}_i$, let $E_i$ be a minimal cutset in separating $\dhemi^\pm (i)$ within $\dcyl(i)$. Let $A$ be the collection of edges in $\Z^d$ having non-empty intersection with
\begin{align}
\cal{N}_{5d}  \left( n \sfS(v) \setminus \bigcup_{i=1}^\ell \wt{S}_i \right) \,.
\end{align}
By \emph{(ii)} above, there is $c(d)>0$ so that $|A| \leq c(d) m n^{d-2}$. We soon take $n$ to infinity, thus we lose no generality supposing $n$ is large enough so that each $\dcyl(i)$ is contained in $\dcyl$, and in particular, that each $E_i$ is contained in the edge set of $\dcyl$ across all configurations $\om$. 

The argument of Lemma~\ref{lem:cutset_reference} shows that the edges $A \cup \left(\bigcup_{i=1}^\ell E_i \right)$ lying in $\dcyl$ separate $\dhemi^\pm$ in $\dcyl$. Though there are more boxes in this case, the complexity of the argument does not go up: we can always reduce to the case that our simple path $\gamma$ last uses any vertex of $\dhemi^-(i)$ for all $i$, and we may also assume $\gamma$ uses a vertex within some $\dhemi^+(j)$ at a later point. Between these two points, we find that we must either use an edge in $A$, or an edge in one of the $E_i$. Thus, 
\begin{align}
\frak{X}(0,v,n) \leq \sum_{i=1}^\ell \hemi (\wt{S}_i , m+d^{1/2},1 ) + c(d) mn^{d-2} \,.
\label{eq:3.24}
\end{align}
The chosen orientation thus far has been arbitrary, so the preceding lemma also applies to $\wt{\frak{X}}(0,v,n)$. Take expectations of both sides in \eqref{eq:3.24} and apply Lemma \ref{center} to each term in the sum of \eqref{eq:3.24}, using the bound $\ell \leq (k+1)^{d-1}$ from \emph{(iii)}.
\begin{align}
\E_p \frak{X}(0,v,n) &\leq \ell  \E_p \wt{\frak{X}}(0,v,m) + \ell c(d) m^{d-2} + c(d) mn^{d-2} \,, \\
&\leq (k+1)^{d-1}  \E_p \wt{\frak{X}}(0,v,m) + (k+1)^{d-1}c(d) m^{d-2} + c(d) mn^{d-2} \,.
\end{align}
Divide through by $n^{d-1}$:
\begin{align}
\frac{\E_p \frak{X}(0,v,n)}{n^{d-1}} &\leq  (k+1)^{d-1}  \frac{\E_p \wt{\frak{X}}(0,v,m)}{n^{d-1}} + \frac{(k+1)^{d-1}m^{d-2}c(d)}{ n^{d-1}} + \frac{c(d) m}{n} \,,\\
&\leq  \left(\frac{k+1}{k}\right)^{d-1}  k^{d-1} \cdot \frac{\E_p \wt{\frak{X}}(0,v,m)}{n^{d-1}} +\left(\frac{k+1}{k}\right)^{d-1} \left(\frac{k}{n}\right)^{d-1} m^{d-2} c(d)  + \frac{c(d) m}{n}\,, \\
&\leq  \left(\frac{k+1}{k}\right)^{d-1}   \frac{\E_p \wt{\frak{X}}(0,v,m)}{m^{d-1}} +\left(\frac{k+1}{k}\right)^{d-1}  \frac{c(d)}{m}  + \frac{c(d) m}{n} \,.
\end{align}
First take the $\limsup$ of both sides in $n$, 
\begin{align}
\limsup_{n \to \infty} \frac{\E_p \frak{X}(0,v,n)}{n^{d-1}} \leq  \frac{\E_p \wt{\frak{X}}(0,v,m)}{m^{d-1}} + \frac{c(d)}{m}\,, 
\end{align}
and then the $\liminf$ of both sides in $m$:
\begin{align}
\limsup_{n \to \infty} \frac{\E_p \frak{X}(0,v,n)}{n^{d-1}} \leq \liminf_{m \to \infty}  \frac{\E_p \wt{\frak{X}}(0,v,m)}{m^{d-1}}  \,,
\end{align}
and the proof is complete upon dividing both sides by $2^{d-1}$: setting $\wt{\sfS} \equiv \sfS$ gives us the existence of the limit in question, and interchanging $\wt{\sfS}$ and $\sfS$ in the above argument tells us this limit does not depend on the chosen orientation. The finiteness of this limit can be seen as follows: given a direction $v \in \mathbb{S}^{d-1}$, the collection of edges intersecting the neighborhood $\cal{N}_{5d} (n\sfS(v) )$ forms a cutset in $\dcyl(\sfS(v),1,n)$ separating $\dhemi^\pm(\sfS(v),1,n)$ and this cutset has cardinality bounded above by $c(d) n^{d-1}$ for some positive constant $c(d)$ not depending on the direction. \end{proof}

We immediately deduce that $\beta_{p,d}$ inherits the symmetries of $\Z^d$.

\begin{coro} Let $d\geq 2$. For all $v \in \mathbb{S}^{d-1}$ and for all linear transformations $L : \R^d \to \R^d$ such that $L(\Z^d) = \Z^d$, we have $\beta_{p,d}( Lv) = \beta_{p,d}(v)$.
\label{sym}
\end{coro}

\begin{proof} Let $v \in \mathbb{S}^{d-1}$, and let $\sfS$ be the chosen orientation. Then 
\begin{align}
\wt{\sfS}(v) := L^{-1} \sfS(Lv)
\end{align}
 is a rotation of $\sfS(v)$ contained in $\hyp (\sfS(v))$. From the preceding Proposition \ref{beta}, we know
 \begin{align}
 \lim_{n\to\infty} \frac{ \E_p \hemi( \wt{\sfS}(v),1,n)}{(2n)^{d-1}} = \lim_{n\to \infty} \frac{\E_p \frak{X}(0,v,n)} { (2n)^{d-1} } \,.
\end{align}
Moreover, because $L$ induces a graph automorphism of $\Z^d$, we know $\E_p \hemi( \wt{\sfS}(v), 1, n) = \E_p \frak{X}(0, Lv,n)$, so that 
\begin{align}
\lim_{n \to \infty} \frac{ \E_p \frak{X}(0, Lv,n)}{(2n)^{d-1}} = \lim_{n\to\infty} \frac{  \E_p \frak{X}(0, v,n)}{(2n)^{d-1}} \,,
\end{align}
as desired. \end{proof}

\subsection{The chosen orientation and properties of $\beta_{p,d}$}\label{sec:norm_3} Defining $\beta_{p,d}$ using cylinders based at squares (instead of discs, for instance) allows us to execute subadditivity arguments with ease. There is a tradeoff between the tidiness of these arguments and the artificial nature of the chosen orientation; we feel we have taken the route which is ultimately cleanest. Part of this tradeoff is that $\sfS$ must vary over most of the sphere in a Lipschitz way. Given an $\sfS$ and $A \subset \mathbb{S}^{d-1}$, say $\sfS$ is \emph{nicely varying over $A$} if there is $M(d) >0$ so that $\sfS$ satisfies 
\begin{align}
{\rm d}_H( \sfS(v) , \sfS(w) ) \leq M \e 
\label{eq:3_chosen_lip}
\end{align} 
whenever $|v - w|_2 < \e$ and $v,w \in A$. Here ${\rm d}_H$ is defined in \eqref{eq:2.5_Haus}. For $\sfS$ nicely varying on $\mathbb{S}^{d-1}$, one can show the functions $v \mapsto \E_p \frak{X}(0,v,n) / (2n)^{d-1} $ converge uniformly to $\beta_{p,d}$, allowing us to prove concentration estimates in Section~\ref{sec:concentration}. 

For topological reasons (the hairy ball theorem) it is not in general possible to have $\sfS$ vary nicely over the entire sphere. However, it will suffice to work with $\sfS$ varying nicely over the upper and lower hemispheres of $\mathbb{S}^{d-1}$. Introduce the closed upper hemisphere $\mathbb{S}_+^{d-1} := \mathbb{S}^{d-1} \cap \{ x \in \R^d : x_d \geq 0 \}$. A corollary of Proposition \ref{chosen} is that we may first define $\sfS$ over $\mathbb{S}_+^{d-1}$ so that $\sfS$ is nicely varying over $\mathbb{S}_+^{d-1}$. With such $\sfS$ defined on the upper hemisphere, extend the definition of $\sfS$ to the rest of $\mathbb{S}^{d-1}$ in a natural way by reflection. This yields $\sfS$ which varies nicely over $\mathbb{S}_+^{d-1}$ and $\mathbb{S}^{d-1} \setminus \mathbb{S}_+^{d-1}$. Henceforth we suppose $\sfS$ has these two properties. 

\begin{prop} Let $d \geq 2$, $p > p_c(d)$ and suppose $\sfS$ varies nicely over $\mathbb{S}_+^{d-1}$ and $\mathbb{S}^{d-1} \setminus \mathbb{S}_+^{d-1}$. Then the functions $v \mapsto \E_p \frak{X}(0,v,n) / (2n)^{d-1}$ converge uniformly to $\beta_{p,d}$.
\label{uniform}
\end{prop}

\begin{proof} Let $\e > 0$, Let $v,w \in \mathbb{S}_+^{d-1} $ be such that $|v-w|_2 < \e$. Let us fix some notation:
\begin{align}
\cyl(v) &:= n \cyl( \sfS(v), 1) \,, \\
\dcyl(v) &:= \dcyl(\sfS(v),1,n) \,,\\
\dhemi^\pm(v) &:= \dhemi^\pm( \sfS(v), 1, n) \,, \\
\cyl(w) &:= \lceil n(1+M\e) \rceil \cyl ( \sfS(w), 1)\,, \\
\dcyl(w) &:= \dcyl(\sfS(w),1,\lceil n(1+M\e) \rceil) \,, \\
\dhemi^\pm(w) &:= \dhemi^\pm( \sfS(w), 1, \lceil n(1+M\e) \rceil ) \,.
\end{align}

Let $E$ be a minimal cutset separating $\dhemi^\pm (v)$ in $\dcyl(v)$. By the hypothesis on $\sfS$, $\cyl(v) \subset \cyl(w)$, and $E$ is contained in the edge set of $\dcyl(w)$. As before, we use $E$ in conjunction with a small collection of edges to produce a cut separating the hemispheres of $\dcyl(w)$. We actually use two other collections of edges to do this. 

Writing $\hyp(w)$ for $\hyp(\lceil n(1+M\e) \rceil \sfS(w) )$, we define the edge set $A$ as in Lemma \ref{center} to be the edges of $\Z^d$ intersecting the neighborhood
\begin{align}
\cal{N}_{5d} \left(  (\cyl (w) \setminus \cyl(v) ) \cap \hyp(w)\right)\,.
\label{eq:uniform_2}
\end{align}
Likewise, let $B$ be the collection of edges having non-empty intersection with 
\begin{align}
\cal{N}_{5d} \left( \pa \cyl(v) \cap \slab ( \sfS(v), nM\e) \right) \,.
\label{eq:uniform_1}
\end{align}
Here we suppose that $\e$ is small enough so that $\slab(\sfS(v), nM\e)$ does not contain the top and bottom faces of the cube $\cyl(v)$. The neighborhood (\ref{eq:uniform_1}) is thus slight thickening of an equatorial band of height $nM\e$ in $\pa \cyl(v) $, and it follows that $|B| \leq c(d) M\e n^{d-1}$ for some $c(d) > 0$.

The neighborhood (\ref{eq:uniform_2}) forms a bridge between the neighborhood (\ref{eq:uniform_1}) defining $B$ and the equator of the larger cube $\cyl(w)$. By construction, we also have $|A| \leq c(d) M\e n^{d-1}$ for some $c(d) > 0$. Figure~\ref{fig:uni} illustrates the cutset $E$ with the edge sets $A$ and $B$.

\begin{figure}[h]
\centering
\includegraphics[scale=1]{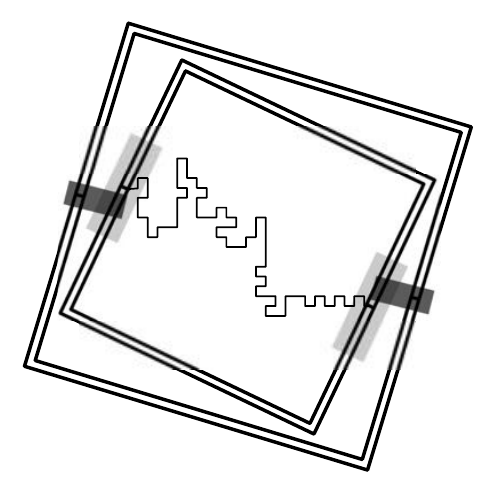}
\caption{The cutset $E$ in the smaller cube $\cyl(v)$ is central. At the equator of $\cyl(v)$, this cutset meets with the edge set $B$, the lightly shaded regions. $B$ is joined to the equator of the larger cube $\cyl(w)$ by $A$, the darker shaded regions.} 
\label{fig:uni}
\end{figure}

The edges of the union $E \cup A \cup B$ contained in $\dcyl( w)$ form a cutset separating the hemispheres $\dhemi^\pm(w)$. The argument for this is nearly identical to the proof of Lemma \ref{lem:cutset_reference}. Indeed, we are looking at nested cubes, with the only difference that one is tilted slightly relative to the other. This tilt is why $B$ is introduced. Thus,
\begin{align}
\frak{X}(0,w, \lceil n(1+M\e) \rceil) \leq \frak{X}(0,v,n) + c(d) M \e n^{d-1} \,,
\end{align}
so that by taking expectations, 
\begin{align}
\frac{\E_p \frak{X}(0,w, \lceil n(1+M\e) \rceil)}{(2\lceil n(1+M\e) \rceil)^{d-1}} \leq \frac{\E_p\frak{X}(0,v,n)}{(2n)^{d-1}} + c(d) M \e \,.
\label{eq:uni}
\end{align}
Taking $n \to \infty$, we have shown when $v,w \in \mathbb{S}_+^{d-1}$ satisfy $|v-w|_2 < \e$, 
\begin{align}
| \beta_{p,d}(v) - \beta_{p,d}(w) | < c(d) M\e \,.
\label{eq:3_close}
\end{align}
A symmetric argument shows the same bounds hold when $v,w \in \mathbb{S}^{d-1} \setminus \mathbb{S}_+^{d-1}$ and $|v-w|_2 < \e$.

Choose a finite collection of unit vectors $\{ v_i \}_{i=1}^m$ (with $m = m(\e)$), so that for any $v \in \mathbb{S}_+^{d-1}$, there is $v_i \in \mathbb{S}_+^{d-1}$ with $|v - v_i|_2 < \e$, and if $v \in \mathbb{S}^{d-1} \setminus \mathbb{S}_+^{d-1}$, there is $v_i \in \mathbb{S}^{d-1} \setminus \mathbb{S}_+^{d-1}$ with $|v-v_i|_2 < \e$. Take $N$ large enough so that whenever $n \geq N$, for each $i$,
\begin{align}
\left| \frac{ \E_p \frak{X}(0,v_i,n) }{ (2n)^{d-1} } - \beta_{p,d}(v_i) \right| < \e \,.
\end{align}
Let $v \in \mathbb{S}^{d-1} $ and take $v_i$ in the same hemisphere so that $|v - v_i |_2 < \e$. Apply (\ref{eq:uni}) twice:
\begin{align}
\frac{\E_p \frak{X}(0,v_i, \left\lceil \lceil n(1+M\e) \rceil(1+M\e) \right\rceil )}{(2 \left\lceil \lceil n(1+M\e) \rceil (1+M\e)\right\rceil)^{d-1}} -c(d) M\e &\leq \frac{\E_p \frak{X}(0,v, \lceil n(1+M\e) \rceil)}{(2\lceil n(1+M\e) \rceil)^{d-1}}\\
&\leq \frac{\E_p\frak{X}(0,v_i,n)}{(2n)^{d-1}} + c(d) M \e  \,.
\end{align}
By (\ref{eq:3_close}), we have
\begin{align} \beta_p(v) - \e - 2c(d)M\e \leq  \frac{\E_p \frak{X}(0,v, \lceil n(1+M\e) \rceil)}{(2\lceil n(1+M\e) \rceil)^{d-1}} \leq \beta_p(v) + \e + 2c(d) M \e \,,
\end{align}
which establishes the desired uniform convergence. \end{proof}

Extend $\beta_{p,d}$ to a function on all of $\R^d$ via homogeneity; for $x \in \R^d$ define
\begin{align}
\beta_{p,d} (x) := \begin{cases}  |x|_2\,\beta_{p,d} ( x / |x|_2) & |x|_2 > 0 \\ 0 & |x|_2 = 0 \end{cases} \,.
\end{align}

\begin{prop} For $d\geq2$ and $p > p_c(d)$, the function $\beta_{p,d} : \R^d \to [0,\infty)$ defines a norm on $\R^d$.
\label{norm}
\end{prop}

\begin{proof} The proof of Proposition 11.6 in \cite{stflour} (or Proposition 4.5 of \cite{RoTh}) tells us that $\beta_{p,d}$ satisfies the weak triangle inequality. By Corollary 11.7 of \cite{stflour}, $\beta_{p,d}$ is a convex function on $\R^d$. To show non-degeneracy of $\beta_{p,d}$, it suffices to show non-degeneracy in the cardinal directions. By Corollary~\ref{sym}, it suffices to show non-degeneracy in a single cardinal direction. 

This non-degeneracy is a consequence of Theorem 7.68 in \cite{Grimmett}, for instance: within a large axis-parallel cube, with high probability, there are at least $c n^{d-1}$ edge-disjoint open paths between the top and bottom faces, for some $c(p,d) >0$. Menger's theorem converts this fact into a high probability lower bound on the size of a minimal cut separating opposing faces of this cube.  \end{proof}

\begin{rmk} That $\beta_{p,d}$ is a norm allows us to define the associated surface energy $\cal{I}_{p,d}$, as in Section \ref{sec:notation}, as well as the unit Wulff crystal $\wh{W}_{p,d}$, which is the unit ball in the norm dual to $\beta_{p,d}$.  Define the \emph{Wulff crystal} $W_{p,d}$ to be the dilate of $\wh{W}_{p,d}$ about the origin so that $\cal{L}^d ( W_{p,d}) = 2^d / d!$. The Wulff crystal $W_{p,d}$ is the limit shape appearing in Theorem \ref{main_L1}. So that this theorem makes sense, we must know that $W_{p,d}$ is contained in $[-1,1]^d$.
\end{rmk}

\begin{lem} For $d\geq2$ and $p > p_c(d)$, the Wulff crystal $W_{p,d}$ is contained in $[-1,1]^d$. 

\label{containment}
\end{lem}

\begin{proof} By Corollary \ref{sym}, the unit Wulff crystal $\wh{W}_{p,d}$ satisfies
\begin{align}
c B_1 \subset \wh{W}_{p,d} \subset c B_\infty
\end{align}
for some $c >0$, where $B_1$ and $B_\infty$ respectively denote unit $\ell^1$- and unit $\ell^\infty$-balls in $\R^d$ centered at the origin. The claim follows from the fact that $\cal{L}^d(B_1) = 2^d / d!$. \end{proof}

\begin{rmk} We may use $\cal{I}_{p,d}$ to define an analogous notion of conductance in the continuum: for $E \subset \R^d$ a set of finite perimeter, we define the \emph{conductance} of $E$ as $\cal{I}_{p,d}(E) / \theta_p(d) \cal{L}^d(E)$. 
\end{rmk}




{\large\section{\B{Concentration estimates for $\beta_{p,d}$}}\label{sec:concentration}}

We now derive concentration estimates for the random variables used to define $\beta_{p,d}$, following an argument of Zhang in Section 9 of \cite{Zhang}. We use results from his paper in conjunction with the following concentration estimate due to Talagrand. \newline

\begin{thm} (Talagrand \cite{Tala}, Section 8.3)\, Let $(\rm V,E)$ be a finite graph with $\{X_e\}_{e \in{\rm E}}$ a collection of iid Bernoulli($p$) random variables. Let $\cal{S}$ denote a family of sets of edges and for $S \in \cal{S}$, let $X_S := \sum_{ e \in S} X_e$. Let $Z_{\cal{S}} := \inf_{S \in \cal{S}} X_S$, and let $M_{\cal{S}}$ be a median of $Z_{\cal{S}}$. There is $c(p) > 0$ so that for all $u > 0$, 
\begin{align}
\prob_p ( | Z_{\cal{S}} - M_{\cal{S}} | \geq u ) \leq 4 \exp \left( - c\min \left( \frac{u^2}{\al}, u \right) \right) \,,
\end{align}
where $\al = \sup_{S \in \cal{S}} |S|$.
\label{tala}
\end{thm}

\begin{rmk} The random variables $\hemi$ and $\face$ are easily expressed as $Z_\cal{S}$ for some family of edge sets $\cal{S}$, but use of Theorem \ref{tala} requires control over the size of the largest edge set in $\cal{S}$ through the term $\al$. We must then control the size of the largest minimal cut separating opposing hemispheres (or faces) of a~cube. 
\end{rmk}

\begin{rmk} The chosen orientation \sfS, introduced in the previous section, has been fixed since the beginning of Section \ref{sec:norm_3}. Following Zhang in \cite{Zhang}, we use Theorem \ref{tala} to prove concentration for a variant of the $\frak{X}(0,v,n)$.
\end{rmk}

 Let $\gamma >0$, and let $\cal{S}_{n,v}(\gamma)$ be the family of cutsets in $\dcyl(\sfS(v),1,n)$ satisfying $|S| \leq \gamma (2n)^{d-1}$, and which separate $\dhemi^\pm(\sfS(v),1,n)$. Define $
Z_{n,v}^{(\gamma)}(\om) := \inf_{S \in \cal{S}_n(\gamma)} |S|_\om$, and apply Theorem \ref{tala} to $Z_{n,v}^{(\gamma)}$, using the bound $\al \leq \gamma (2n)^{d-1}$. 

\begin{prop} Let $\e, \gamma > 0$. There are $c_1(p,\gamma,\e),c_2(p,\gamma,\e) >0$ so that for all $v \in \mathbb{S}^{d-1}$ and $n \geq 1$ 
\begin{align}
\prob_p \left( \frac{ |Z_{n,v}^{(\gamma)} - \E_p Z_{n,v}^{(\gamma)} |}{(2n)^{d-1} } \geq \e  \right) \leq c_1 \exp \left( -c_2 n^{(d-1)/3} \right) \,.
\end{align}

\label{trunc_con}
\end{prop}

\begin{proof} We follow the argument at the beginning of Section 9 in \cite{Zhang}.
Write $A = A(n) := (2n)^{d-1}$ for the $\cal{H}^{d-1}$-measure (or ``area") of the square $n\sfS(v)$. Let $M_{n,v}^{(\gamma)}$ be a median of $Z_{n,v}^{(\gamma)}$. Then, 
\begin{align}
| \E_p Z_{n,v}^{(\gamma)} - M_{n,v}^{(\gamma)} | &\leq \E_p | Z_{n,v}^{(\gamma)} - M_{n,v}^{(\gamma)} |  \,, \\
& \leq \sum_{j=1}^{\lfloor A^{2/3}\rfloor } \prob_p( |Z_{n,v}^{(\gamma)} - M_{n,v}^{(\gamma)}| \geq j) + \sum_{j= \lceil A^{2/3} \rceil }^\infty \prob_p( |Z_{n,v}^{(\gamma)} - M_{n,v}^{(\gamma)}| \geq j ) \,.
\end{align}
Apply Theorem \ref{tala} with $\al \leq \gamma A$ to the right-most sum above:
\begin{align} 
| \E_p Z_{n,v}^{(\gamma)} - M_{n,v}^{(\gamma)}| &\leq A^{2/3} + 4 \left( \sum_{j= \lceil A^{4/3} \rceil }^\infty \exp\left(-c \frac{j}{\gamma A} \right) +  \sum_{j= \lceil A^{2/3} \rceil }^\infty  \exp(-cj) \right) \,, \\
&\leq A^{2/3} + \frac{4}{1 - \exp( - c / \gamma A)} \exp( -c A^{1/3} / \gamma) + \frac{4}{1- \exp(-c)} \exp ( -c A^{2/3} ) \,.
\end{align}
For $n$ large depending on $p$ and $\gamma$, $|\E_p Z_{n,v}^{(\gamma)} - M_{n,v}^{(\gamma)} | \leq (3/2) A^{2/3}$. Use the triangle inequality to conclude that, 
\begin{align}
\prob_p ( | Z_{n,v}^{(\gamma)} - \E_p Z_{n,v}^{(\gamma)} | \geq 4A^{2/3} ) &\leq \prob_p ( | Z_{n,v}^{(\gamma)} - M_{n,v}^{(\gamma)}| + |M_{n,v}^{(\gamma)} - \E_p Z_{n,v}^{(\gamma)} | \geq 4A^{2/3}) \,, \\
&\leq \prob_p( |Z_{n,v}^{(\gamma)} - M_{n,v}^{(\gamma)}| \geq 2A^{2/3} )\,.
\end{align}
Use Theorem \ref{tala} again to complete the proof:
\begin{align}
\prob_p \left( | Z_{n,v}^{(\gamma)} - \E_p Z_{n,v}^{(\gamma)} | \geq 4A^{2/3} \right) &\leq 4 \exp\left( -c \min\left( \frac{4}{\gamma}A^{1/3}, 2A^{2/3} \right)\right) \,.
\end{align}
\end{proof}

To use Proposition \ref{trunc_con} on the $\frak{X}(0,v,n)$, we need the following input. For a percolation configuration $\om$, let $N_{n,v}(\om)$ denote the minimum cardinality $|S|$ over all cutsets $S$ in $\dcyl(\sfS(v),1,n)$ separating $\dhemi^\pm(\sfS(v),1,n)$ such that $|S|_\om = [\frak{X}(0,v,n)](\om)$. 
\begin{prop} (Rossignol-$\Th$ \cite{RoTh}, Proposition 4.2)\, Let $d \geq 2$ and let $p >p_c(d)$. There are positive constants $\gamma(p,d), c_1(p,d)$ and  $c_2(p,d)$ so that for all $u >0$, all $v \in \mathbb{S}^{d-1}$ and all $n \geq 1$, 
\begin{align}
\prob_p( N_{n,v} \geq \gamma u \text{ and } \frak{X}(0,v,n) \leq u) \leq c_1 \exp ( - c_2 u ) \,.
\end{align}
\label{nobig}
\end{prop}

Using Proposition \ref{nobig} with Proposition \ref{trunc_con}, we deduce the following.

\begin{coro} Let $d\geq 2$, $p > p_c(d)$, $v \in \mathbb{S}^{d-1}$ and let $\e >0$. There are positive constants $c_1(p,d,\e)$ and $c_2(p,d,\e)$ so that for all $n \geq 1$,
\begin{align}
\prob_p \left( \frac{ | \frak{X}(0,v,n) - \E_p \frak{X}(0,v,n) |}{(2n)^{d-1}} \geq \e \right) \leq c_1 \exp \left(-c_2 n^{(d-1)/3} \right) \,.
\end{align}

\label{con}
\end{coro}

\begin{proof} As remarked at the end of the proof of Proposition \ref{beta}, uniformly in $v \in \mathbb{S}^{d-1}$ and all configurations $\om$, $[\frak{X}(0,v,n)](\om) \leq c(d) n^{d-1}$ for some $c(d) >0$. Apply Proposition \ref{nobig} with $u = c(d)n^{d-1}$ to obtain $\gamma(p,d)$ so that  
\begin{align}
\prob_p( N_{n,v} \geq \gamma c(d) n^{d-1} ) \leq c_1 \exp \left(- c_2 c(d)n^{d-1} \right) \,.
\label{eq:4_number}
\end{align}
We use this bound shortly. For this $\gamma$ and for $\e > 0$, use Proposition \ref{trunc_con} to obtain $c_1(p,\gamma,\e)$, $c_2(p,\gamma,\e) > 0$ so that
\begin{align}
\prob_p \left( \frac{ | \frak{X}(0,v,n) - \E_p \frak{X}(0,v,n)|}{(2n)^{d-1}} \geq \e \right) &\leq \prob_p ( Z_{n,v}^{(\gamma)} \neq \frak{X}(0,v,n) ) + \prob_p \left( \frac{ |Z_{n,v}^{(\gamma)} - \E_p Z_{n,v}^{(\gamma)} |}{(2n)^{d-1} } \geq \e  \right) \,,\\
&\leq \prob_p( Z_{n,v}^{(\gamma)} \neq \frak{X}(0,v,n) ) + c_1 \exp \left( -c_2 n^{(d-1)/3}\right) \,.
\end{align} 
As$\{ Z_{n,v}^{(\gamma)} \neq \frak{X}(0,v,n) \} \subset \{ N_{n,v} \geq \gamma c(d) n^{d-1} \}$, (\ref{eq:4_number}) implies
\begin{align}
\prob_p \left( \frac{ | \frak{X}(0,v,n) - \E_p \frak{X}(0,v,n)|}{(2n)^{d-1}} \geq \e \right) &\leq \prob_p \left( N_{n,v} \geq \gamma c(d) n^{d-1} \right) + c_1 \exp \left( -c_2n^{(d-1)/3} \right) \,, \\
&\leq c_1 \exp\left(-c_2 c(d)n^{d-1} \right) + c_1 \exp \left( -c_2 n^{(d-1)/3} \right) \,.
\end{align} 
The proof is complete. \end{proof}

We obtain the desired concentration estimates by combining Corollary \ref{con} with Proposition \ref{uniform}. The following is the main result of the section.  

\begin{thm} Let $d\geq 2$, $p > p_c(d)$, and let $\e > 0$. There are positive constants $c_1(p,d,\e), c_2(p,d,\e)$ so that for all $x \in \R^d, v \in \mathbb{S}^{d-1}$, and all $r > 0$, 
\begin{align}
\prob_p \left( \left| \frac{\frak{X}(x,v,r)}{(2n)^{d-1}} - \beta_{p,d}(v) \right| \geq \e \right) \leq c_1 \exp \left(-c_2 r^{(d-1)/3} \right) \,.
\label{eq:4_concentration}
\end{align}

\label{concentration}
\end{thm}

\begin{proof} We first prove \eqref{eq:4_concentration} in the case that $x =0$ and $r = n \in \N$. Use Proposition \ref{uniform} to choose $n_0(\e)$ large so that for all $v \in \mathbb{S}^{d-1}$, $n \geq n_0$ implies
\begin{align}
\left| \frac{ \E_p \frak{X} (0,v,n)}{(2n)^{d-1}} - \beta_{p,d}(v) \right| < \e/2 \,.
\end{align}
For $n \geq n_0$,
\begin{align}
\prob_p\left( \left| \frac{\frak{X}(0,v,n)}{(2n)^{d-1}} - \beta_{p,d}(v) \right| \geq \e \right) &\leq \prob_p\left( \frac{ | \frak{X}(0,v,n) - \E_p \frak{X}(0,v,n) |}{ (2n)^{d-1}} + \left| \frac{\E_p\frak{X}(0,v,n)}{(2n)^{d-1}} - \beta_{p,d}(v) \right| \geq \e \right) \,, \\
&\leq  \prob_p\left( \frac{ | \frak{X}(0,v,n) - \E_p \frak{X}(0,v,n) |}{ (2n)^{d-1}}  \geq \e/2 \right) \,,
\end{align}
and \eqref{eq:4_concentration} is shown to hold by applying Corollary \ref{con} to the right-hand~side.

Now consider general $x$ and $r$; we claim that for any $r >0$ and $x \in \R^d$, there is $c(d) > 0$ so that 
\begin{align}
\frak{X}(x,v,\lceil r \rceil) - c(d) r^{d-2} \leq \frak{X}(x,v,r) \leq \frak{X}(x,v, \lfloor r \rfloor ) + c(d)r^{d-2} \,.
\label{eq:concentration}
\end{align}
To see this, let $A$ be the collection of edges having non-empty intersection with 
\begin{align}
\cal{N}_{5d}( r \sfS(v) \setminus \lfloor r \rfloor \sfS(v) ) \,,
\end{align}
and let $E$ be a minimal cutset in $\dcyl( \sfS(v), 1, \lfloor r \rfloor)$ separating $\dhemi^\pm ( \sfS(v),1, \lfloor r \rfloor)$.
The now standard argument from Lemma \ref{lem:cutset_reference} tells us the edges of $E \cup A$ contained in $\dcyl( \sfS(v),1,r)$ separate $\dhemi^\pm( \sfS(v),1,r)$. That $|A | \leq c(d) r^{d-2}$ establishes the upper bound on $\frak{X}(0,v,r)$ in \eqref{eq:concentration}, and we obtain the lower bound through a similar procedure. 

The proof of Lemma \ref{center} tells us that for $x \in \R^d$ and $r > 0$, there exists $x' \in \Z^d$ so that 
\begin{align}
\frak{X}(x', v, r + d^{1/2} ) - c(d) r^{d-2} \leq \frak{X}(x,v,r) \leq \frak{X}(x',v,r - d^{1/2} ) + c(d) r^{d-2} \,.
\end{align}
Apply (\ref{eq:concentration}) to conclude
\begin{align}
\frak{X}\left(x', v, \left\lceil r + d^{1/2} \right\rceil \right) - c(d) r^{d-2} \leq \frak{X}(x,v,r) \leq \frak{X} \left(x',v, \left\lfloor r - d^{1/2} \right\rfloor \right) + c(d) r^{d-2} \,.
\end{align}
As $x' \in \Z^d$, the variables $\frak{X}(x', v, \lceil r + d^{1/2}\rceil )$ and $\frak{X}(x',v, \lfloor r - d^{1/2} \rfloor)$ have the same law as $\frak{X}(0, v, \lceil r + d^{1/2}\rceil )$ and $\frak{X}(0,v, \lfloor r - d^{1/2} \rfloor)$ respectively, so concentration estimates \eqref{eq:4_concentration} established in the case $x = 0$ and $r = n \in \N$ hold for these variables as well. Within the high probability event
\begin{align}
\left\{ \left| \frac{ \frak{X}\left(x', v, \left\lceil r + d^{1/2} \right\rceil \right)}{ \left(2 \left \lceil r + d^{1/2} \right\rceil \right)^{d-1} } - \beta_{p,d}(v) \right| < \e \right\} \cap \left\{ \left| \frac{ \frak{X} \left(x', v, \left\lfloor r - d^{1/2} \right\rfloor \right)}{ \left(2 \left\lfloor r - d^{1/2} \right\rfloor \right)^{d-1} } - \beta_{p,d}(v) \right| < \e \right\} \,,
\end{align}
and for $r$ taken large depending on $\e$ and $d$, we obtain
\begin{align}
\beta_{p,d}(v) - 3\e \leq \frac{ \frak{X}(x,v,r) }{ (2r)^{d-1} } \leq \beta_{p,d}(v) + 3 \e\,,
\end{align}
completing the proof. \end{proof}




{\large\section{\B{Consequences of concentration estimates}}\label{sec:consequences}}

We now derive important consequences of Theorem \ref{concentration}. In Section \ref{sec:consequences_1}, we obtain information about the random variables $\hemi$ and $\face$ for cylinders with small height, a crucial input for Section \ref{sec:final}. In Section \ref{sec:consequences_2}, we show any polytope $P \subset [-1,1]^d$ satisfying $\cal{L}^d(P) \leq 2^d / d!$ gives an upper bound on $\Chee$. Specializing these results to a sequence of polytopes which are progressively better approximates of the Wulff crystal, we obtain the easier half of Theorem \ref{main_benj}. \newline

\subsection{Lower bounds for cuts in thin cylinders}\label{sec:consequences_1} We apply our concentration estimates to random variables $\face$ for cylinders of small height. It is important to recall the convention established in Remark \ref{rmk:3_suitable}. Throughout this section, adopt the notation $\sfS(x,v) := \sfS(v) + x$.

\begin{lem} Let $d\geq 2$, $p > p_c(d)$ and let $\e > 0$. There exists $\eta(p,d,\e) > 0$ small and positive constants $c_1(p,d,\e)$, $c_2(p,d,\e)$ so that for all $x \in \R^d$, all $v \in \mathbb{S}^{d-1}$, $h \in (0, \eta)$, and $r >0$ taken sufficiently large depending on $h$, we have
\begin{align}
\prob_p \left( \face (\sfS(x,v), h, r) \leq (1- \e) \cal{H}^{d-1}( r \sfS(x,v) ) \beta_{p,d}(v) \right) \leq c_1\exp \left( -c_2 r^{(d-1)/3} \right)  \,.
\end{align}

\label{good_face_square}
\end{lem}

\begin{proof} Write $S$ for  $\sfS(x,v)$ and consider a minimal cutset $E$ in $\dcyl(S,h,r)$ separating $\dface^\pm(S,h,r)$. Recall that $S_h^+$ and $S_h^-$ are the top and bottom faces of the cylinder $\cyl(S,h)$, and consider the collection of edges $A$ which intersect 
\begin{align}
\cal{N}_{5d} ( r ( \pa \cyl(S,h) \setminus (S_h^+ \cup S_h^- ) ) ) \,.
\end{align}
The edges of $E \cup A$ contained in $\dcyl(S,h,r)$ separate $\dhemi^{\pm}(S,h,r)$ within $\dcyl(S,h,r)$, and hence also $\dhemi^\pm(S,1,r)$ in the larger cylinder $\dcyl(S,1,r)$, provided that $h$ and $r$ are suitable for $S$ in the sense of Remark \ref{rmk:3_suitable}. It is for this reason we must take $r$ large depending on $h$. By construction, the cardinality of $A$ is at most $c(d) hr^{d-1}$ for some $c(d) > 0$, and
\begin{align}
\frak{X}(x,v,r) \leq \face(S,h,r) + c(d) hr^{d-1} \,.
\end{align}
Thus,
\begin{align}
\Big \{ \face(S, h,r) \leq (1-\e) \cal{H}^{d-1}(rS) \beta_{p,d}(v) \Big \} &\subset \Big\{ \frak{X}(x,v,r) \leq (1-\e) \cal{H}^{d-1}(rS) \beta_{p,d}(v) + c(d) hr^{d-1} \Big\} \,, \\
&\subset \Big\{ \frak{X}(x,v,r) \leq (1-\e/2) \cal{H}^{d-1}(rS) \beta_{p,d}(v) \Big\} \,,
\end{align}
where $h$ is chosen small depending on $p,d,\e$ to obtain the second line directly above. We complete the proof by applying Theorem \ref{concentration} to the event on the second line. \end{proof}

We now prove the analogue of Lemma \ref{good_face_square} for cylinders of small height based at discs.

\begin{prop} Let $d\geq 2$, $p> p_c(d)$ and let $\e >0$. Given $x \in \R^d$ and $v \in \mathbb{S}^{d-1}$, let $D(x,v)$ be the isometric image of the unit Euclidean ball in $\R^{d-1}$ centered at $x$ and oriented so that $\hyp(D(x,v))$ is orthogonal to $v$. There is $\eta(p,d,\e) > 0$ small and positive constants $c_1(p,d,\e)$ and $c_2(p,d,\e)$ so that for all $x \in \R^d, v \in \mathbb{S}^{d-1}$, $h \in (0, \eta)$ and $r >0$ sufficiently large depending on $h$, we have
\begin{align}
\prob_p \Big( \face( D(x,v), h, r) \leq (1-\e) \cal{H}^{d-1} (rD(x,v)) \beta_{p,d}(v)  \Big) \leq c_1 \exp \left(-c_2 r^{(d-1)/3} \right) \,.
\end{align}
\label{disc}
\end{prop}

\begin{proof} Figure \ref{fig:square_tile} captures the idea of this proof: we tile the disc with a collection of small, nearly exhaustive squares. Let $\e' > 0$, write $D = D(x,v)$ and let $D' := \{ x \in \R^{d-1}: |x|_2 \leq 1\}$. Let $\vp : D' \to \R^d$ be an isometry taking $D'$ to $D$ and let $\Delta^k \equiv \Delta^{k,d-1}$ be the dyadic squares in $[-1,1]^{d-1}$ at scale $k$. Choose $k \in \N$ large enough (depending on $\e'$ and $d$) so that 
\begin{align}
\cal{L}^{d-1} \left( D' \setminus \bigcup_{S' \in \Delta^k,\, S' \subset D' } S' \right) \leq \e' \cal{L}^{d-1}(D') \,,
\end{align}
and enumerate the squares $S' \in \Delta^k$ with $S' \subset D'$ as $S_1', \dots, S_m'$. The number $m$ of these squares depends on $\e'$ and $d$. Shrink each square slightly to form a new disjoint collection $\{ S_i''\}_{i=1}^m$ of closed squares. Specifically, $S_i''$ shall be the $(1-\delta)$-dilate of $S_i'$ about its center for some $\delta \in (0,1)$. For each $i$, define $S_i := \vp(S_i'')$ and choose $\delta$ small enough (also depending on $\e'$ and $d$) so that
\begin{align}
\cal{H}^{d-1} \left( D \setminus \bigcup_{i=1}^m S_i \right) \leq 2\e' \cal{H}^{d-1}(D) \,.
\end{align}

We arrange that the isometry $\vp$ is compatible with the chosen orientation, in that for each $i$, $S_i = \al \sfS(y_i, v)$ for some $y_i \in \R^d$, where $\al := (1-\delta)2^{-k}$. Let $\e >0$, choose $\eta = \eta(p,d,\e/2)$ as in Lemma \ref{good_face_square} and let $h \in (0, \al \eta)$. 

Let $E$ be a minimal cutset separating $\dface^\pm(D,h,r)$ within $\dcyl(D,h,r)$. Let $E_i$ denote the edges of $E$ lying in edge set of $\dcyl(S_i,h,r)$. Each $E_i$ separates $\dface^\pm(S_i,h,r)$ within $\dcyl(S_i,h,r)$, so by the disjointness of the $\{S_i\}_{i=1}^m$, we have
\begin{align}
\sum_{i=1}^m \face \left(S_i, h,r \right) \leq \face(D,h,r) \,,
\end{align}
and thus, 
\begin{align}
\prob_p \Big( \face(D,h,r) \leq (1- \e) &\cal{H}^{d-1} (rD) \beta_{p,d}(v) \Big) \\
&\leq \prob_p \left( \sum_{i=1}^m \face(S_i, h,r) \leq (1- \e) \cal{H}^{d-1} (rD) \beta_{p,d}(v) \right) \,,\\
&\leq \prob_p \left( \sum_{i=1}^m \face(S_i, h,r) \leq \frac{(1- \e)}{1-2\e'} \sum_{i=1}^m \cal{H}^{d-1} (rS_i) \beta_{p,d}(v) \right) \,,
\label{eq:5_disc_1}
\end{align}
with \eqref{eq:5_disc_1} following from our choice of $\delta$ and the squares $S_i'$. As $\cal{H}^{d-1}(rS_i)$ is the same for each $i$, a union bound gives
\begin{align}
\prob_p \Big( \face(D,h,r) \leq (1- \e) &\cal{H}^{d-1} (rD) \beta_{p,d}(v) \Big)\\
 &\leq \sum_{i=1}^m \prob_p \left( \face(S_i, h,r) \leq \left( \frac{1-\e}{1-2\e'}\right) \beta_{p,d}(v) \cal{H}^{d-1}(rS_i) \right) \,, \\
&\leq \sum_{i=1}^m \prob_p \left( \face(S_i,h,r) \leq (1- \e/2) \beta_{p,d}(v) \cal{H}^{d-1}(rS_i) \right) \,.
\label{eq:5_disc_2}
\end{align}
To obtain \eqref{eq:5_disc_2}, we have taken $\e'$ small enough so that $1- \e/2 > \tfrac{1-\e}{1-2\e'}$. Thus, $m$ and $\al$ now depend on $\e$ and $d$. Use that $\vp$ was chosen to be compatible with the chosen orientation $\sfS$, writing each $S_i$ as $\al \sfS(y_i,v)$ for some $y_i$. Making this switch in \eqref{eq:5_disc_2}, we find
\begin{align}
\prob_p \Big( &\face(D,h,r) \leq (1- \e) \cal{H}^{d-1} (rD) \beta_{p,d}(v) \Big)\\
&\leq \sum_{i=1}^m \prob_p \left( \face(\sfS(y_i,v), h/\al,\al r) \leq (1- \e/2) \beta_{p,d}(v) \cal{H}^{d-1}(\al r\sfS(y_i,v)) \right) \,.
\end{align}
Having chosen $h$ so that $h /\al \leq \eta$, we apply Lemma \ref{good_face_square} to each summand directly above, using $\al r$ in place of $r$ and $\e/2$ in place of $\e$ in the statement of this lemma:
\begin{align}
\prob_p \Big( \face(D,h,r) \leq (1- \e) \cal{H}^{d-1} (rD) \beta(v) \Big)\leq m c_1 \exp \left(-c_2 (\al r)^{(d-1)/3} \right) \,.
\end{align}
To apply Lemma \ref{good_face_square}, $r$ is taken large depending on $h$. We take $\al \eta$ to be the $\eta$ in the statement of this proposition and complete the proof by renaming constants and noting their dependencies. \end{proof}

Until this point, we have only used the concentration estimates from Section \ref{sec:concentration} to show the random variables $\face$ cannot be too small. In the next subsection, we put the complementary estimates to use.

\subsection{Upper bounds on $\Chee$, or efficient carvings of ice}\label{sec:consequences_2}

A \emph{convex polytope} is a compact subset of $\R^d$ which can be written as a finite intersection of closed half-spaces. A \emph{polytope} is a compact subset of $\R^d$ which may be written as finite union of convex polytopes. We do not require polytopes to be connected subsets of $\R^d$, but say a polytope is \emph{connected} if its interior is a connected subset of $\R^d$. A polytope $P \subset \R^d$ is a \emph{$d$-polytope} if it is non-degenerate ($\cal{L}^d(P) >0$).

To obtain upper bounds on $\Chee$, we use a polytope $P$ to obtain a valid subgraph of $\B{C}_n$, controlling both the volume and open edge boundary of this subgraph. Equivalently, we view $\B{C}_n$ as a block of ice, and we use the dilate $nP$ as a blueprint for carving this block. 

Our first task is to perform an \emph{efficient} carving at the boundary of $nP$, and this is where the other side of our concentration estimates are used. The next result allows us to work on each face of the polytope $P$ individually. We think of a $(d-1)$-polytope $\sig \subset \R^d$ as a face of a $d$-polytope $P$. For such $\sig$, let $v_\sig$ denote one of the unit vectors orthogonal to $\hyp(\sig)$. 

\begin{prop} Let $d \geq 2$, $p > p_c(d)$ and $\e >0$. Let $\sig \subset \R^d$ be a connected $(d-1)$-polytope. There is a positive constant $\eta(p,d,\e,\sig)$ and another connected $(d-1)$-polytope $\wt{\sig}$ depending on $p, d, \e$ and $\sig$ so that:

(i) $\wt{\sig} \subset \sig$, $\wt{\sig} \cap \cal{N}_\eta (\pa \sig) = \emptyset$ and $\cal{H}^{d-1}( \sig \setminus \wt{\sig} ) \leq \e \cal{H}^{d-1}(\sig)$.

(ii) There are positive constants $c_1(p,d,\e,\sig)$ and  $c_2(p,d,\e,\sig)$ so that when $h \in (0, \eta)$ and for all $r>0$ sufficiently large depending on $h$,
\begin{align}
\prob_p \Big( \hemi(\wt{\sig},h,r) \geq (1+\e) \cal{H}^{d-1}(r \sig) \beta_{p,d}(v_\sig) \Big) \leq c_1 \exp \left(-c_2 r^{(d-1)/3} \right) \,.
\end{align}

\label{poly_face}
\end{prop}

\begin{proof} Let $\e' > 0$, and for the parameter $\eta >0$, define $\wt{\sig}$ to be the closure of $\sig \setminus \cal{N}_{2\eta}^{(1)} (\pa \sig)$, where we recall from Section \ref{sec:gmt_misc} the notation for the $\ell^1$-neighborhood of a set. Then $\wt{\sig}$ is a $(d-1)$ polytope, and because $\sig$ is connected, we may choose $\eta$ sufficiently small depending on $\sig$ and $\e$ so that $\wt{\sig}$ is also connected, and so that 
\begin{align}
\cal{H}^{d-1}( \sig \setminus \wt{\sig} ) \leq \e' \cal{H}^{d-1}(\sig) \,.
\end{align}
For such $\wt{\sig}$, property \emph{(i)} is already satisfied after requiring $\e' \leq \e$. 

To show \emph{(ii)} holds, we employ the strategy used in the proof of Proposition \ref{disc}. Let $h \leq \eta$, and let $\wt{\sig}' \subset \R^{d-1}$ be a $(d-1)$-polytope with an isometry $\vp: \wt{\sig}' \to \wt{\sig}$. Choose $k \in \N$ to be smallest such that $2^{-k} < h$, but large enough so that 
\begin{align}
\cal{L}^{d-1}\left( \wt{\sig}' \setminus \bigcup_{S' \in \Delta^k,\, S' \subset \wt{\sig}'} S' \right) \leq \e' \cal{L}^{d-1}(\wt{\sig}' ) \,,
\end{align}
where, as before, $\Delta^k \equiv \Delta^{k,d-1}$ denotes the dyadic squares in $[-1,1]^{d-1}$ at scale $k$. Enumerate such squares contained in $\wt{\sig}'$ as $S_1', \dots, S_m'$. Let $\delta >0$ and dilate each $S_i'$ about its center by a factor of $(1-\delta)$ to produce a new, disjoint collection $\{ S_i'' \}_{i=1}^m$ of closed squares contained in $\wt{\sig}'$. Let $S_i = \vp(S_i'')$, and choose $\delta$ small enough so that 
\begin{align}
\cal{H}^{d-1} \left( \wt{\sig} \setminus \bigcup_{i=1}^m S_i \right) < 2 \e' \cal{H}^{d-1}(\sig) \,.
\end{align}
As before, write $\al := (1-\delta)2^{-k}$; we lose no generality assuming $\wt{\sig}'$ and $\vp$ are compatible with $\sfS$, so that each $S_i$ is $\al \sfS(y_i, v_\sig)$ for some $y_i \in \R^d$. For each $i$, let $E_i$ denote a cutset in $\dcyl(S_i,\al,r)$ separating $\dhemi^\pm(S_i, \al, r)$. Let $A$ denote the edges intersecting
\begin{align}
\cal{N}_{5d}\left( r \left( \wt{\sig} \setminus \bigcup_{i=1}^m S_i \right) \right)
\end{align}
so that $| A | \leq c(d) \e' \cal{H}^{d-1}(r \sig)$ for some $c(d) >0$. The now standard argument from Lemma~\ref{lem:cutset_reference} tells us the edges of $A \cup \bigcup_{i=1}^m E_i$ contained in $\dcyl(\wt{\sig}, h, r)$ separate $\dhemi^\pm( \wt{\sig}, h,r)$. Here $r$ is taken sufficiently large depending on $h$ to ensure this argument goes through. It is here we use that $k$ satisfies $\al \leq 2^{-k} \leq h$. Under these conditions, we conclude 
\begin{align}
\hemi(\wt{\sig}, h, r) \leq c(d) \e' \cal{H}^{d-1}(r \sig) + \sum_{i=1}^m \hemi( S_i,\al, r)\,,
\end{align}
and thus, 
\begin{align}
\prob_p \Big( \hemi(\wt{\sig}, h, &r) \geq (1 + \e) \cal{H}^{d-1} (r \sig) \beta_{p,d}( v_\sig) \Big) \\
&\leq \prob_p \left( \sum_{i=1}^m \hemi (S_i, \al, r ) \geq ( 1+ \e - c(p,d) \e') \cal{H}^{d-1}(r\sig) \beta_{p,d}(v_\sig) \right) \,.
\end{align}
Now choose $\e'$ small enough depending on $p,d,\e$ so that $1 + \e - c(p,d)\e' \geq 1 + \e/2$:
\begin{align}
\prob_p \Big( \hemi(\wt{\sig}, h, &r) \geq (1 + \e) \cal{H}^{d-1} (r \sig) \beta_{p,d}( v_\sig) \Big) \\
&\leq \prob_p\left( \sum_{i=1}^m \hemi (S_i, \al, r ) \geq ( 1+ \e/2) \sum_{i=1}^m \cal{H}^{d-1}(r S_i) \beta_{p,d}(v_\sig) \right) \,,\\
&\leq \sum_{i=1}^m \prob_p\left(  \hemi (S_i, \al, r ) \geq ( 1+ \e/2)\cal{H}^{d-1}(r S_i)  \beta_{p,d}(v_\sig)  \right) \,, \\
&\leq \sum_{i=1}^m \prob_p\left(  \frak{X}(y_i, v_\sig, \al r) \geq ( 1+ \e/2)  \cal{H}^{d-1}(\al r \sfS(y_i,v_\sig)) \beta_{p,d}(v_\sig) \right) \,.
\end{align}
Here we have used a union bound and that each $S_i = \al \sfS(y_i, v_\sig)$ for some $y_i \in \R^d$. Applying Theorem \ref{concentration} to each summand on the right, we obtain
\begin{align} 
\prob_p \Big( \hemi(\wt{\sig}, h, r) \geq (1 + \e) \cal{H}^{d-1} (r \sig) \beta_p( v_\sig) \Big) \leq m c_1 \exp\left( -c_2 (\al r)^{(d-1)/3} \right) \,,
\end{align}
completing the proof. \end{proof}

We now use a $d$-polytope $P$ to obtain a high probability upper bound on $\Chee$ in terms of the conductance of $P$. 

\begin{thm} Let $d \geq 2$ and let $p > p_c(d)$. Let $P \subset [-1,1]^d$ be a polytope such that $\cal{L}^d(P) \leq 2^d / d!$, and let $\e >0$. There exist positive constants $c_1(p,d,\e,P)$ and $c_2(p,d,\e,P)$ so that 
\begin{align}
\prob_p \left( \Chee \geq (1 + \e) \left(\frac {\cal{I}_{p,d}(nP) }{\theta_p(d) \cal{L}^d(nP) } \right) \right) \leq c_1 \exp \left(-c_2 n^{(d-1)/3} \right) \,.
\end{align}

\label{upper_bound}
\end{thm}

\begin{proof} Begin by working with $P_\delta := (1-\delta)P$ for $\delta \in (0,1)$, so that the Euclidean distance from $P_\delta$ to $\pa [-1,1]^d$ is positive. We choose $\delta$ carefully at the end of the argument. Let $\e, \e' >0$ and enumerate the faces of $P_\delta$ as $\sig_1, \dots, \sig_m$, suppressing the dependence of these faces on $\delta$. Use Proposition \ref{poly_face}, picking $\eta$ depending on $\e'$ and on each face $\sig_i$; for each $i$, we produce $\wt{\sig}_i$ so that
\begin{enumerate}
\item  $\cal{H}^{d-1}( \sig_i \setminus \wt{\sig}_i) \leq \e' \cal{H}^{d-1}(\sig_i)$
\item $\wt{\sig}_i \subset \sig_i$, $\wt{\sig}_i \cap \cal{N}_\eta (\pa \sig_i) = \emptyset$ and $\cal{H}^{d-1}( \sig_i \setminus \wt{\sig}_i ) \leq \e' \cal{H}^{d-1}(\sig_i)$. \label{eq:item_two_upper_bound}
\item There are positive constants $c_1(p,d,\e', P, \delta)$ and $c_2(p,d,\e',P,\delta)$ so that if $h \in(0, \eta)$, and if $r > 0$ is sufficiently large depending on $h$, 
\begin{align}
\prob_p \Big( \hemi(\wt{\sig}_i,h,r) \geq (1+\e') \cal{H}^{d-1}(r \sig_i) \beta_{p,d}(v_\sig) \Big) \leq c_1 \exp \left(-c_2 r^{(d-1)/3} \right) \,.
\label{eq:upper_bound_event}
\end{align}
\end{enumerate}

Assume $\eta$ is small enough so that $\cal{N}_\eta(P_\delta)$ is contained in $[-1,1]^d$, and choose $h \in (0,\eta)$, \emph{henceforth fixed}, so that the cylinders $\big\{ \cyl( \wt{\sig}_i, h, n ) \big\}_{i=1}^m$ are disjoint. We will use Proposition~\ref{poly_face} in each cylinder $\cyl( \wt{\sig}_i, h,n)$ to control the open edge boundary of a subgraph of $\giant$, to be constructed momentarily. Before doing so, we position ourselves to control the volume of this subgraph. 

Let $Q_1, \dots, Q_\ell$ enumerate the dyadic cubes at scale $k$ within $[-1,1]^d$. Suppose these cubes are ordered so that for $\ell_1 \leq \ell_2 \in \{1, \dots, \ell\}$, the collection $Q_1, \dots, Q_{\ell_2}$ enumerates all cubes intersecting $\cal{N}_\eta(P_\delta)$, and that $Q_1, \dots, Q_{\ell_1}$ enumerates all cubes contained in $P_\delta \setminus \cal{N}_\eta(P_\delta)$. Take $k$ sufficiently large and take $\eta$ smaller if necessary so that 
\begin{align}
\cal{L}^d \left( \bigcup_{j= \ell_1 + 1}^{\ell_2} Q_j \right) < \e' \cal{L}^d(P_\delta) \,,
\end{align}
and for each $j \in \{1, \dots, \ell\}$, let $\cal{E}_n^{(j)}$ be the event that 
\begin{align}
\left\{ \frac{ | \B{C}_\infty \cap nQ_j|}{ \cal{L}^d(n Q_j) } \in ( \theta_p(d) - \e' , \theta_p(d) + \e' ) \right\} \,.
\label{eq:carve_vol_events}
\end{align}

We now construct a subgraph of $\giant$ from $P$. Fix a percolation configuration $\om$, and for each face $\sig_i$, let $E_n^{(i)}(\om)$ denote a cutset within $\dcyl( \wt{\sig}_i, h,n)$ separating $\dhemi^\pm(\wt{\sig}_i,h,n)$, with $|E_n^{(i)}(\om) |_\om =\hemi(\wt{\sig}_i, h,n)$ in the configuration $\om$. Let $A_n$ be the edges intersecting
\begin{align}
\cal{N}_{5d} \left( n \left( \pa P_\delta \setminus \bigcup_{i=1}^m \wt{\sig}_i \right) \right) \,.
\end{align}
We chose $\wt{\sig}_i$ to satisfy \eqref{eq:item_two_upper_bound}, thus $|A_n| \leq c(d) \e' \cal{H}^{d-1}(\pa P_\delta) n^{d-1} $ for some $c(d)>0$. Define 
\begin{align}
\Gamma_n : = \left( \bigcup_{i=1}^m E_n^{(i)}(\om) \right) \cup A_n \,.
\end{align}

\begin{figure}[h]
\centering
\includegraphics[scale=1.5]{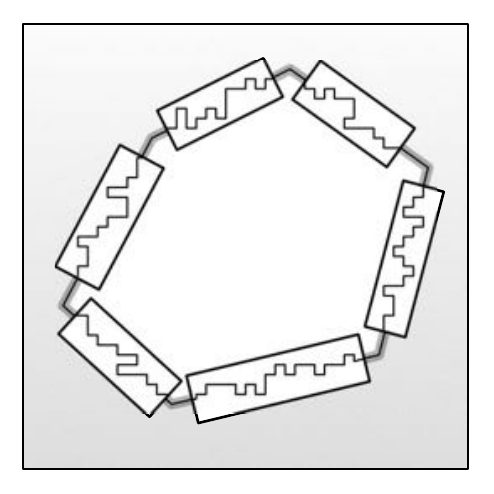}
\caption{The polytope $nP$ has six faces. Each of the boxes at the boundary of $nP$ is one of the $\cyl( \wt{\sig}_i, h, n)$, and within each is the corresponding cutset $E_n^{(i)}$. The set $A_n$ is the grey outline of each corner.} 
\label{fig:carve}
\end{figure}

Define the vertex set $H_n(\om)$ to be all $x \in \B{C}_n$ separated from $\infty$ by $\Gamma_n(\om)$. The proof of Lemma \ref{lem:cutset_reference} tells us $H_n(\om)$ is non-empty and in fact contains every vertex $x \in \B{C}_\infty \cap Q_j$ for $j \in \{1, \dots, \ell_1\}$. For this proof to go through, we ensure $n$ is large enough depending on $h$ so that the cylinders $\cyl( \wt{\sig}_i, h, n)$ are suitable in the sense of Remark \ref{rmk:3_suitable}. As $h$ has been fixed and depends only on $\e'$ and $P$, this is no issue, and $H_n(\om)$ is well-defined for all $n$ sufficiently large. Though $H_n(\om)$ was defined as a collection of vertices, we now view it as a graph whose structure comes from restricting $\B{C}_n$ and suppress the dependence of $H_n$ and $\Gamma_n$ on the percolation configuration. 

We now exhibit control on the volume and open edge boundary of $H_n$, first working within the intersection $\cal{E}$ of the high probability events $\cal{E}_n^{(j)}$ defined in \eqref{eq:carve_vol_events} to control $|H_n|$. For all $\om \in \cal{E}$, 
\begin{align}
( \theta_p(d) - \e') \left( \sum_{j=1}^{\ell_1} \cal{L}^d(nQ_j) \right) - \ell c(d) \left(2^{-k} n \right)^{d-1}  \leq |H_n| \leq (\theta_p(d) + \e') \sum_{j=1}^{\ell_2} \cal{L}^d(nQ_j) \,,
\end{align}
where the term subtracted on the left arises because the $Q_j$ only have disjoint \emph{interiors}. For $n$ sufficiently large (depending on $p,d,\e',P$), we have 
\begin{align}
( \theta_p(d) - 2\e') \sum_{j=1}^{\ell_1} \cal{L}^d(nQ_j)  \leq |H_n| \leq (\theta_p(d) + \e') \sum_{j=1}^{\ell_2} \cal{L}^d(nQ_j) \,,
\end{align}
and hence that
\begin{align}
( \theta_p(d) - 2\e')(1-\e') \cal{L}^d(nP_\delta) &\leq |H_n| \leq (\theta_p(d) + \e')(1+\e') \cal{L}^d(nP_\delta) \,, \\
( \theta_p(d) - 2\e')(1-\e')(1-\delta)^d \cal{L}^d(nP) &\leq |H_n| \leq (\theta_p(d) + \e')(1+\e')(1-\delta)^d \cal{L}^d(nP) \,.
\label{eq:5_carve_1}
\end{align}

We now show that $H_n$ is a valid subgraph of $\giant$ when $\delta$ is chosen appropriately. On $\cal{E}$, 
\begin{align}
| \giant | &\geq ( \theta_p(d) - \e')(2n)^d - \ell c(d) (2^{-k} n)^{d-1} \,, \\
&\geq (\theta_p(d) - 2\e') (2n)^d \,.
\end{align}
for $n$ sufficiently large. As $\cal{L}^d(P) \leq 2^d/ d!$, choosing $\delta$ in accordance with \eqref{eq:5_carve_1} so that 
\begin{align}
(\theta_p(d) + \e') ( 1+ \e') (1-\delta)^d = (\theta_p(d) -2\e') \,,
\end{align}
ensuring $H_n$ is a valid subgraph of $\giant$ within $\cal{E}$. Defining $\delta$ this way implies $\delta \to 0$ as $\e' \to 0$. 

Not only have we shown $H_n$ is valid within a high probability event, we also have exhibited a lower bound on $|H_n|$. To bound $\Chee$, it then suffices to bound $|\pa^\om H_n|$ from above. The construction of $H_n$, guarantees $\pa^\om H_n \subset \Gamma_n$. Using the disjointness of the cylinders $\cyl(\wt{\sig}_i, h, n)$,
\begin{align}
|\pa^\om H_n| \leq \sum_{i=1}^m | E_n^{(i)} (\om)|_\om +  c(d) \cal{H}^{d-1}(\pa P_\delta) \e'n^{d-1} \,.
\end{align}
For $i\in \{1,\dots, m\}$, let $\cal{F}_n^{(i)}$ be the event in \eqref{eq:upper_bound_event}. On the intersection $\cal{F}$ of all $\cal{F}_n^{(i)}$,
\begin{align}
| \pa^\om H_n| &\leq (1+\e') \cal{I}_{p,d}(nP) + c(d) \cal{H}^{d-1}(\pa P)  \e' n^{d-1} \,, \\
&\leq (1+\e' + c(p,d) \e') \cal{I}_{p,d}(nP) \,.
\end{align}
Thus, on the intersection of $\cal{E}$ and $\cal{F}$, we have
\begin{align}
\Chee \leq  \left(\frac{ 1 + \e' + c(p,d)\e' }{ ( \theta_p(d) - 2 \e' )(1 - \e')(1- \delta)^d } \right) \frac{ \cal{I}_{p,d}(nP)}{\cal{L}^d(nP)} \,,
\end{align}
and we take $\e'$ small enough (recall $\delta = \delta(\e')$ goes to zero as $\e'$ does) so that
\begin{align}
\Chee \leq (1+\e)\frac{ \cal{I}_{p,d}(nP)}{\theta_p(d)\cal{L}^d(nP)} \,.
\end{align}
Use the bounds in Corollary \ref{density_control} on $\cal{E}$ and in Proposition \ref{poly_face} on $\cal{F}$ to conclude 
\begin{align}
\prob_p \left( \left( \cal{E} \cap \cal{F} \right)^c \right) \leq mc_1 \exp\left(-c_2 n^{(d-1)/3} \right) + \ell c_1 \exp\left(-c_2 (2^{-k} n)^{d-1} \right) \,,
\end{align}
which completes the proof, upon tracking the dependencies of $\ell$ and $k$. \end{proof}

Using Proposition \ref{poly_limit}, Theorem \ref{upper_bound} and Borel-Cantelli we deduce the following. 

\begin{coro} Let $d \geq 2$ and let $p > p_c(d)$. Consider the Wulff crystal $W_{p,d}$ corresponding $\beta_{p,d}$, and let $\e > 0$. The event below occurs $\prob_p$-almost surely:
\begin{align} \left\{ \limsup_{n \to \infty} n \Chee \leq (1 + \e) \frac{\cal{I}_{p,d}(W_{p,d}) }{\theta_p(d) \cal{L}^d(W_{p,d}) } \right\} \,.
\end{align}
\label{upper_bound_2}
\end{coro}

\begin{proof} Recall $\cal{L}^d(W_{p,d}) = 2^d / d!$. Let $\e,\e' >0$ and apply Proposition \ref{poly_limit}  to obtain a polytope $P_{\e'} \subset W_{p,d}$ with $| \cal{I}_{p,d} (P_{\e'}) - \cal{I}_{p,d} (W_{p,d}) | < \e'$ and with $\cal{L}^d( W_{p,d} \setminus P_{\e'}) < \e'$. Apply Theorem \ref{upper_bound} to $P_{\e'}$ to obtain positive constants $c_1(p,d,\e,P_{\e'})$ and $c_2(p,d,\e,P_{\e'})$ so that 
\begin{align}
\prob_p \left( n \Chee \geq (1 + \e/2) \left(\frac{1+\e'}{1-\e'}\right)  \left(\frac {\cal{I}_{p,d}(W_{p,d}) }{\theta_p(d) \cal{L}^d(W_{p,d}) } \right) \right) &\leq \prob_p \left( n \Chee \geq (1 + \e/2) \left(\frac {\cal{I}_{p,d}(P_{\e'}) }{\theta_p(d) \cal{L}^d(P_{\e'}) } \right) \right)\\
&\leq c_1 \exp \left(-c_2 n^{(d-1)/3} \right)\,.
\end{align}
Choosing $\e'$ sufficiently small depending on $\e$ and applying Borel-Cantelli completes the proof.
 \end{proof}

Corollary \ref{upper_bound_2} is the easier half of Theorem \ref{main_benj}. Before moving to the next section, we make an observation to aid in the proof of Theorem \ref{main_L1}. For $K \subset [-1,1]^d$ convex with non-empty interior, define the \emph{empirical measure} associated to $K$ as 
\begin{align}
\ov{\nu}_K(n) := \frac{1}{n^d} \sum_{ x \in \giant \cap nK} \delta_{x/n} \,.
\label{eq:5_convex_empirical}
\end{align}
Following the proof of Theorem \ref{upper_bound}, it is not difficult to deduce the following result (recall that the metric $\dw$ introduced in \eqref{eq:2_metric}). 

\begin{coro} Let $d\geq 2$, $p > p_c(d)$ and let $W \subset [-1,1]^d$ be a translate of $W_{p,d}$. For $\e >0$, there are positive constants $c_1(p,d,\e)$ and $c_2(p,d,\e)$ so that 
\begin{align}
\prob_p \left( \dw \left(\ov{\nu}_W(n), \nu_W \right) > \e\right) \leq c_1 \exp \left( - c_2 n^{d-1} \right) \,.
\end{align}
\label{facilitate}
\end{coro}

\begin{rmk} Corollary \ref{facilitate} follows from the approximation result Proposition \ref{poly_limit}, from the density result Corollary \ref{density_control} (used as in the proof of Theorem~\ref{upper_bound}) applied to a fine mesh of dyadic cubes and finally from the definition of the metric $\dw$. No concentration estimates for $\beta_{p,d}$ are needed, so the proof of Corollary \ref{facilitate} is less involved than that of Theorem \ref{upper_bound}, and we choose to omit it.
\end{rmk}




{\large\section{\B{Coarse graining}}\label{sec:coarse_original}}

Having spent the last section passing from continuous objects to discrete objects, we now move in the more difficult direction. To each $G_n \in \cal{G}_n$, we associate a set of finite perimeter $P_n \subset [-1,1]^d$ with comparable conductance. A natural candidate for $P_n$ is
\begin{align}
\frac{1}{n} \left( \bigcup_{x\, \in\, \rmV(G_n)} Q(x) \right) \,,
\label{eq:6_natural}
\end{align}
where $Q(x)$ is the unit dual cube correspinding to $x \in \Z^d$. However, the perimeter of  the set in \eqref{eq:6_natural} is directly related to $n^{-(d-1)}|\pa G_n|$ instead of $n^{-(d-1)}| \pa^\om G_n|$ and may grow with $n$ due to vacant percolation, unless $p$ is taken close to one. This suggests a renormalization argument, and indeed, in the present section we clarify and modify a procedure due to Zhang \cite{Zhang}. Because of the intricacy of the coarse graining is, we do not build the $P_n$ until Section \ref{sec:contiguity}. 

The construction works only in dimensions strictly larger than two. We comment on this in Section \ref{sec:webbing}, and for now we simply assume $d \geq 3$ for the remainder of the paper. \newline

\subsection{Preliminary notation}\label{sec:coarse_0} Let $k$ be a natural number called the \emph{renormalization parameter}. Given $x \in \Z^d$, define the \emph{$k$-cube} centered at $x$ as:
\begin{align}
\un{B}(x) := (2k)x + [-k,k]^d  \,,
\end{align}
suppressing the dependence of $\un{B}(x)$ on $k$ to avoid cumbersome notation. Underscores are used to denote sets of $k$-cubes. If $\un{G}$ is a set of $k$-cubes and $x \in \rmV(\mathbb{Z}^d)$, write $x \in \un{G}$ if $x$ is contained in one of the $k$-cubes of $\un{G}$. Likewise, if $e \in \rmE(\Z^d)$ is an edge, we write $e \in \un{G}$ if both endpoint vertices of $e$ lie in $\un{G}$. 

A \emph{$3k$-cube} centered at $x$ is defined as follows: 
\begin{align}
\un{B}_3(x) := (2k)x + [-3k,3k]^d \,.
\label{eq:6.1_3k}
\end{align}
We emphasize that $x$ lies in $\Z^d$, thus each $3k$-cube contains exactly $3^d$ $k$-cubes. Two cubes $\un{B}(x)$ and $\un{B}(x')$ are \emph{adjacent} if $x \sim x'$, or equivalently if they share a face. Two cubes $\un{B}(x)$ and $\un{B}(x')$ are \emph{$*$-adjacent} if $x \sim_* x'$, or equivalently if either $\un{B}(x') \subset \un{B}_3(x)$ or $\un{B}(x) \subset \un{B}_3(x')$. 

\subsection{Discovering a cutset}\label{sec:coarse_1} 

We describe Zhang's method in general, then apply it to the Cheeger optimizers in the next section. The idea is to form a collection of $k$-cubes containing $\pa_o G_n$, and then to discover within these cubes a more tame cutset separating $G_n$ from $\infty$.

Let $G = G(\om) \subset \B{C}_\infty$ be a finite connected graph and from $G$, define several sets of $k$-cubes:
\begin{align}
\un{G} := \Big\{ \un{B}(x) \:: \un{B}(x) \cap ( G \cup \pa_o G) \neq \emptyset \Big\}\,, \hspace{5mm}  \hspace{5mm} \un{A} := \Big\{ \un{B}(x) \:: \un{B}(x) \cap \pa_o G \neq \emptyset \Big\} \,.
\end{align}
Figure \ref{fig:zhang_1} depicts a possible $G$ and $\un{A}$. As $G$ is finite, so too is $\un{G}$, thus the cubes $\un{B}(x)$ not in $\un{G}$ split into a single infinite $*$-connected component called the \emph{ocean}, labeled $\un{Q}$, and finitely many finite $*$-connected components $\un{Q}'(1), \dots, \un{Q}'(u')$, called \emph{ponds}. 

\begin{figure}[h]
\centering
\includegraphics[scale=1]{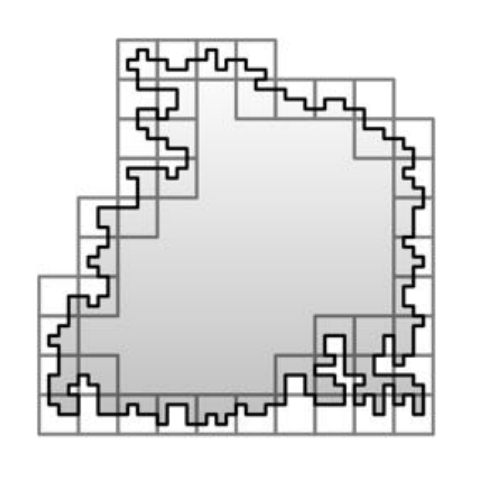}
\caption{The black contour and its interior are $\pa_oG$ and $G$ respectively. Notice that $\un{A}$, depicted by the squares covering $\pa_o G$, is not necessarily the boundary of $\un{G}$.  } 
\label{fig:zhang_1}
\end{figure}

Let $\Delta \un{Q}$ denote the $k$-cubes $*$-adjacent to a cube in the ocean $\un{Q}$ but not contained in $\un{Q}$. Likewise, for each pond $\un{Q}'(i)$, let $\Delta \un{Q}'(i)$ denote the $k$-cubes $*$-adjacent to $\un{Q}'(i)$ but not contained in $\un{Q}'(i)$. 

\begin{rmk}The next step in Zhang's construction is to pass to the unique configuration $\om'$ obtained by closing each open edge in $\pa_o G$. This is done while preserving both $G$ and $\B{C}_\infty$, so that we still work with the graphs $G(\om)$ and $\B{C}_\infty(\om)$, now with each modified by closing each open edge of $\pa_o G$. Counterintuitively, $\B{C}_\infty$ is then a disconnected graph after passing to the configuration $\om'$. This is only a \emph{formal} procedure as we eventually pass \emph{back} to the configuration $\om$. 
\label{rmk:6_cluster}
\end{rmk}

Pass to the configuration $\om'$. Each pond may intersect an open cluster connected to the ocean, and we emphasize these open clusters need not be contained in $\B{C}_\infty$. A pond is \emph{live} if it intersects an open cluster also intersecting the ocean. If $\un{Q}'(i)$ intersects an open cluster also intersecting a distinct live pond, say it is \emph{almost-live}. If $\un{Q}'(i)$ intersects an open cluster also intersecting a distinct pond labeled as almost-live, call $\un{Q}'(i)$ \emph{almost-live} also. Thus, the label almost-live \emph{propagates} through the $\un{Q}'(i)$ via open clusters, starting with the live ponds. 

A pond is \emph{dead} if it is neither live nor almost-live. Refine the collection of ponds $\{ \un{Q}'(i) \}_{i=1}^{u'}$ to the live and almost-live ponds $\un{Q}(1), \dots, \un{Q}(u)$; Figure \ref{fig:zhang_2} depicts a possible configuration of ponds. 

\begin{figure}[h]
\centering
\includegraphics[scale=1]{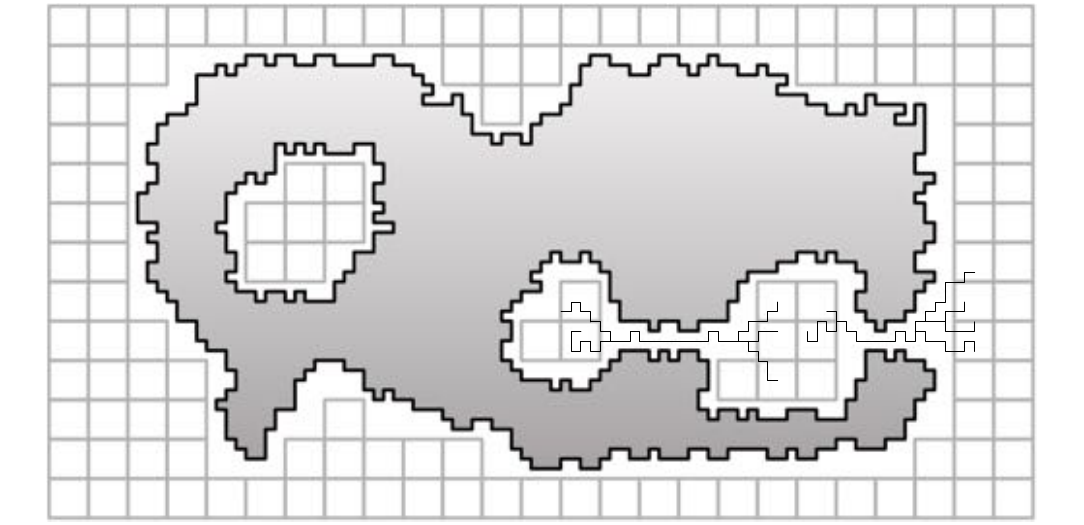}
\caption{The graph $G$ is the shaded region between closed curves. The connected components of cubes in the diagram are ponds or the ocean. The left-most pond is dead, the right-most pond is live and the middle pond is almost-live. The portions of the thin curves which do not intersect any cube represent the set $\brid$.} 
\label{fig:zhang_2}
\end{figure}

Now construct a graph called $\brid$ which joins the live and almost-live ponds. Let $C$ be the union of all open clusters intersecting $\un{Q}$, and let $C_i$ be the union of all open clusters intersecting the live or almost-live pond $\un{Q}(i)$. The components of $C$ and $C_i$ are not necessarily in $\B{C}_\infty$. To specify the vertices of $C_i$ contained in $\un{Q}(i)$ only, define $Q_i := C_i \cap \un{Q}(i)$, and likewise define $Q := C \cap \un{Q}$. Let $\brid$ be the \emph{remainder} of these components in $\un{G}$:
\begin{align}
\brid := \left[ \left( \bigcup_{\un{B}(x) \in \un{G} } \un{B}(x)  \right)  \cap \left( C \cup \left( \bigcup_{i=1}^u C_i \right) \right) \right] \setminus \left( Q \cup \left( \bigcup_{i=1}^u Q_i \right)\right) \,.
\end{align}
The set $\brid$ inherits a graph structure from $C$ and the $C_i$. Let us make an observation.

\begin{lem} In the configuration $\om'$, the vertex sets of $\brid$ and $G$ are disjoint, and all edges of  $\pa \brid$ are closed, except those joining a vertex of $\brid$ and a vertex in some $Q_i$ or those joining a vertex of $\brid$ with a vertex in $Q$.

\label{bridge_disjoint}
\end{lem}

\begin{proof} 
To show $\brid$ and $G$ are disjoint, it suffices to show $C_i \cap G = \emptyset$ for each $i$ and that $C \cap G = \emptyset$. If the intersection of $C$ and $G$ were non-empty, there would then be an open path beginning from a vertex of $G$ and ending at a vertex contained in $\un{Q}$, which is impossible in the configuration $\om'$, as it would imply $\pa G$ is not a cutset separating the vertices of $G$ from $\infty$. The same reasoning shows that $C_i \cap G = \emptyset$ for each $i$. 

As $C$ and the $C_i$ are unions of open clusters, it is impossible that $\pa C$ or any $\pa C_i$ contain open edges. Due to the construction of $\brid$, the only open edges present in $\pa \brid$ either join $\brid$ with some vertex in the ocean $\un{Q}$, or they join $\brid$ with a vertex in $\un{Q}(i)$ for some $i$. \end{proof}

The coarse-grained image of $\brid$ is the last ingredient in a cube set which will nearly contain a closed cutset. Define
\begin{align}
\un{\brid} := \Big\{ \un{B}(x) : \un{B}(x) \text{ contains a vertex of $\brid$} \Big\} \,,
\end{align}
and then define the cube set $\un{\Gamma}$:
\begin{align}
\un{\Gamma} := \Delta \un{Q} \cup \un{\brid} \cup \left( \bigcup_{i=1}^u \Delta \un{Q}(i) \right) \,.
\end{align}

\begin{figure}[h]
\centering
\includegraphics[scale=1]{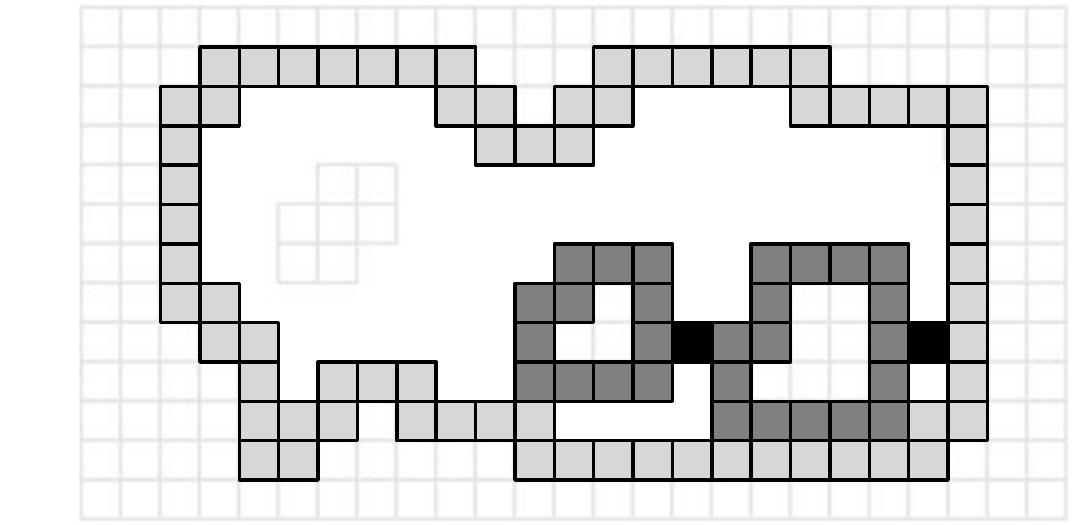}
\caption{This image is built from Figure \ref{fig:zhang_2}. We have removed $\pa_oG$ and $\brid$ from the diagram for clarity. The light-grey cubes depict $\Delta \un{Q}$, the dark-grey cubes depict the two $\Delta \un{Q}(i)$ and the black cubes depict $\un{\brid}$. The cubes adjacent to the black cubes are also in $\un{\brid}$, so $\un{\brid}$ is not necessarily disjoint from the boundary of the ponds and ocean.} 
\label{fig:zhang_3}
\end{figure}

\begin{rmk} The cube set $\un{\Gamma}$ insulates $G$ from $\infty$, and is depicted in Figure \ref{fig:zhang_3}. The boundaries of the ocean and live / almost-live ponds cover the relevant parts of $\pa_o G$, and $\un{\brid}$ connects these boundaries, allowing $\un{\Gamma}$ to be $*$-connected. 
\end{rmk}

We introduce one more piece of notation to state the central result of this section. For a $k$-cube $\un{B}(x)$, define the corresponding \emph{augmented} cube as follows:
\begin{align}
\un{B}^+(x) := 2k x + [ -2k-1 , 2k +1 ]^d \,.
\end{align}

\begin{prop} In the configuration $\om'$, the augmented cube set 
\begin{align}
\un{\Gamma}^+ := \Big\{ \un{B}^+(x) : \un{B}(x) \in \un{\Gamma} \Big\}
\end{align}
 contains a closed cutset $\Gamma$ which separates $G$ from $\infty$.

\label{gamma_cut}
\end{prop}

\begin{proof} 
Let $\gamma'$ be a path from $G$ to $\infty$. We show $\gamma'$ uses a closed edge contained in $\un{\Gamma}^+$. We lose no generality supposing $\gamma'$ is simple. As $\un{G}$ and each pond are finite sets, there is a first vertex $v_0 \in \un{Q}$ used by $\gamma'$. Consider the subpath of $\gamma'$ starting at the beginning of $\gamma'$ and ending at $v_0$. Name the reversal of this subpath $\gamma$, so that $\gamma$ is a path from $v_0$ to $G$. It will suffice to show $\gamma$ uses a closed edge which is contained in $\un{\Gamma}^+$. 

If the edge following $v_0$ in $\gamma$ is closed, we are happy, as this edge lies in $\Delta \un{Q}$. Thus we may suppose that the edge following $v_0$ in $\gamma$ is open, so that $\gamma$ joins $\brid$. The path $\gamma$ must connect with $G$. As $\brid$ and $G$ are disjoint (by Lemma \ref{bridge_disjoint}), and because $\gamma$ eventually uses a vertex of $G$, $\gamma$ eventually leaves $\brid$. If $\gamma$ leaves $\brid$ through a closed edge, this edge lies in one of the augmented cubes corresponding to the set $\un{\brid}$, and the proposition holds. 

We may then suppose $\gamma$ first leaves $\brid$ through an open edge. By Lemma \ref{bridge_disjoint}, and because $\gamma$ cannot return to $\un{Q}$, $\gamma$ must pass into some $Q_i$. As $\gamma$ is simple and the $Q_i$ are disjoint from $G$, there is a vertex $v_1$, last among all $Q_i$ used by $\gamma$. Let $\gamma_1$ denote the subpath of $\gamma$ obtained by starting from $v_1$. If the first edge of $\gamma_1$ is closed, our claim holds as this edge lies in $\Delta \un{Q}(i)$ for some $i$. Thus we may suppose the first edge of $\gamma_1$ is open, so that $\gamma_1$ rejoins $\brid$. However, $\gamma_1$ can no longer exit $\brid$ through a pond or through the ocean, and thus $\gamma_1$ must exit $\brid$ through a closed edge.

This establishes that $\un{\Gamma}^+$ contains a closed cutset separating $G$ from $\infty$. Via some deterministic method, choose a minimal cutset within $\un{\Gamma}^+$ and label it $\Gamma$.
\end{proof}

\subsection{Properties of the discovered cutset}\label{sec:properties} We now derive properties of the cutset $\Gamma$ and the cube set $\un{\Gamma}$. We begin by showing $\un{\Gamma}$ is contained in the coarse grained image of $\pa_o G$.

\begin{lem} The $k$-cube set $\un{\Gamma}$ is contained in $\un{A}$.
\label{sub_A}
\end{lem}

\begin{proof} Suppose $\un{B}(x) \in \un{\brid}$. We claim $\un{B}(x)$ either contains a vertex of $G$ or the endpoint vertex of an edge in $\pa_o G$, that is, $\un{B}(x) \in \un{G}$. If not, either $\un{B}(x) \in \un{Q}'(i)$ or $\un{B}(x) \in \un{Q}$. As $\un{B}(x) \in \un{\brid}$, there is $y \in \un{B}(x)$ lying in an open cluster connected to a live or almost-live pond. Thus, if $\un{B}(x)$ is a member of a pond, this pond is live or almost-live. Then either $y \in Q_i$ for some $i$ or $y \in Q$. Both are impossible, as we cut out such vertices in the construction of $\brid$, and we conclude $\un{\brid} \subset \un{G}$. 

Continue to suppose $\un{B}(x) \in \un{\brid}$. As $\un{B}(x)$ contains $y \in \brid \subset C \cup \bigcup_i C_i$, and as this union is disjoint from $G$, any path $\gamma$ from $y$ to $G$ within $\un{B}(x)$ uses an edge in $\pa G$. But any $y \in C \cup \bigcup_{i} C_i$ is connected to $\infty$ via a path disjoint from $G$. Thus the path $\gamma$ from $y$ to $G$ in $\un{B}(x)$ must actually use an edge of $\pa_oG$, and $\un{\brid} \subset \un{A}$. 

We now show $\Delta \un{Q}(i) \subset \un{A}$ for each $i$. Let $\un{B}(x) \in \Delta \un{Q}(i)$, so that $\un{B}(x)$ is $*$-adjacent to a cube $\un{B}(x') \in \un{Q}(i)$. Then $\un{B}(x)$ either contains  a vertex of $G$ or an endpoint vertex of an edge in $\pa_o G$, otherwise $\un{B}(x)$ would be a member of $\un{Q}(i)$. If $\un{B}(x)$ contains an endpoint vertex of an edge in $\pa_0 G$, we are done, thus we suppose $\un{B}(x)$ contains a vertex $y$ of $G$.

Note that $\un{B}(x)$ and $\un{B}(x')$ have at least one vertex $z$ in common, and $z$ (by virtue of lying within some $\un{Q}(i)$) is connected to $\infty$ in via a $\Z^d$-path disjoint from $G$. Any path joining $y$ and $z$ in $\un{B}(x)$ then intersects $\pa_oG$, showing $\Delta \un{Q}(i) \subset \un{A}$. An identical argument shows $\Delta \un{Q} \subset \un{A}$. \end{proof}

We now apply analogues of Proposition~\ref{star_conn} to $\Delta \un{Q}(i)$ and $\Delta \un{Q}$ to establish the following essential result. 

\begin{lem} The $k$-cube set $\un{\Gamma}$ is $*$-connected. 
\label{un_gamma_star}
\end{lem}

\begin{proof} It follows directly from Lemma 2 of Tim\'{a}r \cite{Timar} that $\Delta\un{Q}$ and each $\Delta \un{Q}(i)$ are $*$-connected cube sets. Let $D$ be a connected component of $\brid$, and let $\un{D}$ be the collection of $k$-cubes containing a vertex of $D$, so that $\un{D} \subset \un{\brid}$. It follows from the construction of $\brid$ that $\un{D}$ either intersects $\Delta \un{Q}(i)$ for some $i$, or $\un{D}$ intersects $\un{Q}$. As $D$ is connected in $\Z^d$, it is immediate that coarse grained image $\un{D}$ is $*$-connected. The set $\un{\brid}$ is itself the union of all such cube sets $\un{D}$, and it follows from the defining properties of live and almost-live ponds that $\un{\Gamma}$ is $*$-connected. \end{proof}

We finish the section by showing that, in the configuration $\om'$, each cube in $\un{\Gamma}$ has one of two rare geometric properties, defined below. Each $k$-cube $\un{B}(x)$ has $2d$ faces $\sig_1(x), \dots, \sig_{2d}(x)$, each of which is a $(d-1)$-dimensional square of side-length $2k$. A \emph{surface} of $\un{B}(x)$ is a vertex set of the form $\sig_i(x) \cap \Z^d$; each $k$-cube $\un{B}(x)$ possesses $2d$ distinct surfaces. In the context of a $3k$-cube $\un{B}_3(x)$, a surface is any surface of any of the $k$-cubes $\un{B}(x') \subset \un{B}_3(x)$. 

\begin{defn} A $k$-cube $\un{B}(x)$ is \emph{Type-I} if there is an open path $\gamma$ and a surface $\sig \cap \Z^d$ in $\un{B}_3(x)$ so that $\gamma$ joins a vertex in $\un{B}^+(x)$ to a vertex of $\pa \un{B}_3(x) \cap \Z^d$, with no vertex along $\gamma$ joined via another open path to $\sig \cap \Z^d$. We require all paths in this definition to use only edges which are \emph{internal} to $\un{B}_3(x)$, that is, no edge in any path has both endpoints in $\pa \un{B}_3(x)$.
\end{defn}

\begin{defn} A $k$-cube $\un{B}(x)$ is \emph{Type-II} if there are disjoint open paths $\gamma_1$ and $\gamma_2$, each joining vertices of $\un{B}^+(x)$ to vertices in $\pa \un{B}_3(x)$, with no open path in $\un{B}_3(x)$ from any vertex in $\gamma_1$ to a vertex in $\gamma_2$. All paths in this definition are required to use only edges internal to $\un{B}_3(x)$.
\label{def:typeII}
\end{defn}

Figure \ref{fig:zhang_4} illustrates these geometric properties. Because of the requirement that all paths in the above definitions are internal, the event that a $k$-cube $\un{B}(x)$ is Type-I or Type-II does not depend on the state of any edge contained in $\pa \un{B}_3(x)$.

\begin{figure}[h]
\centering
\includegraphics[scale=1]{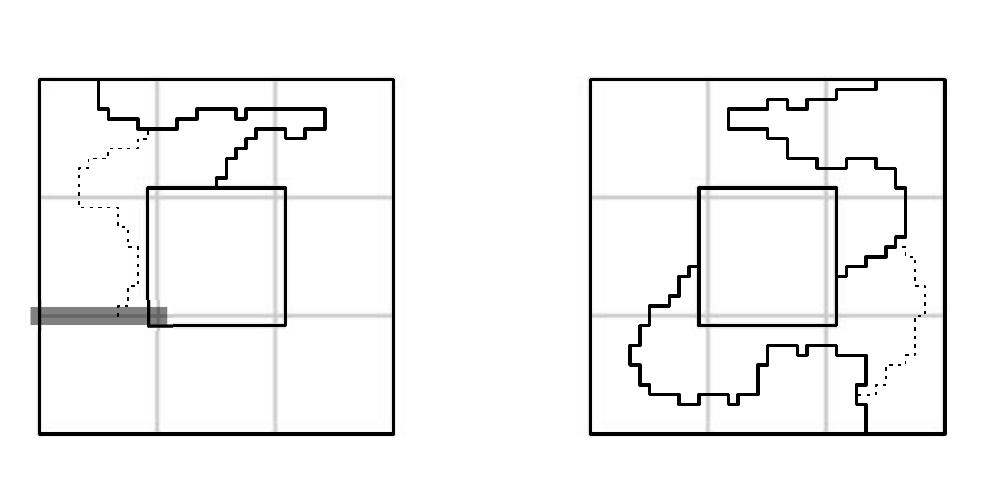}
\caption{On the left is an illustration of what \emph{cannot} happen in a Type-I cube. The dotted line is an open path joining the solid line (also an open path) to one of the surfaces of the $3k$-cube. Likewise, on the right is an illustration of what \emph{cannot} happen in a Type-II cube.} 
\label{fig:zhang_4}
\end{figure}

\begin{prop} Suppose $G$ is not contained within any $3k$-cube. Then, in the configuration $\om'$, each $k$-cube of $\un{\Gamma}$ is either Type-I or Type-II.

\label{bad_cubes}
\end{prop}

\begin{proof} Following Zhang, consider two cases. First suppose $\un{B}(x) \in \un{\Gamma}$ is a member of
\begin{align}
\Delta \un{Q} \cup \left( \bigcup_{i=1}^u \Delta \un{Q}(i) \right) \,.
\end{align}
Such a $\un{B}(x)$ is $*$-adjacent to a cube $\un{B}(x')$ which neither intersects $G$ nor an endpoint vertex of $\pa_o G$. By Lemma \ref{sub_A}, $\un{B}(x) \in \un{A}$, so that $\un{B}(x)$ contains an endpoint vertex of $\pa_o G$. Thus, $\un{B}^+(x)$ contains a vertex $y \in G$. There can be no open path from any surface of $\un{B}(x')$ to $y$: such a path could not use an edge of $\pa_oG$ but could be extended to a path from $y$ to $\infty$ using no other vertices of $G$. On the other hand, as $G$ is not contained in any $3k$-cube, there must be an open path from $y$ to a vertex of $\pa \un{B}_3(x)$. We may arrange this open path uses edges internal to $\un{B}_3(x)$ by stopping it at the first vertex of $\pa \un{B}_3(x)$ it meets. Thus in the first case, $\un{B}(x)$ is Type-I.

In the second case, suppose $\un{B}(x)$ is a member of 
\begin{align} 
\un{\brid} \setminus \left(  \Delta \un{Q} \cup \left( \bigcup_{i=1}^u \Delta \un{Q}(i) \right) \right) \,,
\end{align}
and let $y \in \un{B}(x) \cap \brid$. Then $y$ lies in some connected component $D$ of either $C$ or one of the $C_i$. The component $D$ cannot be contained in $\un{B}_3(x)$, otherwise one of the $k$-cubes $*$-adjacent to $\un{B}(x)$ would be a member of either $\un{Q}$ or some $\un{Q}(i)$. This is impossible as it would imply $\un{B}(x) \in \Delta \un{Q}$ or $\un{B}(x) \in  \Delta \un{Q}(i)$ for some $i$. It follows that $y$ is joined to the the boundary of $\un{B}_3(x)$ by an open internal path (contained in $D$). 

On the other hand, thanks to Lemma \ref{sub_A}, the cube $\un{B}(x)$ contains an endpoint vertex of an edge in $\pa_o G$, and $\un{B}^+(x)$ contains a vertex $z \in G$. As $G$ is not contained in any  $3k$-cube, the vertex $z$ is connected to the boundary of $\un{B}_3(x)$ by an open internal path (in $G$). But $D$ and $G$ are disjoint, thus the corresponding paths from $y$ and $z$ to the boundary of $\un{B}_3(x)$ cannot lie in the same open cluster. We conclude that in this second case, $\un{B}(x)$ is Type-II.  \end{proof}

We close the section by citing that it is rare for a cube to be Type-I or Type-II when $k$ is large.

\begin{prop} (Grimmet \cite{Grimmett}, Lemma 7.89 and Zhang \cite{Zhang}, Section 3)\, Let $d\geq 2$ and let $p > p_c(d)$. There are positive constants $c_1(p,d)$ and $c_2(p,d)$ so that for each $k$-cube $\un{B}(x)$, 
\begin{align}
\prob_p ( \un{B}(x) \text{ is Type-I or Type-II} ) \leq c_1 \exp(- c_2 k ) \,.
\end{align}
\label{rare_type}
\end{prop}

{\large\section{\B{Coarse graining applied}}\label{sec:coarse_applied}}

The point of applying Zhang's construction to the $G_n$ is that, as we will see, the associated cutset $\Gamma_n$ has cardinality on the order of $n^{d-1}$ with high probability. The cutset $\Gamma_n$ may be thought of as a \emph{deformation} of the contour $\pa_o G_n$, and the vertices enclosed by this deformation form a subgraph of $F_n \subset \giant$ which we \emph{flatten} into a continuum set whose perimeter is bounded by $|\Gamma_n|$. The rescaled continuum set then has a bounded perimeter (not depending on $n$). This is how we pass from Cheeger optimizers to sets of finite perimeter. 

Our notion of closeness in Theorem \ref{main_L1} comes from the $\ell^1$-metric, thus $|G_n\, \Delta\, F_n|$ must be small. Figures \ref{fig:zhang_1} and \ref{fig:6.2_large} suggests the extra volume enclosed by $\Gamma_n$ may be substantial, for instance due to ponds trapped by $\Gamma_n$. These large enclosed components are surgically removed using another application of Zhang's construction. The boundary of $F_n$ is then more complicated than we first asserted: it is made of all contours created through Zhang's construction. We tie these contours together via an auxiliary edge set called the webbing and execute a Peierls argument to bound their total size. \newline

\subsection{Contours}\label{sec:webbing}

We formalize the argument sketched above. Given $G_n = G_n(\om) \in \cal{G}_n$, list the connected components of $G_n$ as $G_n^{(1)}, \dots, G_n^{(M)}$. For each $G_n^{(q)}$, let $\om_q'$ be the unique configuration obtained from $\om$ by closing each open edge in $\pa^\om G_n^{(q)}$. Within $\om_q'$, apply Zhang's construction, producing a closed cutset $\Gamma_n^{(q)}$ and $k$-cube set $\un{\Gamma}_n^{(q)}$ having the properties discussed in Section \ref{sec:properties}.  

\begin{rmk} In light of our notation for $3k$-cubes, we emphasize that each $\un{\Gamma}_n^{(q)}$ is a set of $k$-cubes. 
\end{rmk}

Let $\om'$ be the configuration built from $\om$ in which all edges of $\pa^\om G_n$ are closed. Given a collection of edges $S$, a connected component of $\Lambda$ of $\B{C}_\infty$ (in $\om'$, see Remark \ref{rmk:6_cluster}) is \emph{surrounded} by $S$ if every path from $\Lambda$ to $\infty$ uses an edge of $S$. Only the connected components of $\B{C}_\infty$ surrounded by $S$ matter, not other open clusters. Define
\begin{align}
\e(d) := 1 - \frac{d}{(d-1)^2} \,,
\label{eq:epsilon_d}
\end{align}
and observe $\e(d)$ is positive when $d \geq 3$. Within $\om'$, the cutsets $\Gamma_n^{(q)}$ may surround other connected components of $\B{C}_\infty$ aside from the $G_n^{(q)}$. If $\Lambda$ is such a component, say $\Lambda$ is \emph{large} if $|\Lambda | \geq n^{1-\e(d)}$, and say $\Lambda$ is \emph{small} otherwise. Enumerate the large components $L_1, \dots, L_m$ of $\B{C}_\infty$ surrounded by any of the cutsets $\Gamma_n^{(q)}$, but do not include any of the $G_n^{(q)}$ in this list. Likewise enumerate the small components $S_1, \dots, S_t$ of $\B{C}_\infty$ surrounded by any of the cutsets $\Gamma_n^{(q)}$. Our notation suppresses the dependence of the $L_i$ and the $S_j$ on $n, \om$ and $G_n$. Figure \ref{fig:6.2_large} depicts how large and small components may arise. 

\begin{figure}[h]
\centering
\includegraphics[scale=1]{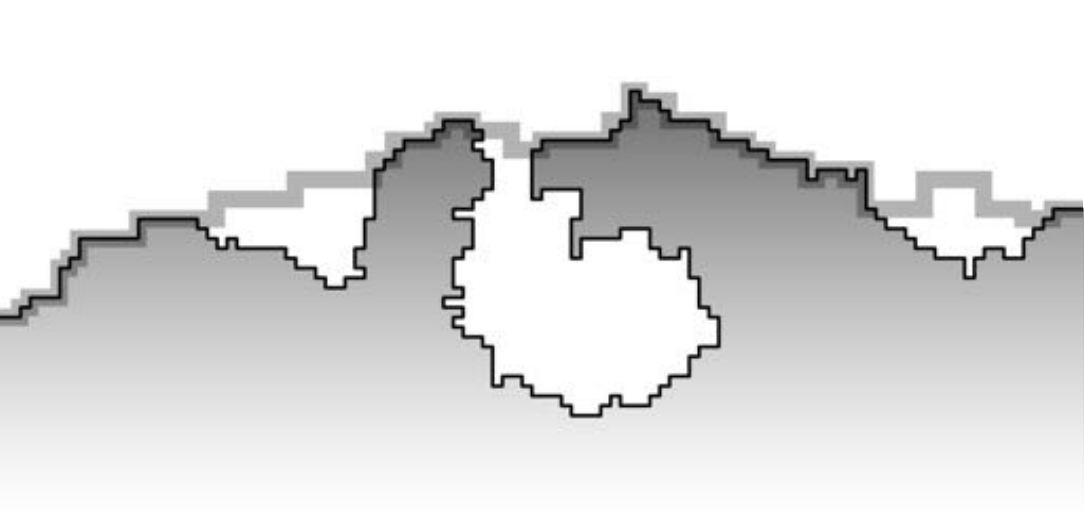}
\caption{The black contour is a close-up of the boundary of some $G_n^{(q)}$. The thicker grey contour is the associated cutset $\Gamma_n^{(q)}$. It is possible that connected components of $\B{C}_\infty$ are bounded between these two contours (see Remark \ref{rmk:6_cluster}).} 
\label{fig:6.2_large}
\end{figure}

Define $F_n \subset \B{C}_\infty$ from $G_n$ by \emph{filling in} the small components.
\begin{align}
F_n : = G_n \cup \left( \bigcup_{j=1}^t S_j \right) \,.
\label{eq:6.2_F}
\end{align}

We build cutsets surrounding the large components to fashion the boundary of $F_n$. Given a large component $L_i$, let $\om_i'$ denote the configuration in which all open edges of $\pa^\om L_i$ are closed. Within $\om_i'$, apply Zhang's construction, producing a closed cutset $\wh{\Gamma}_n^{(i)}$ separating $L_i$ from $\infty$, and a corresponding $k$-cube set $\un{\wh{\Gamma}}_n^{(i)}$. The edge sets $\Gamma_n^{(q)}$ and $\wh{\Gamma}_n^{(i)}$ together represent the boundary of $F_n$, and we define:
\begin{align}
\Gamma_n := \left(\bigcup_{q =1}^M \Gamma_n^{(q)}\right) \cup \left(   \bigcup_{i=1}^m \wh{\Gamma}_n^{(i)} \right)\,, \hspace{5mm} \un{\Gamma}_n := \left(\bigcup_{q =1}^M \un{\Gamma}_n^{(q)}\right) \cup \left(   \bigcup_{i=1}^m \wh{\un{\Gamma}}_n^{(i)} \right)\,.
\label{eq:6.2_gamma}
\end{align}

\begin{figure}[h]
\centering
\includegraphics[scale=1]{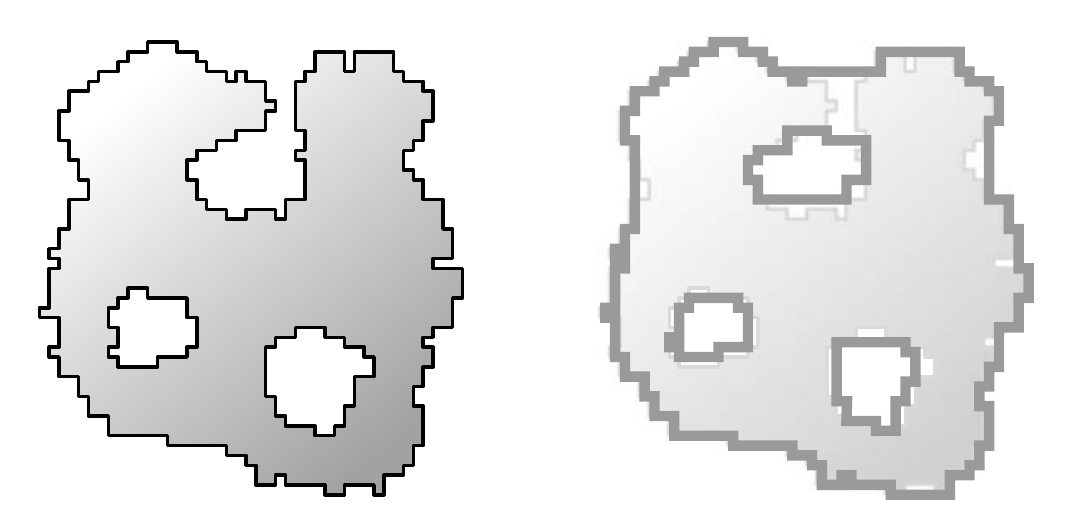}
\caption{On the left is $G_n \in \cal{G}_n$. On the right, the thick grey contours together form the edge set $\Gamma_n$. The inner contours arise from large components and are of the form $\wh{\Gamma}_n^{(i)}$. The outer contour corresponds to $G_n$ itself. It is natural to wonder how these contours \emph{interact}; we address this question at the start of Section \ref{sec:contiguity}. } 
\label{fig:6.2_contours}
\end{figure}

\subsection{Webbing}\label{sec:webbing} We now tie the contours together using an edge set called the webbing. We think of the webbing as a one-dimensional object because it is made up of paths; as the number of large components grows with $n$, the size of the webbing becomes too large in dimension two, which is why our argument works only in dimensions three and higher. 

By, for instance, fixing an ordering of finite subsets of $\Z^d$, choose in a unique way an endpoint vertex $\zeta_q$ of an edge in $\Gamma_n^{(q)}$. Do the same for each $\wh{\Gamma}_n^{(i)}$, calling these vertices $z_i$. These are the endpoints of paths making up the webbing. Let $\al > 0$ be a parameter to be chosen later. Consider all cubes of the form $\lceil n^\al \rceil x + [ - \lceil n^\al \rceil , \lceil n^\al\rceil ]^d$ intersecting $[-2n,2n]^d$, where $x \in \Z^d$. List these cubes as $\{ \ov{B}_j \}_{j=1}^\ell$, ordered so that consecutive cubes share a face. For $n$ sufficiently large (depending on the renormalization parameter $k$), all $\zeta_q$ and $z_i$ lie in the union of the $\ov{B}_j$. For each $j \in \{1, \dots, \ell\}$, let $m_j$ be the number of $z_i$ contained in the cube $\ov{B}_j$. 

In each $\ov{B}_j$, do the following: begin with $z_i \in \ov{B}_j$ least in our ordering of finite subsets of $\Z^d$. Pick (in some unique way) a not necessarily open $\Z^d$-path in $\ov{B}_j$ joining this ``smallest" vertex to the next smallest $z_k$, arranging for the path to use the fewest edges possible. Continue building paths between the current vertex in $\ov{B}_j$ and the next smallest within $\ov{B}_j$ until all $z_i$ in $\ov{B}_j$ are used. The union of all paths created this way is a graph $T_j$ called a \emph{tangle}. Define $T_j$ to be empty if $m_j =0$. 

We link each tangle using a single long path. Begin with $\ov{B}_1$ and select $z_i \in \ov{B}_1$ minimal in our ordering, then select a minimal vertex from $\ov{B}_2$ and connect the two vertices by a uniquely chosen shortest path in $\Z^d$. Do nothing when these vertices are identical, and if at any point a cube $\ov{B}_j$ contains no $z_i$, take instead the vertex of $\Z^d \cap \ov{B}_j$ minimal in our ordering. Continue joining consecutive vertices until the minimal vertex in the last $\lceil n^\al \rceil$-cube has been used. Link this last vertex to $\zeta_1$, and then to successive $\zeta_q$'s by uniquely chosen shortest paths until we reach $\zeta_M$. The union of all paths created here is denoted $\str$. We have suppressed the $n$, $\om$ and $G_n$ dependence of both $\str$ and the $T_j$. The webbing is
\begin{align}
\web_n := \str\cup \left( \bigcup_{j=1}^\ell T_j \right) \,.
\end{align}
To bound the size of $\web_n$, we first bound the number of large components associated to each $G_n$. 

\begin{lem} Let $d\geq 3$ and $p > p_c(d)$. For each $G_n \in \cal{G}_n$, let $m(G_n,n,\om)$ be the number of large components surrounded by the $\Gamma_n^{(q)}$. Let $M_n(\om)$ be the maximum of $m(G_n)$ over all $G_n \in \cal{G}_n$. There are positive constants $c_1(p,d,k), c_2(p,d,k)$ and $c_3(p,d)$ so that 
\begin{align}
\prob_p \left(  M_n > c_3 n^{d-1 - 1/(d-1) } \right) \leq c_1 \exp \left( -c_2 n^{1/(d-1)} \right) \,.
\end{align}

\label{not_too_many_large}
\end{lem}

\begin{proof} Fix $G_n \in \cal{G}_n$ and specialize Corollary \ref{coro_BBHK} to the setting $d \geq 3$, $\al = 1 -\e(d)$: work within the high probability event $\{ R \leq n\}$. By  Lemma \ref{sub_A}, all large components $L_i$ corresponding to $G_n$ are contained in $\cal{C}_{2n}$. On $\{R \leq n\}$, 
\begin{align}
| \pa^\om L_i | \geq cn^{1/(d-1)} \,,
\label{eq:not_too_many_large}
\end{align}
for each $i \in \{1,\dots,m\}$. Also work in the high probability event (from Lemma \ref{surface_boundary}) that there is $\eta_3(p,d) > 0$ so that for all $G_n \in \cal{G}_n$, we have $|\pa^\om G_n| \leq \eta_3 n^{d-1}$. Use this upper bound with (\ref{eq:not_too_many_large}) and the fact that distinct large components have disjoint open edge boundaries to obtain 
\begin{align}
m \leq \frac{\eta_3}{c} n^{d-1-1/(d-1)} \,.
\end{align}
We use the estimates from Lemma \ref{surface_boundary} and Corollary \ref{coro_BBHK} to complete the proof. \end{proof}

We now deduce bounds on the number of edges in the $\web_n$.

\begin{rmk} In the following proof, we fix the parameter $\al$ (controlling the side-length of the cubes $\ov{B}_j$) to be $\tfrac{1}{(d-1)}$. 
\end{rmk}

\begin{prop} Let $d \geq 3$ and $p > p_c(d)$. Let $W_n = W_n(\om)$ be the maximum of $|\rmE (\web_n)|$ taken over all $G_n \in \cal{G}_n$. There are positive constants $c_1(p,d,k)$, $c_2(p,d,k)$ and $c_3(p,d)$ so that 
\begin{align}
\prob_p \left( W_n > c_3 n^{d-1} \right) \leq c_1 \exp \left( -c_2 n^{1/(d-1)} \right) \,.
\end{align}

\label{small_web}
\end{prop}

\begin{proof} Work in the high probability event from Lemma \ref{not_too_many_large} that the maximum number $M$ of large components $L_i$ across all $G_n \in \cal{G}_n$ is at most $c n^{d-1 -1/(d-1)}$. Also work in the high probability event from Corollary \ref{conn_finite} that the number of connected components of any $G_n \in \cal{G}_n$ is at most $\eta_4 > 0$.

Fix $G_n \in \cal{G}_n$ for the rest of the proof. Consider the tangle $T_j$ for $G_n$ associated to the $\lceil n^\al \rceil$-cube $\ov{B}_j$. Based on our construction of each tangle, the number of edges $| \text{\rm E}(T_j)|$ is at most the $\ell^1$-diameter of $\ov{B}_j$ times the number of $z_i$ within $\ov{B}_j$. Choosing constants below appropriately, 
\begin{align}
\sum_{j=1}^\ell | \text{\rm E}(T_j)| &\leq 8d n^\al \sum_{j=1}^\ell m_j \,,\\
&\leq c(d) mn^\al \,, \\
&\leq c(p,d) n^{\al + n-1 - 1/(d-1) } \,,
\end{align}
where to obtain second line directly above, we use that each vertex $z_i$ lies in at most $2d$ distinct $\lceil n^\al \rceil$-cubes, and to obtain the third line we used Lemma \ref{not_too_many_large}. It remains to bound $|\rmE(\str)|$. A shortest $\Z^d$-path between the vertices of two adjacent $\lceil n^\al \rceil$-cubes uses at most $16 n^\al$ edges, and there are at most $c(d) n^{d(1-\al)}$ cubes in total. The final paths in the construction of $\str$ which join the vertices $\zeta_q$ each use at most $c(d)n$ edges. Thus,
\begin{align}
| \text{\rm E}(\str) | \leq c(d) n^\al n^{d(1-\al)} + \eta_4 c(d) n \,,
\end{align}
so that upon choosing $\al := 1/(d-1)$,
\begin{align}
|\text{\rm E}(\web_n) | &\leq c(p,d) \left[ n^{\al + (d-1) - 1/(d-1) } + n^{\al + d(1-\al) } +n \right]  \,,\\
&\leq c(p,d) n^{d-1} \,.
\end{align}
Use the estimates from Lemma \ref{not_too_many_large} and from Corollary \ref{conn_finite} to complete the proof. \end{proof}

We are finished working with $\lceil n^\al \rceil$-cubes. Define the coarse-grained image of each $\web_n$ as 
\begin{align}
\un{\web}_n := \Big\{ \un{B}(x) : \un{B}(x) \cap \web_n \neq \emptyset \Big\}
\end{align}
so that each $\un{\web}_n$ is a collection of $k$-cubes depending on $n,\om$ and $G_n$. The last lemma of this subsection follows directly from the construction of $\web_n$ and from Lemma \ref{un_gamma_star}:

\begin{lem} For each $G_n \in \cal{G}_n$, the $k$-cube set $\un{\Gamma}_n \cup \un{\web}_n$ corresponding to $G_n$ is $*$-connected.
\label{web_conn}
\end{lem}

\subsection{A Peierls argument}\label{sec:Peierls}

We use a Peierls argument to show that when the renormalization parameter $k$ is taken large, each $|\Gamma_n|$ is with high probability on the order of $n^{d-1}$. We make a small observation first.

\begin{lem} For each $G_n \in \cal{G}_n$, any edge of $\Z^d$ is contained in at most $(11k)^d$ distinct edge sets among the $\Gamma_n^{(q)}$ and the $\wh{\Gamma}_n^{(i)}$ corresponding to $G_n$.
\label{i_love_coffee}
\end{lem}

\begin{proof} Fix $G_n \in \cal{G}_n$. If $\Gamma_n^{(q)}$ uses $e \in \rmE(\Z^d)$, there is $\un{B}(x) \in \un{\Gamma}_n^{(q)}$ so that $e \in \un{B}^+(x)$. By Lemma \ref{sub_A}, $\un{B}^+$ also contains a vertex $y \in G_n^{(q)}$. Suppose another cutset, say $\wh{\Gamma}_n^{(i)}$, also uses $e$. Identical reasoning tells us the five-fold dilate of $\un{B}(x)$,
\begin{align}
\un{B}_5(x) : = 2k x + [-5k,5k]^d
\end{align}
contains both $y$ and a vertex $z \in L_i$. If cutsets corresponding to other connected components of $G_n$ or other $L_i$ also use $e$, at least one vertex in each of these graphs must also lie in $\un{B}_5(x)$. As the components of $G_n$ and the $L_i$ are all disjoint, and because $\un{B}_5(x)$ contains at most $(11k)^d$ vertices, the claim holds.\end{proof}

In the Peierls argument below, the renormalization parameter $k$ is fixed once and for all. 

\begin{prop} Let $d\geq 3$ and $p > p_c(d)$. There is $\gamma= \gamma(p,d) >0$ and positive constants $c_1(p,d)$ and $c_2(p,d)$ so that 
\begin{align}
 \prob_p\left ( \max_{G_n \in \cal{G}_n} | \Gamma_n | \geq \gamma n^{d-1} \right) \leq c_1 \exp \left(- c_2 n^{1/(d-1)} \right) \,.
\end{align}
\label{finite_per}
\end{prop}

\begin{proof} Let $\cal{E}_{\textsf{web}}$ be the event from Proposition \ref{small_web} that for all $G_n \in \cal{G}_n$, the corresponding graphs $\web_n$ satisfy $| \text{\rm E} (\web_n) | \leq c(p,d)n^{d-1}$. We work with a fixed $G_n \in \cal{G}_n$ and corresponding $\Gamma_n$ throughout the proof, and begin by using the bounds in Proposition \ref{small_web} and in Lemma \ref{surface_boundary}:
\begin{align}
\prob_p &\Big( | \Gamma_n | \geq \gamma n^{d-1} \Big) \leq c_1 \exp \Big(-c_2 n^{1/(d-1)} \Big)  + \prob_p \Big( \Big\{ | \Gamma_n| \geq \gamma n^{d-1} \Big\} \cap \Big\{ | \pa^\om G_n | \leq \eta_3 n^{d-1} \Big\} \cap \cal{E}_{\textsf{web}} \Big) \,,\\
&\leq c_1 \exp \Big(-c_2 n^{1/(d-1)}\Big) + \sum_{j = \gamma n^{d-1}}^\infty  \prob_p \Big( \Big\{ | \Gamma_n| =j \Big\} \cap \Big\{ | \pa^\om G_n | \leq (\eta_3/ \gamma) j \Big\} \cap \cal{E}_{\textsf{web}}  \Big) \,.
\end{align}
Take $\gamma$ large depending on $k,p$ and $d$ so that $\eta_3 / \gamma < [ 2 \cdot 4^d (11k)^d(4k)^{d+1} ]^{-1}$: 
\begin{align}
\label{eq:finite_per_1}
\prob_p\Big( &| \Gamma_n | \geq \gamma n^{d-1} \Big) \leq c_1 \exp \Big(-c_2 n^{1/(d-1)} \Big)\\ 
&+  \sum_{j = \gamma n^{d-1}}^\infty  \prob_p \Big( \Big\{ | \Gamma_n| =j \Big\} \cap \Big\{ | \pa^\om G_n | < [2 \cdot 4^d (11k)^d (4k)^{d+1} ]^{-1} j \Big\} \cap \cal{E}_{\textsf{web}}   \Big) \,.
\end{align}
Equip $\un{\Gamma}_n$ with a graph structure: the vertices are the $k$-cubes in $\un{\Gamma}_n$, and an edge exists between two vertices $\un{B}(x)$ and $\un{B}(y)$ if $x \sim_{*} y$. The maximum degree of any vertex in \emph{this} graph is $3^d$, so by Theorem \ref{turan}, there is a subcollection of cubes $\un{\Gamma}_n' \subset \un{\Gamma}_n$ so that  $| \un{\Gamma}_n' | \geq |\un{\Gamma}_n| / 4^d$, and whenever $\un{B}(x),\un{B}(y) \in \un{\Gamma}_n'$, the corresponding $3k$-cubes $\un{B}_3(x)$ and $\un{B}_3(y)$ have disjoint interiors. Thus,
\begin{align}
\label{eq:finite_per_2}
| \un{\Gamma}_n' | \geq \frac{| \un{\Gamma}_n|}{4^d} \geq \frac{ |\Gamma_n| }{ 4^d (11k)^d (4k)^{d+1} } \,,
\end{align}
as when $k \geq d$, there are at most $(4k)^{d+1}$ edges of $\Z^d$ having an endpoint in a given augmented $k$-cube, and Lemma \ref{i_love_coffee} implies each such edge lies in at most $(11k)^d$ distinct cutsets among the $\Gamma_n^{(q)}$ and $\wh{\Gamma}_n^{(i)}$. Consider the following event:
\begin{align}
\label{eq:finite_per_3}
\Big\{ | \Gamma_n| =j \Big\} \cap \Big\{| \pa^\om G_n | < [2 \cdot 4^d (11k)^d (4k)^{d+1} ]^{-1} j \Big\} \cap \cal{E}_{\textsf{web}} \,.
\end{align}
Within (\ref{eq:finite_per_3}), we know from (\ref{eq:finite_per_2}) that at most half of the cubes in $\un{\Gamma}_n'$ may contain an edge of $\pa^\om G_n$. Thus there is a further subcollection $\un{\Gamma}_n'' \subset \un{\Gamma}_n'$ so that
\begin{align}
| \un{\Gamma}_n''| \geq \frac{|\Gamma_n|}{ 2 \cdot 4^d (11k)^d(4k)^{d+1} } \,,
\end{align}
such that each cube of $\un{\Gamma}_n''$ is either Type-I or Type-II by Proposition \ref{bad_cubes}.

Of course, $\un{\Gamma}_n''$ inherits from $\un{\Gamma}_n'$ the property that any two $\un{B}(x), \un{B}(y) \in \un{\Gamma}_n''$ are such that $\un{B}_3(x)$ and $\un{B}_3(y)$ have disjoint interiors. Thus, for distinct $\un{B}(x)$ and $\un{B}(y)$ in $\un{\Gamma}_n''$, the event that $\un{B}(x)$ is Type-I or Type-II is independent from the event that $\un{B}(y)$ is Type-I or Type-II. 

Continue to work within \eqref{eq:finite_per_3}. Write $s := |\un{\Gamma}_n|$; on $\cal{E}_{\textsf{web}}$, we have $|\un{\web}_n| \leq c(p,d) n^{d-1}$. By Proposition \ref{web_conn}, $\un{\Gamma}_n \cup \un{\web}_n$ is $*$-connected, so Proposition \ref{span_tree} implies there are at most
\begin{align}
(3n)^d [c(d)]^{s + cn^{d-1} }
\end{align}
distinct possibilities for the $k$-cube set $\un{\Gamma}_n \cup \un{\web}_n$. The factor of $(3n)^d$ is a crude upper bound on the number of vertices in $\Z^d \cap [-n,n]^d$. There are at most $2^{s + cn^{d-1}}$ ways to choose $\un{\Gamma}_n$ from $\un{\Gamma}_n \cup \un{\web}_n$, at most the same number of ways to choose $\un{\Gamma}_n'$ from $\un{\Gamma}_n$ and likewise for $\un{\Gamma}''$ from $\un{\Gamma}_n'$. We use a union bound and aforementioned independence to obtain
\begin{align}
\prob_p \Big( \Big\{ | \Gamma_n| =j \Big\} \cap \Big\{| \pa^\om G_n | < [2 \cdot 4^d (&11k)^d (4k)^{d+1} ]^{-1} j \Big\} \cap \cal{E}_{\textsf{web}} \Big) \\
&\leq (3n)^d [c(d)]^{s + cn^{d-1} } [ c_1 \exp(-c_2 k ) ]^{s / (2 \cdot 4^d) } \,.
\end{align}
Above, we've used the lower bound on $|\un{\Gamma}_n''|$ in terms of $s$ following from (\ref{eq:finite_per_2}) and Proposition \ref{rare_type}. From (\ref{eq:finite_per_2}) we have $s \geq j / (11k)^d(4k)^{d+1}$, and as $j \geq \gamma n^{d-1}$, we may take $\gamma$ larger if necessary, again in a way depending on $k,p$ and $d$, so that $s \geq cn^{d-1}$, giving
\begin{align}
\prob_p \Big( \Big\{ | \Gamma_n| =j \Big\} \cap \Big\{| \pa^\om G_n | < [2 \cdot 4^d (&11k)^d (4k)^{d+1} ]^{-1} j \Big\} \cap \cal{E}_{\textsf{web}} \Big) \\ &\leq (3n)^d [c(d)]^{2s } [ c_1 \exp(-c_2 k ) ]^{s / (2 \cdot 4^d) } 
\end{align}

Choose $k$ large enough depending on $p$ and $d$ so that $[c(d)]^2 [ c_1 \exp(-c_2 k ) ]^{1 / 2 \cdot 4^d}  < e^{-1}$, at which point $k$ is fixed. For this $k$, 
\begin{align}
\prob_p \Big( \Big\{ | \Gamma_n| =j \Big\} \cap  \Big\{ | \pa^\om G_n | < [2 \cdot 4^d\eta_4 (4k)^{d+1} ]^{-1} j \Big\} \cap \cal{E}_{\textsf{web}} \Big )\leq (3n)^{d} \exp ( -s) \,,
\end{align}
and we combine this bound with (\ref{eq:finite_per_1}) and ( \ref{eq:finite_per_2}) to obtain
\begin{align}
\prob_p \Big( | \Gamma_n | \geq \gamma n^{d-1} \Big) &\leq  c_1\exp \Big(-c_2 n^{1/(d-1)}\Big) + \sum_{j= \gamma n^{d-1}}^\infty (3n)^{d} \exp(-s) \,, \\
&\leq  c_1\exp \Big(-c_2 n^{1/(d-1)} \Big) + \sum_{j= \gamma n^{d-1}}^\infty (3n)^{d} \exp\left(-j /  [ (11k)^d (4k)^{d+1} ]\right) \,. \end{align}
We choose $\gamma$ sufficiently large depending on $p,d$ and $k = k(p,d)$ to complete the proof. \end{proof}

\begin{rmk}
In the proof of Proposition \ref{finite_per}, the renormalization parameter $k = k(p,d)$ has been fixed. All constructions given in Sections \ref{sec:coarse_original} and \ref{sec:coarse_applied} depended implicitly on this parameter. 
\end{rmk}

\subsection{Properties of $F_n$}

Having exhibited control on $|\Gamma_n|$, we now return to the $F_n$ defined in \eqref{eq:6.2_F}. We asserted that $\Gamma_n$ could be thought of as the boundary of $F_n$; the next proposition justifies this. By Lemma \ref{sub_A}, we can only conclude $F_n \subset [-n -2k, n +2k ]^d$ instead of $F_n \subset [-n,n]^d$. Recall the convention that for $E \subset \R^d$ and $S \subset \rmE(\Z^d)$, we say $E \cap S \neq \emptyset$ if $E$ contains an endpoint vertex of an edge in $S$. 

\begin{prop} Let $d\geq 3$ and $p > p_c(d)$. Define $\ell(n) := \lfloor n^{(1-\e(d))/2d}  \rfloor\,.$ There are positive constants $c_1(p,d)$ and $c_2(p,d)$ so that with probability at least
\begin{align} 
1- c_1 \exp \left(-c_2 n^{(1-\e(d))/2d} \right)  \,,
\end{align}
whenever $F_n$ corresponds to $G_n \in \cal{G}_n$, and whenever $B = [- \ell(n), \ell(n) ]^d +x \,,$ for some $x \in \Z^d$ satisfies $B \cap F_n \neq \emptyset$ and $B \cap F_n \neq B \cap \B{C}_\infty$, then either $B \cap \Gamma_n \neq \emptyset$, or else the three-fold dilate of $B$ around its center intersects $\pa^\om G_n$ non-trivially.
\label{gamma_boundary_of_F}
\end{prop}

\begin{proof} Fix $G_n \in \cal{G}_n$, and consider $F_n$ corresponding to $G_n$. Let $B$ be as in the statement of the proposition, so that $B \cap F_n \neq \emptyset$ and $B \cap F_n \neq B \cap \B{C}_\infty$. Write $B = [-  \ell(n) ,  \ell(n)]^d + x$ for $x \in \Z^d$, and define
\begin{align}
B_3 &:= [-  3\ell(n),  3\ell(n) ]^d + x \,.
\end{align}

Suppose $B$ contains $y \in \B{C}_\infty \setminus F_n$ connected to $\infty$ by a (not necessarily open) $\Z^d$-path $\gamma'$ using no edges of $\Gamma_n$. As $B \cap F_n \neq \emptyset$, $B$ contains some vertex $z$ which is either a member of some $G_n^{(q)}$ or some small component $S_j$. If $\gamma$ is a path from $z$ to $y$ in $B$, $\gamma$ must use an edge of $\Gamma_n$, otherwise $\Gamma_n$ would not surround some $G_n^{(q)}$ or some $S_j$.

Thus we may suppose every vertex $y \in ( \B{C}_\infty \setminus F_n ) \cap B$ is surrounded by one of the cutsets $\Gamma_n^{(q)}$ or $\wh{\Gamma}_n^{(i)}$. Any $y$ with this property lies in some large component $L_i$. Choose $z \in F_n \cap B$, and suppose $B_3 \cap \pa^\om G_n = \emptyset$, so there is no open path from $z$ to $y$ in $B_3$. 

Work within the high probability event from Lemma \ref{opt_macro} that for each $G_n \in \cal{G}_n$, every connected component of $G_n$ satisfies $| G_n^{(q)} | \geq \eta_1 n^d$. For all $n$ sufficiently large, the connected component of $F_n$ containing $z$ is not contained within $B_3$. Likewise, by the largeness of each $L_i$, and due to our choice of $\ell(n)$, no $L_i$ is contained in~$B_3$. 

Thus there is an open path from $z$ to $\pa_* (B_3 \cap \Z^d)$ and an open path from $y$ to $\pa_* (B_3 \cap \Z^d)$, and these paths are not joined by any open path in $B_3$. We have shown that, if $B$ does not contain an edge of $\Gamma_n$, and if $B_3 \cap \pa^\om G_n = \emptyset$, then $B_3$ has the Type-II property (see Definition \ref{def:typeII}). Use Proposition \ref{rare_type} with a union bound (taken over all such $B$ centered at $x \in \Z^d \cap [-n -2k, n +2k ]^d$) to see that with probability at most
\begin{align}
(3n)^d c_1 \exp \Big ( -c_2 n^{(1-\e(n))/ 2d} \Big) \,,
\end{align}
there is a cube $B$ of this form with the Type-II property. Combine this with the bounds of Lemma~\ref{opt_macro} to complete the proof. \end{proof}

To prepare for the contiguity argument in the next section, we demonstrate that when $F_n$ and $G_n$ are encoded as measures, they roughly agree on Borel sets. In Section \ref{sec:notation}, from each $G_n \in \cal{G}_n$ we built the empirical measure $\mu_n \in \cal{M}( [-1,1]^d)$ defined in \eqref{eq:2.3_empirical}. For each $F_n$ associated to $G_n \in \cal{G}_n$, define the \emph{empirical measure} $\wt{\mu}_n$ associated to $F_n$ similarly:
\begin{align}
\wt{\mu}_n := \frac{1}{n^d} \sum_{x\, \in\, \rmV(F_n)} \delta_{x /n } \,.
\label{eq:6.3_empirical}
\end{align}
Note that $\wt{\mu}_n$ is a signed Borel measure on $[-1-2k/n,1+2k/n]^d$. 

\begin{lem} Let $d\geq 3$ and let $p > p_c(d)$. There are positive constants $c_1(p,d)$, $c_2(p,d)$ and $\eta_3(p,d)$ so that  for each Borel $K \subset [-1-2k/n,1+2k/n]^d$,
\begin{align}
\prob \left( \max_{G_n \in \cal{G}_n} | \mu_n(K) - \wt{\mu}_n(K) | >  \eta_3 n^{-\e(d)} \right) \leq c_1 \exp \left(-c_2 n^{(d-1)/2d} \right) \,.
\end{align}
\label{adding_k_small}
\end{lem}

\begin{proof} Work within the high probability event from Lemma \ref{surface_boundary} $| \pa^\om G_n| \leq \eta_3 n^{d-1}$ for all $G_n \in \cal{G}_n$. For each small component $S_j$, the edge set $\pa^\om S_j$ intersects $\pa^\om G_n$. The $\pa^\om S_j$ are pairwise disjoint, thus the number $t$ of small components $S_j$ is at most $\eta_3 n^{d-1}$. From the definition of a small component, 
\begin{align} 
|F_n \setminus G_n | = \sum_{j=1}^t |S_j | \leq \eta_3 n^{d-\e(d)}
\end{align}
The lemma follows from the definitions of empirical measures for $G_n$ \eqref{eq:2.3_empirical} and $F_n$ \eqref{eq:6.3_empirical}. \end{proof}




{\large\section{\B{Contiguity}}\label{sec:contiguity}}

We now pass from each $F_n$ \eqref{eq:6.2_F} to a continuum object through another coarse graining procedure. The emperical measures $\wt{\mu}_n$ \eqref{eq:6.3_empirical} associated to each $F_n$ become \emph{flattened} or homogenized into measures representing sets of finite perimeter. \\

\begin{rmk} We rely heavily on the notation introduced in the previous section, in particular, we reference the large components $\{ L_i\}_{i=1}^m$ defined in Section \ref{sec:webbing} for a given $G_n \in \cal{G}_n$. We no longer use $k$-cubes, and the renormalization parameter $k$ of the last section will not come up except to say that the empirical measures $\wt{\mu}_n$ are elements of $\cal{M}( [-1 -2k/n, 1+ 2k/n]^d)$. The parameter $k$ has itself been fixed since the proof of Proposition \ref{finite_per}. 
\label{rmk:7_fixed}
\end{rmk}

Let us build the relevant continuum objects. Given a finite collection of edges $S$, define 
\begin{align}
\hull(S) : = \Big\{ x \in \Z^d : \text{any $\Z^d$-path from $x$ to $\infty$ must use an edge of $S$} \Big\} \,,
\end{align}
and recall that for $x \in \Z^d$, the unit dual cube $Q(x)$ is defined as $[-1/2,1/2]^d + x$. For fixed $G_n \in \cal{G}_n$, enumerate the connected components of $G_n$ as $ \{G_n^{(q)}\}_{q=1}^M$ and for each $q \in \{ 1, \dots, M\}$, define 
\begin{align}
A_q := \Big\{ i \in \{1, \dots, m\} : L_i \text{ is surrounded by } \Gamma_n^{(q)} \Big \} \,.
\end{align}
For each $q \in \{1, \dots, M\}$, define $H_n^{(q)}$ as the following \emph{collection of vertices}
\begin{align}
H_n^{(q)} := \hull \left( \Gamma_n^{(q)} \right) \setminus \left( \bigcup_{i\, \in\, A_q} \hull\left(\, \wh{\Gamma}_n^{(i)} \right) \right) \,,
\label{eq:7_hull}
\end{align}
and let $H_n = \bigcup_{q=1}^M H_n^{(q)}$. Define the polytope $P_n$ from $H_n$ via
\begin{align}
P_n = \left( \bigcup_{x\, \in\, H_n} n^{-1} Q(x) \right) \cap [-1,1]^d \,.
\label{eq:7_polytope}
\end{align}
Finally, form the measure $\nu_n = \nu_n(\om, G_n) \in \cal{M}( [-1,1]^d)$ representing $P_n$ in the sense of Section~\ref{sec:notation_3}: for $E \subset [-1,1]^d$ Borel, 
\begin{align}
\nu_n(E) := \theta_p(d) \cal{L}^d( E \cap P_n) \,,
\label{eq:7_nu}
\end{align}
The goal of this section is to show for each $G_n$, the measures $\mu_n$, $\wt{\mu}_n$ and $\nu_n$ are all close in the metric $\dw$ defined in \eqref{eq:2_metric}. 

\subsection{Contour and perimeter control} We thought of the edge sets $\Gamma_n$, defined in \eqref{eq:6.2_gamma}, as collections of contours, as in Figure \ref{fig:6.2_contours}. To prove $\dw$-closeness of $\wt{\mu}_n$ and $\nu_n$, we first examine how such contours interact with one another and rule out pathological configurations. 

For a fixed $G_n$, consider the large components $L_i$: each is surrounded by some cutset $\Gamma_n^{(q)}$. A large component $L_i$ is \emph{bad} if it is surrounded by $\Gamma_n^{(q)}$, with the corresponding connected component $G_n^{(q)}$ of $G_n$ surrounded by $\wh{\Gamma}_n^{(i)}$. If $L_i$ is bad, subtracting the hull of $\wh{\Gamma}_n^{(i)}$ in \eqref{eq:7_hull} from the hull of $\Gamma_n^{(q)}$ removes $G_n^{(q)}$ itself, and we cannot expect $\nu_n$ and $\wt{\mu}_n$ to be close. 

\begin{lem} 
For each $G_n \in \cal{G}_n$, it is impossible for any associated large component to be bad, in the sense just defined. 
\label{no_bad_large}
\end{lem}

\begin{proof} Fix $G_n \in \cal{G}_n$, and suppose $L_i$ is a bad large component associated to $G_n$. Recall that $\om_q'$ is the configuration obtained from $\om$ by closing each open edge in $\pa^\om G_n^{(q)}$. As $L_i$ is surrounded by $\Gamma_n^{(q)}$, it follows that $\pa_o L_i$ is a closed cutset separating $L_i$ from $\infty$ in $\om_q'$ (one must remember that the large components were defined as the connected components of the infinite cluster after passing to $\om_q'$). Back in $\om$, it follows that 
\begin{align}
\pa_o L_i \cap \pa^\om L_i \subset \pa^\om G_n^{(q)} \,.
\label{eq:7_no_bad}
\end{align}

Let us see how this gives rise to a contradiction: let $y \in G_n^{(q)}$. Working in the original configuration $\om$, consider a simple path $\gamma$ from $y$ to $\infty$ within $\B{C}_\infty$, so that $\gamma$ uses only open edges. We may assume that $y$ is the only vertex of $G_n^{(q)}$ used by $\gamma$, as $\gamma$ must eventually leave $G_n^{(q)}$ and not return. Because $G_n^{(q)}$ is surrounded by $\wh{\Gamma}_n^{(i)}$, $\gamma$ must use an open edge $e$ in $\wh{\Gamma}_n^{(i)}$. As $\wh{\Gamma}_n^{(i)}$ is a closed cutset in the configuration $\om_i'$, this open edge $e$ must lie in $\pa^\om L_i$. We have shown $\gamma$ uses a vertex of $L_i$, and thus $\gamma$ contains an open path from $L_i$ to $\infty$ using no vertices of $G_n^{(q)}$,  contradicting \eqref{eq:7_no_bad}. \end{proof}

We extract another useful observation from the proof of Lemma \ref{no_bad_large}.

\begin{lem} Let $G_n \in \cal{G}_n$. Each large component $L_i$ corresponding to $G_n$ is surrounded by exactly one cutset $\Gamma_n^{(q)}$ associated to a connected component $G_n^{(q)}$ of $G_n$.
\label{case2}
\end{lem}

\begin{proof} Fix $G_n \in \cal{G}_n$. Each large component $L_i$ associated to $G_n$ is surrounded by at least one of the cutsets $\Gamma_n^{(q)}$. Suppose $L_i$ is surrounded by $\Gamma_n^{(q)}$ and by $\Gamma_n^{(q')}$ for $q \neq q'$. Appealing to \eqref{eq:7_no_bad} in the proof of Lemma \ref{no_bad_large}, we find $\pa_oL_i \cap \pa^\om L_i \subset \pa^\om G_n^{(q)}$ and $\pa_oL_i \cap \pa^\om L_i \subset \pa^\om G_n^{(q')}$. As $L_i \subset \B{C}_\infty$, the edge set $\pa_o L_i \cap \pa^\om L_i$ is non-empty, thus $\pa^\om G_n^{(q)} \cap \pa^\om G_n^{(q')}$ is non-empty. But this is impossible, as distinct connected components of $G_n$ must have disjoint open edge boundary. \end{proof}

We use Lemma \ref{no_bad_large} to relate $F_n$ and $H_n$.

\begin{lem} For each $G_n \in \cal{G}_n$, the vertices of $F_n$ are contained in $H_n$. Moreover, the vertex set of $F_n$ is identically $H_n \cap \textbf{\rm \textbf{C}}_\infty$. 
\label{case1}
\end{lem}

\begin{proof} Fix $G_n \in \cal{G}_n$. We begin with \emph{Claim (1)}: if $y \in F_n$ and $y \notin H_n$, there are $L_i$ and $G_n^{(q)}$ so that $y$ is surrounded by $\Gamma_n^{(q)}$ and $\wh{\Gamma}_n^{(i)}$, with $L_i$ itself surrounded by $\Gamma_n^{(q)}$. \emph{Claim (1)} follows directly from the definition of $H_n$: as $y \in F_n$, we have $y \in \hull(\Gamma_n^{(q)})$ for some $q$. From the definition of $H_n$, $y \notin H_n$ implies $y$ is surrounded by some $\wh{\Gamma}_n^{(i)}$ for $i \in A_q$, where we recall that $A_q$ indexes the large components $L_i$ which are surrounded by $\Gamma_n^{(q)}$.  

Suppose for the sake of contradiction there is $y \in F_n \setminus H_n$, and consider $L_i$ and $G_n^{(q)}$ given by \emph{Claim (1)}. Pass to the configuration $\om_i'$ in which each edge of $\pa^\om L_i$ is closed. In this configuration, $\wh{\Gamma}_n^{(i)}$ consists only of closed edges. Let $\Lambda$ be the open cluster containing $y$ in the configuration $\om_i'$, so that $\Lambda$ is surrounded by $\wh{\Gamma}_n^{(i)}$. 

\emph{Claim (2)} is that $\Lambda$ contains $F_n^{(q)}$, defined to be the connected component of $F_n$ containing $G_n^{(q)}$. Let $z \in F_n^{(q)}$ and suppose for the sake of contradiction that $z \notin \Lambda$. Let $\gamma$ be a path from $y$ to $z$ within $F_n^{(q)}$, so that in $\om$, the path $\gamma$ uses only open edges. From the assumption $z \notin \Lambda$, if we pass to $\om_i'$, we see $\gamma$ uses a closed edge $e$, which is necessarily an element of $\pa^\om L_i$ back in $\om$. 

As $\gamma$ is a path in $F_n^{(q)}$, it joins two vertices which are either in $G_n$ or in one of the small components $S_j$. But $e \in \pa^\om L_i$, so an endpoint of $e$ must also lie in $L_i$. It is impossible for $e$ to satisfy all these requirements. Thus, \emph{Claim (2)} holds, and consequently $G_n^{(q)} \subset \Lambda$. In particular, $G_n^{(q)}$ is surrounded by $\wh{\Gamma}_n^{(i)}$, which implies through \emph{Claim (1)} that $L_i$ is bad. We apply Lemma \ref{no_bad_large} to conclude the vertex set of $F_n$ is contained in $H_n$. 

Thus, $F_n \subset H_n \cap \B{C}_\infty$, and to complete the proof it remains to show the opposite containment. This is immediate from the construction of $H_n$. Indeed, suppose $x \in H_n \cap \B{C}_\infty$. Then $x$ is surrounded by some $\Gamma_n^{(q)}$, and hence $x$ is either in $G_n^{(q)}$ for some $q$, or $x$ is an element of one of the large or small components ($L_i$ or $S_j$) associated to $G_n$. It is impossible for $x$ to lie in any $L_i$, as these components were excised in the construction \eqref{eq:7_hull} of $H_n$.  \end{proof}

In the last result of this subsection, we use Proposition \ref{finite_per} to bound the perimeter of each $P_n$. 

\begin{coro} Let $d\geq 3$ and $p > p_c(d)$. There are positive constants $c_1(p,d), c_2(p,d)$ and $\gamma(p,d)$ so that 
\begin{align}
\prob_p\left( \max_{G_n \in \cal{G}_n} \per( nP_n ) \geq \gamma n^{d-1} \right) \leq c_1 \exp\left(-c_2 n^{1/(d-1)} \right)  \,.
\end{align}
\label{finite_per_2}
\end{coro}

\begin{proof} Work in the event $\cal{E}$ corresponding to Proposition \ref{finite_per} that for each $G_n \in \cal{G}_n$, the corresponding cutset $\Gamma_n$ satisfies $| \Gamma_n | < \gamma n^{d-1}$. Define the polytope $n\wt{P}_n$ as 
\begin{align}
n\wt{P}_n :=  \bigcup_{x \in H_n} Q(x) \,.
\end{align}
Every boundary face of $n\wt{P}_n$ has $\cal{H}^{d-1}$-measure one, and these boundary faces are in one-to-one correspondence with $\Gamma_n$. Thus, in $\cal{E}$, the polytope $n\wt{P}_n$ has perimeter at most $\gamma n^{d-1}$.   As $nP_n = n\wt{P}_n \cap [-n,n]^d$, the perimeter of $nP_n$ is at most $\gamma n^{d-1} + 2d(2n)^{d-1}$, completing the proof. \end{proof}

\subsection{A contiguity argument} Recall the metric $\dw$ introduced in \eqref{eq:2_metric} We adapt the argument of Section 16.2 of \cite{stflour} to our situation to show $\mu_n$ and $\nu_n$ are $\dw$-close with high probability. 

\begin{rmk} We use another renormalization argument, now at a different scale. It is convenient to reuse notation from Section \ref{sec:coarse_original}. Define $\ell(n) := \lfloor n^{(1-\e(d))/2d}  \rfloor$, and suppress the dependence of on $n$ by writing $\ell(n)$ as $\ell$. \emph{Redefine} $\un{B}(x)$ to be the $\ell$-cube $(2\ell) x + [-\ell, \ell]^d$. We also work with $3\ell$-cubes, defined as in \eqref{eq:6.1_3k}, insofar as they are present in the statement of Proposition \ref{gamma_boundary_of_F}. 
\end{rmk} 

For $\delta >0$, introduce the $\Z^d$-process $\{ Z_x^{(\delta)} \}_{x \in \Z^d}$, with each $Z_x^{(\delta)}$ the indicator function of 
\begin{align}
\left\{ \frac{ | \B{C}_\infty \cap \un{B}(x) | }{  \cal{L}^d(\un{B}(x)) } \in \left( \theta_p(d) - \delta,\, \theta_p(d) + \delta \right) \right\} \,. 
\end{align}
Using Corollary \ref{density_control} and a careful examination of the contours defining $F_n$ and $nP_n$, we  show $\wt{\mu}_n$ and $\nu_n$ are close. Recall that $\e(d)$ was defined in \eqref{eq:epsilon_d}.

\begin{prop} Let $d \geq 3$ and let $p > p_c(d)$. Let $Q \subset [-1,1]^d$ be an axis-parallel cube. For all $\delta  >0$, there are positive constants $c_1(p,d,\delta)$ and $c_2(p,d,\delta)$  so that 
\begin{align}
\prob_p \left( \max_{G_n \in \cal{G}_n} | \,\wt{\mu}_n (Q) - \nu_n(Q)| \geq \delta \right) \leq c_1 \exp \left(- c_2 n^{(1-\e(d))/2d} \right) \,.
\end{align}
\label{finite_close}
\end{prop}

\begin{proof} Fix $G_n \in \cal{G}_n$, and let $F_n$, $\wt{\mu}_n$, $P_n$ and $\nu_n$ be the objects constructed above for this $G_n$. Throughout the proof, we use bounds involving positive constants $c(d), c(p,d)$ and so on, which may change from line to line. Let $\un{L}$ denote the following collection of $\ell$-cubes:
\begin{align}
\un{L} := \Big \{ \un{B}(x) \:: \un{B}(x) \cap [-n-2k,n+2k]^d \neq \emptyset \Big\} \,,
\end{align}
For each $\ell$-cube $\un{B}(x)$, we have the bounds
\begin{align}
\wt{\mu}_n \left( n^{-1} \un{B}(x) \right) \leq c(d) \left( \frac{\ell}{n} \right)^d\,, \hspace{5mm}  \nu_n \left( n^{-1} \un{B}(x) \right) \leq c(d) \left( \frac{\ell}{n} \right)^d \,.
\label{eq:7_finite_close_0}
\end{align}
The boundary $\pa Q$ intersects at most $c(d)n^{d-1}$ cubes $\un{B}(x)$, thus by \eqref{eq:7_finite_close_0}, we have
\begin{align}
| \wt{\mu}_n(Q) - \nu_n(Q) | \leq c(d) \left( \frac{ \ell^d}{n} \right) + \sum_{\un{B}(x) \in \un{L} } \left|\, \wt{\mu}_n \left(n^{-1} \un{B}(x) \right) - \nu_n\left( n^{-1} \un{B}(x)\right) \right| \,.
\end{align}
Define 
\begin{align}
\cal{E}_1 := \left\{ \max_{G_n \in \cal{G}_n} \per(nP_n) < \gamma n^{d-1} \right\}\,, \hspace{5mm}  \cal{E}_2 := \left\{ \max_{G_n \in \cal{G}_n} |\Gamma_n| < \gamma n^{d-1} \right\} \,,
\end{align}
\begin{align}
\cal{E}_3 := \left\{ \max_{G_n \in \cal{G}_n} | \pa^\om G_n | \leq \eta_3 n^{d-1} \right\} \,,
\end{align}
so that $\cal{E}_1$, $\cal{E}_2$ and $\cal{E}_3$ are respectively high probability events from Corollary \ref{finite_per_2}, Proposition \ref{finite_per} and Lemma \ref{surface_boundary}. Finally, let $\cal{E}_4$ be the high probability event in the statement of Proposition \ref{gamma_boundary_of_F}. Work within the intersection of $\cal{E}_1$ through $\cal{E}_4$. This allows us to think of $nP_n$ and $F_n$ as objects with perimeters on the order of $n^{d-1}$. 

Motivated by $\cal{E}_4$, define $\un{L}' \subset \un{L}$ as
\begin{align}
\un{L}' := \left\{ \un{B}(x) \in \un{L} : \begin{matrix} \un{B}(x) \cap nP_n = \emptyset\, \text{ or }\, \un{B}(x) \cap nP_n = \un{B}(x) \\ \text{ \textbf{and} } \\ \un{B}(x) \cap F_n = \emptyset\, \text{ or }\, \un{B}(x) \cap F_n = \un{B}(x) \cap \B{C}_\infty \end{matrix} \right\} \,.
\end{align}
From working in $\cal{E}_1$ through $\cal{E}_4$, there are at most $c(p,d) n^{d-1}$ $\ell$-cubes $\un{B}(x)$ in $\un{L} \setminus \un{L}'$. This is especially due to the event $\cal{E}_4$ from Proposition \ref{gamma_boundary_of_F}, which was designed for use here. Thus,
\begin{align}
| \wt{\mu}_n(Q) - \nu_n(Q) | \leq c(p,d) \left( \frac{ \ell^d}{n} \right) + \sum_{\un{B}(x) \in \un{L}' } \left|\, \wt{\mu}_n \left(n^{-1} \un{B}(x) \right) - \nu_n\left( n^{-1} \un{B}(x) \right) \right| \,.
\label{eq:7_finite_close_1}
\end{align}
Define a further subcollection of boxes, $\un{L}'' \subset \un{L}'$:
\begin{align}
\un{L}'' :=  \left\{ \un{B}(x) \in \un{L}' : 
\begin{matrix} \un{B}(x) \cap nP_n = \emptyset\, \text{ and }\, \un{B}(x) \cap F_n = \emptyset \\ 
\text{ \textbf{or} } \\ 
\un{B}(x) \cap nP_n = \un{B}(x)\, \text{ and }\, \un{B}(x) \cap F_n = \un{B}(x) \cap \B{C}_\infty 
\end{matrix} \right\} \,.
\end{align}
We claim $\un{L}'' = \un{L}'$. To prove this, we show two of the four cases defining $\un{L}'$ are impossible. 

\emph{Case (i):}  Suppose $ \un{B}(x) \cap F_n = \un{B}(x) \cap \B{C}_\infty $ and $ \un{B}(x) \cap nP_n = \emptyset$. Appealing to Lemma~\ref{case1}, as $F_n \subset H_n$, this is impossible unless $\B{C}_\infty \cap \un{B}(x) = \emptyset$, one of the two allowed options.

\emph{Case (ii):} Suppose $\un{B}(x)  \cap F_n= \emptyset$ and $\un{B}(x) \cap nP_n = \un{B}(x)$. If $\un{B}(x) \cap \B{C}_\infty  = \emptyset$, we are in one of the two allowed options, so we may assume there is $y \in \un{B}(x) \cap \B{C}_\infty$. As $\un{B}(x) \cap nP_n = \un{B}(x)$, it follows that $y \in H_n$. Thus $y$ is surrounded by some $\Gamma_n^{(q)}$, and either $y \in F_n$ or $y \in L_i$ for some $i$. The former option is impossible by hypothesis, and $y \in L_i$ for some $i$. By Lemma~\ref{case2}, $L_i$ is surrounded by exactly one of the $\Gamma_n^{(q)}$, and $y \in H_n$ implies $y \in H_n^{(q)}$ and $y \notin H_n^{(q')}$ whenever $q' \neq q$. But in the construction of $H_n^{(q)}$, the hull of $\wh{\Gamma}_n^{(i)}$ is removed from the hull of $\Gamma_n^{(q)}$. As the hull of $\wh{\Gamma}_n^{(i)}$ contains $L_i$ and hence $y$, it is impossible that $y \in H_n^{(q)}$, a contradiction.

We thus conclude $\un{L}'' = \un{L}'$. Replace $\un{L}'$ by $\un{L}''$ in \eqref{eq:7_finite_close_1} and use the defining properties of $\un{L}''$ with the definitions of $\wt{\mu}_n$ and $\nu_n$: 
\begin{align}
|\, \wt{\mu}_n(Q) - \nu_n(Q) | &\leq c(p,d) \left( \frac{ \ell^d}{n} \right) + \sum_{\un{B}(x) \in \un{L}'' } \left| \wt{\mu}_n \left(n^{-1} \un{B}(x) \right) - \nu_n\left( n^{-1} \un{B}(x) \right) \right| \,, \\
&\leq c(p,d) \left( \frac{ \ell^d}{n} \right) + \sum_{\un{B}(x) \in \un{L}'' } \left(  \left|  \frac{| \B{C}_\infty \cap \un{B}(x) | }{n^d}-  \frac{\theta_p(d) \cal{L}^d( \un{B}(x) )}{n^d}    \right| \right) \,.
\label{eq:7_finite_close_2}
\end{align}
Form one last high probability event $\cal{E}_5$ using the $\Z^d$-process $\{Z_x^{(\delta)} \}_{x \in \Z^d}$: let $\cal{E}_5$ be the event that $Z_x^{(\delta)} = 1$ for all $x$ with $\un{B}(x) \in \un{L}''$. By Corollary \ref{density_control}, there are $c_1(p,d,\delta), c_2(p,d,\delta) > 0$ so that 
\begin{align}
\prob(\cal{E}_5^c) \leq c(d) n^d c_1 \exp \Big( - c_2 n^{(1-\e(d))/2d} \Big) \,.
\end{align}
Working now in the intersection of $\cal{E}_1$ through $\cal{E}_5$, bound $|\,\wt{\mu}_n(Q) - \nu_n(Q) |$, continuing from \eqref{eq:7_finite_close_2}:
\begin{align}
|\, \wt{\mu}_n(Q) - \nu_n(Q) | &\leq c(p,d) \left(\frac{ \ell^d}{n}\right)  +| \un{L}''| \max_{\un{B}(x) \in \un{L}''} \left(  \left|  \frac{| \B{C}_\infty \cap \un{B}(x) | }{n^d}-  \frac{\theta_p(d) \cal{L}^d( \un{B}(x) )}{n^d}    \right| \right) \,, \\
&\leq c(p,d)  \left(\frac{ \ell^d}{n}\right)  + \frac{ | \un{L}'' |}{n^d} ( 2 \cal{L}^d( \un{B}(x) ) \delta ) \,,\\
&\leq c(p,d) \left(\frac{ \ell^d}{n}\right) + c(d) \delta \,,
\end{align}
where we have used the bound $|\un{L}''| \leq |\un{L}| \leq c(d) (n /\ell)^d$ in going from the second line to the third line directly above. Take $n$ sufficiently large to get $|\,\wt{\mu}_n(Q) - \nu_n(Q) |  \leq c(p,d) \delta$. The proof is completed by using bounds for the probabilities of $\cal{E}_1, \dots, \cal{E}_5$. \end{proof}

We combine the preceding result with Lemma \ref{adding_k_small} to establish $\dw$-closeness of $\mu_n$ and $\nu_n$. The following is the central theorem of this section.

\begin{thm} Let $d\geq 3$, $p > p_c(d)$ and let $\delta > 0$. There are positive constants $c_1(p,d,\delta),\, c_1(p,d,\delta)$ so that 
\begin{align}
\prob_p \left( \max_{G_n \in \cal{G}_n} \dw(\mu_n, \nu_n) \geq \delta \right) \leq c_1 \exp \left(- c_2 n^{(1-\e(d))/2d} \right) \,.
\end{align}
\label{emp_close}
\end{thm}

\begin{proof} Let $\delta > 0$ and let $\Delta^k \equiv \Delta^{k,d}$ denote the dyadic cubes in $[-1,1]^d$ at scale $k$, introduced in Section \ref{sec:notation_3}. There is no confusion between the integer $k$ used for dyadic scales and the renormalization parameter from Sections \ref{sec:coarse_original} and \ref{sec:coarse_applied}, as the latter is fixed (see Remark \ref{rmk:7_fixed}). For $Q \in \Delta^k$, use Lemma \ref{adding_k_small} and Proposition \ref{finite_close} to find positive constants $c_1(p,d,\delta), c_2(p,d,\delta)$ so that
\begin{align}
\prob_p \left( \max_{G_n \in \cal{G}_n} |\, \mu_n( Q) - \nu_n(Q) | < \delta \right) \geq 1 - c_1 \exp \Big(-c_2 n^{(1-\e(d))/2d} \Big) \,.
\label{eq:emp_close}
\end{align}

Choose $j$ large enough so that $2^{-j} < \delta$, and let $Q_1, \dots, Q_m$ enumerate all dyadic cubes at scales between $0$ and $j-1$ contained in $[-1,1]^d$. The number $m$ of these cubes depends only on $\delta$ and $d$. Let $\cal{E}_i$ be the high probability event corresponding to \eqref{eq:emp_close} for $Q_i$, and work in $\cal{E} := \bigcap_{i=1}^m \cal{E}_i$, so that by definition \eqref{eq:2_metric} of the metric $\dw$,
\begin{align}
\dw(\mu_n, \nu_n) &\leq \sum_{k=0}^{j-1} \frac{1}{2^k} \sum_{Q \in \Delta^k} \frac{1}{| \Delta^k|} | \mu_n(Q) - \nu_n(Q) | + \sum_{k=j}^{\infty} \frac{1}{2^k} \sum_{Q \in \Delta^k} \frac{1}{| \Delta^k|} | \mu_n(Q) - \nu_n(Q) | \,, \\
&\leq 2\delta + \sum_{k=j}^{\infty} \frac{1}{2^k} \sum_{Q \in \Delta^k} \frac{1}{| \Delta^k|} | \mu_n(Q) - \nu_n(Q) | \,.
\label{eq:7_emp_close}
\end{align}
We control the sum directly above via crude bounds: there is $c(d) >0$ so that for each dyadic cube $Q$, we have $\mu_n(Q) \leq c(d)$ and $\nu_n(Q) \leq c(d)$. Through our choice of $j$, the sum in \eqref{eq:7_emp_close} is then bounded by $c(d) \delta$. Thus, in $\cal{E}$, we have $\dw( \mu_n, \nu_n) \leq c(d) \delta$. As $m$ depends only on $\e$ and $d$, the proof is completed using \eqref{eq:emp_close} with a union bound to control the probability of $\cal{E}^c$.  \end{proof}
 
\subsection{Closeness to sets of finite perimeter} We explore consequences of Theorem \ref{emp_close} before moving to the final section. Recall from Section \ref{sec:notation_3} that $\cal{B}_d$ is the ball about the zero measure of radius $3^d$ in the total variation norm. For $\gamma, \xi > 0$, define the following collection of measures in $\cal{B}_d$. 
\begin{align}
\cal{P}_{\gamma,\,\xi} := \Big\{ \nu_F \:: F \subset [-1,1]^d,\, \per(F) \leq \gamma,\, \cal{L}^d(F) \leq \cal{L}^d( (1+\xi) W_{p,d}) \Big\} \,,
\end{align}
where given $F \subset [-1,1]^d$ Borel, the measure $\nu_F$ representing $F$ is defined as in Section~\ref{sec:notation_3}.

\begin{coro} Let $d \geq 3$, $p >p_c(d)$ and let $\delta >0$. There are positive constants $c_1(p,d,\delta,\xi)$, $c_2(p,d,\delta,\xi)$ and $\gamma(p,d)$ so that 
\begin{align}
\prob_p \left( \max_{G_n \in \cal{G}_n} \dw(\mu_n, \cal{P}_{\gamma,\,\xi}) \geq \delta \right) \leq c_1 \exp \left( -c_2 n^{(1-\e(d))/2d} \right) \,.
\end{align}
\label{close_finite}
\end{coro}

\begin{proof} Let $\delta, \delta', \xi >0$ and let $\gamma(p,d)$ be as in Corollary \ref{finite_per_2}. We first show with high probability, the measures $\nu_n$ lie in $\cal{P}_{\gamma,\, \xi}$, and then we apply Theorem \ref{emp_close}. Work in the intersection of 
\begin{align}
\cal{E}_1 := \left\{ \max_{G_n \in \cal{G}_n} \per(nP_n) < \gamma n^{d-1} \right\}\,, \hspace{5mm} \cal{E}_2 := \left\{ \max_{G_n \in \cal{G}_n} \dw(\mu_n, \nu_n) < \min(\delta, \delta') \right\} \,,
\end{align}
\begin{align}
\cal{E}_3 := \left \{ \frac{ |\giant|}{(2n)^d}  \in ( \theta_p(d) - \delta' , \theta_p(d) + \delta' ) \right \} \,,
\end{align}
respectively from Corollary \ref{finite_per_2}, Theorem \ref{emp_close} and Corollary \ref{density_control}. As we are in $\cal{E}_2$, for each $nP_n$ corresponding to $G_n \in \cal{G}_n$ we have
\begin{align}
\theta_p(d) \cal{L}^d(nP_n) &< \delta' n^d + | G_n| \,,\\
&< \delta' n^d + | \giant | / d! \,.
\end{align}
From working in $\cal{E}_3$, we further conclude
\begin{align}
\cal{L}^d(nP_n) &< n^d \left( \frac{\delta'}{ \theta_p(d)} + \frac{2^d}{ d!} \left( 1 + \frac{\delta'}{ \theta_p(d)} \right) \right) \,, \\
&< n^d \left( \cal{L}^d( (1+\xi) W_{p,d})\right) \,,
\end{align}
where we have taken $\delta'$ small according to $p,d$ and $\xi$. As we are in $\cal{E}_1$, we conclude $\nu_n \in \cal{P}_{\gamma,\,\xi}$ for each $G_n \in \cal{G}_n$. \end{proof}

The next result links $\dw$ to the notion of weak convergence.

\begin{lem} For $\zeta, \zeta_n \in \cal{P}_{\gamma,\, \xi}$, $\dw(\zeta_n, \zeta) \to 0$ if and only if $\zeta_n$ converges to $\zeta$ weakly.
\label{rmk:7_weakness}
\end{lem}

\begin{proof} Let $\zeta,\, \zeta_n \in \cal{P}_{\gamma,\, \xi}$; if $\dw(\zeta_n, \zeta) \to 0$, it follows from \eqref{eq:2_metric} that $\zeta_n(Q) \to \zeta(Q)$ for each dyadic cube $Q$. As any open subset $U \subset [-1,1]^d$ may be decomposed into a countable union of almost disjoint dyadic cubes, and as each measure in $\cal{P}_{\gamma, \,\xi}$ is absolutely continuous with respect to Lebesgue measure, we conclude $\liminf_{n \to \infty} \zeta_n(U) \geq \zeta(U)$. By the Portmanteau theorem, $\zeta_n$ converges to $\zeta$ weakly. 

Conversely, if $\zeta,\, \zeta_n \in \cal{P}_{\gamma,\, \xi}$ are such that $\zeta_n \to \zeta$ weakly, the Portmanteau theorem also tells us $\zeta_n(A) \to \zeta(A)$ for all continuity sets $A \subset [-1,1]^d$ of $\zeta$, in particular whenever $A$ is a dyadic cube (using the absolute continuity of $\zeta$). Thus $\dw(\zeta_n, \zeta) \to 0$.
\end{proof}

We use Lemma \ref{rmk:7_weakness} to establish compactness of $\cal{P}_{\gamma,\,\xi}$.

\begin{lem} The collection of measures $\cal{P}_{\gamma,\,\xi}$ is compact subset of the metric space $(\cal{B}_d, \dw)$. 

\label{compact}
\end{lem}

\begin{proof} By Banach-Alaoglu, the set $\cal{B}_d$ is compact when equipped with the topology of weak convergence. The continuous functions on $[-1,1]^d$ (equipped with the supremum norm topology) form a separable space, so $\cal{B}_d$ is sequentially compact in the topology of weak convergence. By Lemma \ref{rmk:7_weakness}, it suffices to show  $\cal{P}_{\gamma,\,\xi}$ is sequentially closed. 

Let $\{ \nu_{F_n}\}_{n=1}^\infty$ be a sequence of measures in $\cal{P}_{\gamma,\,\xi}$ converging with respect to $\dw$. Using the definition \eqref{eq:2_metric} of $\dw$, one can show (first by approximating open sets by finite unions of dyadic cubes, and then approximating Borel sets by open sets) that for any $E \subset [-1,1]^d$ Borel, the sequence $\cal{L}^d( E \cap F_n)$ is Cauchy. Thus the indicator functions $\1_{F_n}$ converge pointwise a.e. to some $\1_F$, and the bounded convergence theorem converts this into $L^1$-convergence. 

As $\1_{F_n} \to \1_F$ in $L^1$-sense, \eqref{eq:2_metric} implies $\dw( \nu_{F_n}, \nu_F) \to 0$ as $n \to \infty$, and it remains to check that $\nu_F \in \cal{P}_{\gamma,\,\xi}$. By Fatou's lemma and Lemma \ref{lsc}, 
\begin{align}
\cal{L}^d(F) \leq \liminf_{n\to \infty} \cal{L}^d(F_n)\,, \hspace{5mm}  \per(F) \leq \liminf_{n \to \infty} \per (F_n) \,,
\end{align}
and hence $\nu_F \in \cal{P}_{\gamma, \xi}$. \end{proof}

We work with $\dw$ extensively in the next section. In fact most of our effort will go towards proving the following precursor to Theorem~\ref{main_L1}.

\begin{thm} For $d \geq 3$ and $p > p_c(d)$, let $W_{p,d}$ be the Wulff crystal from Theorem \ref{main_L1}. Define the following subset of $\cal{M}([-1,1]^d)$:
\begin{align}
\cal{W} := \Big\{ \nu_E : E = W_{p,d} + x, \text{ with } W_{p,d} + x \subset [-1,1]^d \Big\} \,.
\end{align}
$\prob_p$-almost surely,
\begin{align}
\max_{G_n \in \cal{G}_n} \dw( \mu_n, \cal{W} ) \xrightarrow[n \to \infty]{}  0 \,. 
\end{align}
\label{main}
\end{thm}



{\large\section{\B{Lower bounds and main results}}\label{sec:final}}

We prove the main theorems of the paper in this section, throughout which we assume $d\geq 3$ and $p > p_c(d)$. The strategy is as follows: first use Corollary~\ref{close_finite} to anchor the empirical measures $\mu_n$ near (in the sense of $\dw$) measures representing sets of finite perimeter.  Whenever an empirical measure $\mu_n$ is close to such a measure $\nu_F$, we relate the conductance of the corresponding $G_n$ to the conductance of the continuum set $F$. 

The challenge is to show that when $\dw(\mu_n, \nu_F)$ is small, $|\pa^\om G_n|$ and $\cal{I}_{p,d}(nF)$ are close. This is done using a covering lemma, working locally near the boundary of $F$. This local perspective guides a surgery performed on $\pa^\om G_n$ to invoke concentration estimates from Section \ref{sec:consequences_1}. This strategy shares much with the argument in Section 6 of \cite{CeTh_LOWER}. In particular, we rely on the compactness of $\cal{P}_{\gamma, \, \xi}$ established in Lemma \ref{compact}. \newline

\subsection{Setup, the reduced boundary and a covering lemma}\label{sec:setup} Let $\al_d$ denote the volume of the $d$-dimensional Euclidean unit ball. Given a closed Euclidean ball $B(x,r)$ centered at $x \in \R^d$ of radius $r >0$ and a unit vector $v \in \mathbb{S}^{d-1}$, define the \emph{lower half-ball} of $B(x,r)$ in the direction $v$:
\begin{align}
B_-(x,r,v) := \Big\{ y \in B(x,r) : (y -x ) \cdot v \leq 0 \Big\}\,.
\end{align}

\begin{defn}
 For $F \subset \R^d$ Borel, let $\nabla \1_F$ be the distributional derivative of the indicator function $\1_F$. This is a vector-valued measure whose total variation $|| \nabla \1_F ||(\R^d)$ is the perimeter of $F$. For $F \subset \R^d$ a set of finite perimeter, the \emph{reduced boundary} $\pa^* F$ of $F$ is the set of points $x \in \R^d$ such that $(i)$ and $(ii)$ hold:

$(i)$ $|| \nabla \1_F || ( B(x,r)) > 0$ for any $r > 0$. 
 
$(ii)$ If we define 
\begin{align}
v_r(x) := - \frac{\nabla \1_F ( B(x,r)) }{|| \nabla \1_F || ( B(x,r))} \,,
\end{align} 
then $v_r(x)$ tends to a unit vector $v_F(x)$, called the \emph{exterior normal} to $F$ at $x$ as $r \to 0$.
\end{defn}

The following \emph{covering lemma} is specialized to $\cal{I}_{p,d}$.

\begin{lem} (\cite{stflour}, Section 14.3)\,  Let $F \subset \R^d$ be a set of finite perimeter, and let $\cal{I}_{p,d}$ be the surface energy defined in \eqref{eq:new_surface_energy} for $\beta_{p,d}$. For $\delta > 0$ and $s \in (0,1/2)$, there is a finite collection of disjoint balls $\{ B(x_i, r_i) \}_{i=1}^m$ with $x_i \in \pa^* F$ and $r_i \in (0,1)$ for all $i \in \{1, \dots, m\}$, each satisfying
\begin{align}
\cal{L}^d \Big( F \cap B(x_i, r_i)\,\, \Delta \,\, B_-(x_i,r_i, v_F(x_i) ) \Big) \leq \delta \al_d r_i^d \,,
\label{eq:8_covering_volume}
\end{align}
\begin{align}
\left| \cal{I}_{p,d}(F) - \sum_{i=1}^m \al_{d-1} r_i^{d-1} \beta_{p,d}(v_F(x_i)) \right| \leq \wt{\delta}_F(s) \,,
\label{eq:8_covering_surface}
\end{align}
where $\wt{\delta}_F(s) := \tfrac{s}{4} \cal{I}_{p,d}(F)$. 
\label{covering}
\end{lem}

\begin{rmk} Given a set $F \subset [-1,1]^d$ of finite perimeter and a ball $B(x,r)$ with $x \in \pa^*F$ arising from Lemma \ref{covering}, we abbreviate $B_-(x,r, v_F(x))$ as $B_-(x,r)$.
\end{rmk}

We define two global parameters appearing throughout this section. Given a collection of balls $ \{ B(x_i, r_i) \}_{i=1}^m$ as in Lemma \ref{covering}, define 
\begin{align}
\e_F := \delta \min_{i=1}^m \al_d (r_i)^d \,,
\label{eq:8_epsilon}
\end{align}
so that $\e_F$ depends on $F, \delta$ and $s$. Also define
\begin{align}
\lambda_F(s) := (1-2s) \cal{I}_{p,d}(F) \,.
\label{eq:8_lambda}
\end{align}

\begin{rmk} Given $G_n \in \cal{G}_n$ and a ball $B(x_i,r_i)$ as in Lemma \ref{covering} for $\delta >0$ and $s \in (0,1/2)$, let
\begin{align}
\cal{E}(G_n,i) :=  \Big\{ | \pa^\om G_n \cap nB(x_i, r_i) | \leq (1-s) n^{d-1} \al_{d-1} (r_i)^{d-1} \beta_{p,d}(v_i) \Big\}\, 
\label{eq:8_event}
\end{align}
\end{rmk}

The next lemma controls the event that $|\pa^\om G_n|$ is too small using the events $\cal{E}(G_n,i)$. 

\begin{lem} Suppose $F \subset [-1,1]^d$ is a set of finite perimeter. Let $\{B(x_i, r_i)\}_{i=1}^m$ be a collection of balls as in Lemma \ref{covering} for $\delta > 0$ and $s \in (0,1/2)$. For each $G_n \in \cal{G}_n$, 
\begin{align}
\Big\{| \pa^\om G_n| \leq \lambda_F(s) n^{d-1} \Big\} \subset  \bigcup_{i=1}^m \cal{E}(G_n,i) \,,
  \end{align}
where $\lambda_F(s)$ is defined in \eqref{eq:8_lambda}.
\label{9.1}
\end{lem}

\begin{proof} Because the balls $\{ B(x_i,r_i) \}_{i=1}^m$ were chosen in accordance with Lemma \ref{covering}, we combine \eqref{eq:8_covering_surface} with the defintion of $\wt{\delta}_F(s)$ to obtain
\begin{align}
\left| \cal{I}_{p,d} (F) - \sum_{i=1}^m \al_{d-1} (r_i)^{d-1} \beta_{p,d}(v_i) \right| \leq \frac{s}{2} \left( \sum_{i=1}^m \al_{d-1} (r_i)^{d-1} \beta_{p,d}(v_i) \right) \,,
\end{align}
so that 
\begin{align}
\lambda_F(s) \leq (1-s) \left( \sum_{i=1}^m \al_{d-1} (r_i)^{d-1} \beta_{p,d}(v_i) \right) \,.
\label{eq:final_1}
\end{align}
Use the disjointness of the balls in $\{ B(x_i,r_i) \}_{i=1}^m$, (\ref{eq:final_1}) and the definition \eqref{eq:8_event} of $\cal{E}(G_n,i)$. 
 \begin{align}
\Big\{ | \pa^\om G_n| \leq \lambda_F(s) n^{d-1} \Big\} &\subset \left\{ \sum_{i=1}^m | \pa^\om G_n \cap nB(x_i, r_i) | \leq ( 1-s ) n^{d-1} \sum_{i=1}^m \al_{d-1} (r_i)^{d-1} \beta_{p,d}(v_i) \right\} \,,\\
 &\subset \bigcup_{i=1}^m \cal{E}(G_n,i) 
 \end{align}
We complete the proof using the definition \eqref{eq:8_epsilon} of $\e_F$. \end{proof}

\subsection{Local surgery on each $\pa^\om G_n$} In this subsection, we think of $F \subset [-1,1]^d$ as a fixed polytope and work with a fixed $G_n \in \cal{G}_n$. Let $\{ B(x_i, r_i) \}_{i=1}^m$ be a collection of balls as in Lemma \ref{covering} for $F$. Also fix $B(x_i, r_i) \in \{ B(x_i, r_i)\}_{i=1}^m$, and denote this ball as $B(x,r)$, with $v := v_F(x) \in \mathbb{S}^{d-1}$ be the exterior normal vector associated to $x \in \pa^* F$.

Let $B_-(x,r)$ be the lower half-ball associated to $B(x,r)$ and $v$. Let $D(x,r)$ be the closed equatorial disc of this ball, so that $\hyp(D(x,r))$ is orthogonal to $v$. For $h >0$ small, define $r' := (1-h^2)^{1/2} r$, and let $D(x,r') \subset D(x,r)$ be the closed disc of radius $r'$ centered at $x$. Note that $D(x,r')$ is built so that $\cyl( D(x, r'), hr') \subset B(x,r)$. These geometric objects guide a surgery we perform on $\pa^\om G_n$.

 Let $J_n = J_n(\om)$ be the open edges intersecting $\cal{N}_{5d} (n D(x,r))$. Closing each edge in $J_n$ and each edge in $\pa^\om G_n$ breaks $\B{C}_\infty \cap nB(x,r)$ into a finite number of connected components. A component $\Lambda$ is \emph{outward} if it is contained in $G_n \cap (nB(x,r) \setminus nB_-(x,r))$ and is \emph{inward} if it lies in $nB_-(x,r) \setminus G_n$. 

We are only interested in $\Lambda$ containing vertices incident to edges in $J_n$. Enumerate all such outward components as $\Lambda_1^+ , \dots , \Lambda_{\ell^+}^+$, and all such inward components as $\Lambda_1^- , \dots, \Lambda_{\ell^-}^-$. A component (outward or inward) is \emph{good} if it is contained in $n \cyl( D(x,r), hr')$ and is \emph{bad} otherwise. Our notation suppresses the dependence of these components on $G_n$, $B(x,r)$ and $F$.

\begin{figure}[h]
\centering
\includegraphics[scale=1]{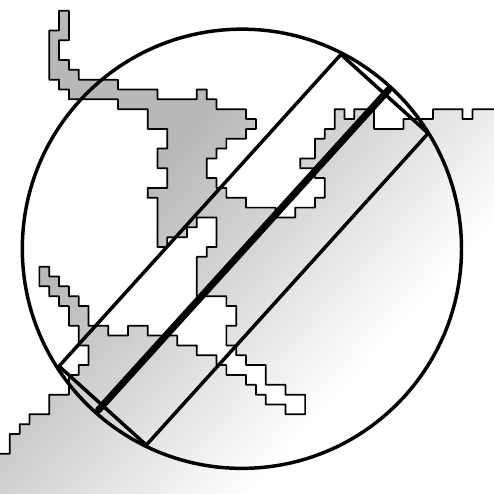}\hspace{10mm}
\includegraphics[scale=1]{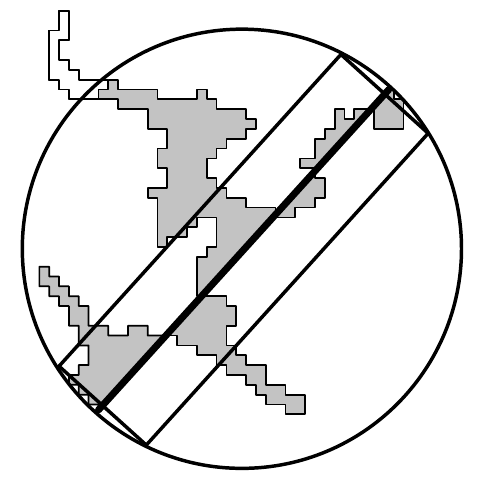}
\caption{The thin cylinder $n\cyl (D(x,r'), hr')$ is drawn as a rectangle, the central disc $nD(x,r)$ is the bold line. On the left is $G_n$ viewed up close. On the right, inward and outward components are in grey (outward components point up and to the left). There are three good components and three bad components.} 
\label{fig:final_2}
\end{figure}

\begin{rmk} Every outward component is a subgraph of $G_n$. Thus, \emph{outward} is understood as relative to the bottom half-ball $nB_-(x,r)$. Figure \ref{fig:final_2} illustrates the objects introduced so far. We regard outward and inward components $\Lambda_j^\pm$ as subgraphs of $\textbf{C}_\infty \cap nB(x,r)$, so that the edge sets $\text{\rm E}(\Lambda_j^\pm)$ are collections of open edges. 
\end{rmk}

The next lemma efficiently truncates bad components. Let $\al \in [0,h/2]$, and given $\Lambda_j^+$, define
\begin{align}
\slice_j^+(\al) := \Big\{ e \in \text{\rm E}( \Lambda_j^+) : e \cap [ nD(x,r) + n\al r' v ] \neq \emptyset \Big\} \,,
\end{align}
where in the above intersection, the edge $e$ is regarded as a line-segment in $\R^d$. Thus $\slice_j^+(\al)$ is the set of edges in $\Lambda_j^+$ touching a prescribed translate of $nD(x,r)$. For an inward component $\Lambda_j^-$, likewise define
\begin{align}
\slice_j^-(\al) := \Big\{ e \in \text{ \rm E}(\Lambda_j^-) : e \cap [ nD(x,r) - n\al r' v ] \neq \emptyset \Big\} \,.
\end{align}

\begin{lem} Let $G_n \in \cal{G}_n$ and $B(x,r)$ with $r \in (0,1)$ be fixed, and let $\Lambda_j^\pm$ denote the outward and inward components constructed above from $G_n$, $B(x,r)$ and $F$. Let $h>0$. There is a positive constant $c(d)$ so that for each outward component $\Lambda_j^+$, there is $h_j^+ \in~[0,h/2]$ so that
\begin{align}
| \slice_j^+(h_j^+)  | \leq c(d) \frac{ |\Lambda_j^+|}{nhr} \,,
\label{eq:8_trunc_bad_1}
\end{align}
and for each inward component $\Lambda_j^-$, there is $h_j^- \in [0,h/2]$ so that 
\begin{align}
| \slice_j^-(h_j^-) | \leq c(d) \frac{ |\Lambda_j^+|}{nhr} \,.
\label{eq:8_trunc_bad_2}
\end{align}
\label{trunc_bad}
\end{lem}

\begin{proof}
Let $\Lambda_j^+$ be an outward component. For $k \in \{ 1, \dots, \lceil n h \rceil /2 \}$, define $\al_k := k / 2n$. We have
\begin{align}
\bigcup_{k=1}^{ \lceil nh \rceil /2 } \slice_{j}^+(\al_k) \subset \text{\rm E} ( \Lambda_j^+) \,.
\end{align}
When $k$ and $k'$ satisfy $|k - k'|_2 \geq 10d$, the edge sets $\slice_{j}^+(\al_k)$ and $\slice_{j}^+(\al_{k'})$ are disjoint. Thus,
\begin{align}
\sum_{k=1}^{ \lceil nh \rceil /2} | \slice_{j}^+(\al_k) | \leq (10d) |  \text{\rm E} ( \Lambda_j^+) | \,.
\end{align}
For at least one $k \in \{ 1, \dots, \lceil n h \rceil /2 \}$, we must have 
\begin{align}
| \slice_{j}^+ (\al_k) | \leq c(d) \frac{| \Lambda_j^+ | }{n hr } \,,
\label{eq:8_trunc_bad_3}
\end{align}
for some $c(d) > 0$, and where we have slipped $r$ into the denominator because $r~\in~(0,1)$.

Any $\al_k$ satisfying \eqref{eq:8_trunc_bad_3} is at most $h /2$. Pick one such $\al_k$ and relabel it $h_j^+$. Analogous reasoning for inward components gives $h_j^- \in [0, h/2]$ for each inward $\Lambda_j^-$ so that
\begin{align}
| \slice_j^-( h_j^-) | \leq c(d) \frac{ | \Lambda_j^-| }{nhr} \,,
\end{align}
completing the proof. \end{proof}

\begin{rmk} We continue to use the edge sets given by Lemma \ref{trunc_bad} throughout this subsection, but only when working with bad components. When $\Lambda_j^\pm$ is bad, define
\begin{align}
\slice_j^\pm := \slice_j^\pm( h_j^\pm) \,,
\label{eq:8_slice}
\end{align}
and if $\Lambda_j^\pm$ is good, define $\slice_j^\pm$ to be empty. Figure \ref{fig:final_4} depicts the edge sets $\slice_j^\pm$. As with the $\Lambda_j^\pm$, we suppress the dependence of the $\slice_j^\pm$ on $G_n$, $h>0$, $B(x,r)$ and $F$. 
\label{rmk:8_slices}
\end{rmk}

\begin{figure}[h]
\centering
\includegraphics[scale=1]{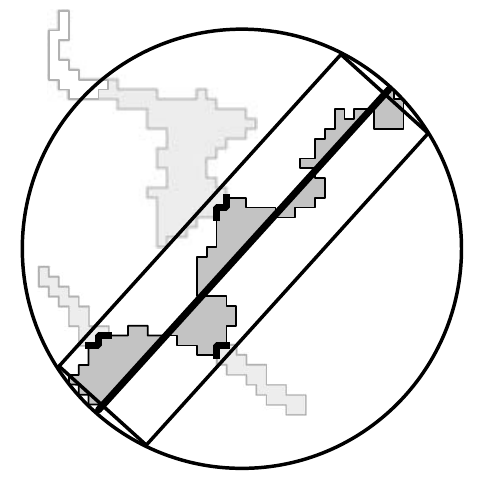}

\caption{The short, bold curves are the efficiently chosen sets of open edges $\slice_j^\pm$ from Lemma \ref{trunc_bad}. We have faded the portions of the bad components which are cut off by the $\slice_j^\pm$. } 
\label{fig:final_4}
\end{figure}

The following is an immediate consequence of Lemma \ref{trunc_bad}.

\begin{coro} Let $G_n \in \cal{G}_n$, $h>0$ and $B(x,r)$ be fixed. Let $\slice_j^\pm$ be the edge sets constructed from $G_n$, $h>0$ and $B(x,r)$. There is $c(d) > 0$ so that
\begin{align}
\sum_{j=1}^{\ell^+} |\slice_j^+| + \sum_{j=1}^{\ell^-} |\slice_j^-| \leq \frac{c(d)}{nhr } \left( \sum_{j=1}^{\ell^+} |\Lambda_j^+| + \sum_{j=1}^{\ell^-} |\Lambda_j^- | \right) \,.
\label{eq:8_coro_trunc_cons}
\end{align}
\label{trunc_cons}
\end{coro}

\begin{rmk} Corollary \ref{trunc_cons} tells us that to control the total size of the $\slice_j^\pm$, it suffices to control the total volume of the $\Lambda_j^\pm$. At the beginning of Section \ref{sec:contiguity}, we built a polytope $P_n$ whose perimeter we could bound, and whose representative measure $\nu_n$ was $\dw$-close to $\mu_n$. Proposition~\ref{prop:closeness} shows the $\ell^1$-distance of $\1_{\B{C}_\infty \cap n P_n}$ and $\1_{\B{C}_\infty \cap nF}$ is small when $\dw(\nu_n, \nu_F)$ is. Proposition \ref{volume} controls the total volume of the $\Lambda_j^\pm$ when these indicator functions are close.
\label{rmk:great_explain}
\end{rmk}

\begin{prop} Let $F \subset [-1,1]^d$ be a polytope. Recall the metric $\dw$ defined in \eqref{eq:2_metric} and the constant $\gamma$ from Corollary \ref{finite_per_2}. Given $G_n \in \cal{G}_n$, let $P_n \subset [-1,1]^d$ be the polytope defined from $G_n$ in \eqref{eq:7_polytope} with representative measure $\nu_n$. For $\delta >0$, there is $\e(d,\delta, F) >0$ and an event $\cal{E}_0$ so that for all $G_n \in \cal{G}_n$,
\begin{align}
\Big\{ \dw(\nu_n,\nu_F) < \e \Big\} &\cap \Big\{ \max_{G_n \in \cal{G}_n} \per(P_n) \leq \gamma \Big\} \cap \cal{E}_0 \\
&\subset  \Big\{ \left\| \1_{\B{C}_\infty \cap nP_n} - \1_{\B{C}_\infty \cap nF} \right\|_{\ell^1} \leq \delta n^d \Big\} \cap \Big\{ \max_{G_n \in \cal{G}_n} \per(P_n) \leq \gamma \Big\} \cap \cal{E}_0 \,.
\end{align}
Moreover, there are positive constants $c_1(p,d,\delta, F), c_1(p,d,\delta, F)$ so that 
\begin{align}
\prob_p (\cal{E}_0^c) \leq c_1 \exp\Big(-c_2 n^{d-1} \Big) \,.
\label{eq:8_closeness}
\end{align}
\label{prop:closeness}
\end{prop}

\begin{proof} Fix $G_n \in \cal{G}_n$ and hence $P_n$ and $\nu_n$. Let $\Delta^k \equiv \Delta^{k,d}$ be the dyadic cubes in $[-1,1]^d$ at scale $k \in \N$. Define:
\begin{align}
\textsf{Q}_0 := \Big\{ Q \in \Delta^k : Q \cap ( \pa F \cup \pa P_n) \neq \emptyset \Big\} \,,
\end{align}
and observe 
\begin{align}
\big\| \1_{\B{C}_\infty \cap nP_n} - &\1_{\B{C}_\infty \cap nF} \big\|_{\ell^1} \leq \sum_{Q\, \in\, \Delta^k} \left\| \1_{\B{C}_\infty \cap nP_n \cap nQ} - \1_{\B{C}_\infty \cap nF \cap nQ} \right\|_{\ell^1} \,, \\
 &\leq \sum_{Q\, \in\, \Delta^k \setminus \textsf{Q}_0} \left\| \1_{\B{C}_\infty \cap nP_n \cap nQ} - \1_{\B{C}_\infty \cap nF \cap nQ} \right\|_{\ell^1} +   \sum_{Q\, \in\, \textsf{Q}_0} \left\| \1_{\B{C}_\infty \cap nP_n \cap nQ} - \1_{\B{C}_\infty \cap nF \cap nQ} \right\|_{\ell^1} \,.
\end{align}
For $n$ sufficiently large depending on $k$, there is $c(d) >0$ so that
\begin{align}
\left\| \1_{\B{C}_\infty \cap nP_n} - \1_{\B{C}_\infty \cap nF} \right\|_{\ell^1} &\leq \sum_{Q\, \in\, \Delta^k \setminus \textsf{Q}_0}  \left\| \1_{\B{C}_\infty \cap nP_n \cap nQ} - \1_{\B{C}_\infty \cap nF \cap nQ} \right\|_{\ell^1} + c(d) (\gamma + \per(F)) n^d 2^{-dk} \,.
\end{align}
Define the following collections of dyadic cubes:
\begin{align}
\textsf{Q}_1 &:= \Big\{ Q \in \Delta^k : P_n \cap Q = Q \Big\} \,, \\
\textsf{Q}_2 &:= \Big\{ Q \in \Delta^k : F \cap Q = Q \Big\} \,, 
\end{align}
For each $Q \in \textsf{Q}_1$, we have $\B{C}_\infty \cap nP_n \cap nQ = \B{C}_\infty \cap nQ$. Likewise, for each $Q \in \textsf{Q}_2$, we have $\B{C}_\infty \cap nF \cap nQ = \B{C}_\infty \cap nQ$. Using these observations, we conclude
\begin{align}
\left\| \1_{\B{C}_\infty \cap nP_n} - \1_{\B{C}_\infty \cap nF} \right\|_{\ell^1} &\leq \sum_{Q\, \in\, \textsf{Q}_1 \Delta \textsf{Q}_2}  \left\| \1_{\B{C}_\infty \cap nP_n \cap nQ} - \1_{\B{C}_\infty \cap nF \cap nQ} \right\|_{\ell^1} + c(d) (\gamma + \per(F)) n^d 2^{-dk} \,.
\end{align}
For $\e > 0$ and for each $Q \in \Delta^k$, introduce the event 
\begin{align}
\cal{E}_Q := \left\{ \frac{\B{C}_\infty \cap nQ}{\cal{L}^d (nQ) } \in \left( \theta_p(d)(1-\e), \theta_p(d)(1+\e) \right) \right\} \,,
\end{align}
and let $\cal{E}_0$ be the intersection $\bigcap_{Q \in \Delta^k} \cal{E}_Q$. Within $\cal{E}_0$, 
\begin{align}
\left\| \1_{\B{C}_\infty \cap nP_n} - \1_{\B{C}_\infty \cap nF} \right\|_{\ell^1} &\leq n^d 2^k |\Delta^k| (1+\e) \dw(\nu_n, \nu_F) + c(d) (\gamma + \per(F)) n^d 2^{-dk} \,.
\end{align}
The term $2^k |\Delta^k|$ above comes from the definition of $\dw$. Choose $k$ sufficiently large depending on $d,\delta$ and $F$ so that $\delta/ 4 \leq c(d) (\gamma + \per(F)) 2^{-dk} \leq \delta /2$. When $\dw(\nu_n, \nu_F) < \e$, 
\begin{align}
\left\| \1_{\B{C}_\infty \cap nP_n} - \1_{\B{C}_\infty \cap nF} \right\|_{\ell^1} &\leq n^d 2^k |\Delta^k| (1+\e) \dw(\nu_n, \nu_F) + \frac{\delta}{2}n^d \,, \\
&\leq c(d) \delta^{-(d+1)/d} (\gamma + \per(F))^{(d+1)/d} (1+\e) \e n^d + \frac{\delta}{2} n^d \,, \\
&\leq \delta n^d \,,
\end{align}
where to obtain the last line, we choose $\e$ small depending on $d, \delta$ and $F$. We complete the proof by a union bound and Corollary \ref{density_control} applied to each $\cal{E}_Q$. \end{proof}

As outlined in Remark \ref{rmk:great_explain}, Proposition \ref{volume} below will be used with Proposition \ref{prop:closeness} and Corollary \ref{trunc_cons} to control the total size of all $\slice_j^\pm$ constructed. 

\begin{prop} Let $G_n \in \cal{G}_n$, let $F \subset [-1,1]^d$ be a polytope and let $B(x,r)$ be a ball with $r \in (0,1)$ and
\begin{align}
\cal{L}^d\Big( (B(x,r) \cap F) \,\, \Delta \,\, B_-(x,r) \Big) \leq \delta \al_d r^d \,.
\label{eq:8_volume_2}
\end{align}
Within the event
\begin{align}
 \Big\{ \left\| \1_{\B{C}_\infty \cap nP_n} - \1_{\B{C}_\infty \cap nF} \right\|_{\ell^1} \leq \delta n^d \al_d r^d \Big\} \cap \Big\{ \max_{G_n \in \cal{G}_n} | \pa^\om G_n | \leq \eta_3 n^{d-1} \Big\} \,,
\label{eq:8_volume_1}
\end{align}
there is  $c(d) >0$ so that for $n$ sufficiently large depending on $d,\delta,F$ and $r$, 
\begin{align}
\sum_{j=1}^{\ell^+} |\Lambda_j^+| + \sum_{j=1}^{\ell^-} |\Lambda_j^-|  \leq c(d) \delta n^d \al_d r^d \,,
\label{eq:8_volume_3}
\end{align}
where the $\Lambda_j^\pm$ are the outward and inward components associated to $G_n$ and $B(x,r)$.

\label{volume}
\end{prop}

\begin{proof} Write $B(x,r)$ as $B$ and $B_-(x,r)$ as $B_-$ for brevity. We first handle the outward components. If $\Lambda_j^+$ is outward, it is contained in $G_n \cap n(B \setminus B_-)$.  These components are pairwise disjoint, and hence by Lemma \ref{case1},
\begin{align}
\sum_{j=1}^{\ell^+} | \Lambda_j^+| &\leq | G_n \cap n(B \setminus B_-) | \, \\
&\leq | \B{C}_\infty \cap nP_n \cap n(B \setminus B_-) | \, \\
&\leq \left\| \1_{\B{C}_\infty \cap nP_n} - \1_{\B{C}_\infty \cap nF} \right\|_{\ell^1} + | \B{C}_\infty \cap n F \cap n(B \setminus B_-) | \,.
\label{eq:8_vol_1}
\end{align}
Take $n$ sufficiently large depending on $d, F$ and $r$, to obtain
\begin{align}
\sum_{j=1}^{\ell^+} |\Lambda_j^+| &\leq \left\| \1_{\B{C}_\infty \cap nP_n} - \1_{\B{C}_\infty \cap nF} \right\|_{\ell^1} + c(d) n^d \cal{L}^d((B \cap F)\, \Delta\, B_-)\,.
\label{eq:8_vol_2}
\end{align}
Within the event \eqref{eq:8_volume_1} and using the bound \eqref{eq:8_volume_2}, we find
\begin{align}
\sum_{j=1}^{\ell^+} |\Lambda_j^+| &\leq \delta n^d \al_dr^d + c(d) \delta n^d \al_d r^d \,.
\label{eq:8_vol_3}
\end{align}

The inward components are also pairwise disjoint, and each is contained within $(\B{C}_\infty \setminus G_n) \cap nB_-$:
\begin{align}
\sum_{j=1}^{\ell^-} | \Lambda_j^- | &\leq | (\B{C}_\infty \setminus G_n) \cap nB_- | \, \\
&\leq | (\B{C}_\infty \cap nP_n) \setminus G_n | + | (\B{C}_\infty \setminus nP_n) \cap nB_- | \,\\
&\leq | (\B{C}_\infty \cap nP_n)  \setminus G_n| + \left\| \1_{\B{C}_\infty \cap nP_n } - \1_{\B{C}_\infty \cap nF} \right\|_{\ell^1} + | (\B{C}_\infty \setminus nF) \cap nB_- | \,.
\end{align}
The second line above follows from Lemma~\ref{case1}. The discrete precursor to $P_n$ was  $F_n \subset \B{C}_\infty$, defined in \eqref{eq:6.2_F}. The polytope $F$ in the statement of this proposition is \emph{not related} to this $F_n$ (this is the only instance the letter $F$ is overloaded with meaning). Use the definition of $F_n$ and Lemma~\ref{case1}:
\begin{align}
\sum_{j=1}^{\ell^-} | \Lambda_j^- | &\leq | F_n  \setminus G_n| + \left\| \1_{\B{C}_\infty \cap nP_n } - \1_{\B{C}_\infty \cap nF} \right\|_{\ell^1} + | (\B{C}_\infty \setminus nF) \cap nB_- | \,, \\
&\leq | F_n \setminus G_n| + \delta n^d \al_d r^d + c(d) n^d \cal{L}^d ( (B \cap F)\,  \Delta\, B_-) \,,
\label{eq:8_vol_4}
\end{align}
when $n$ is taken sufficiently large depending on $r,F$ and $d$. Within the event \eqref{eq:8_volume_1}, Lemma~\ref{adding_k_small} implies $|F_n \setminus G_n| \leq n^d n^{-\e(d)}$, where $\e(d)$ is defined in \eqref{eq:epsilon_d}. All that matters is that $\e(d) >0$, which is the case when $d \geq 3$, and we use this in \eqref{eq:8_vol_4} to deduce:
\begin{align}
\sum_{j=1}^{\ell^-} | \Lambda_j^- | &\leq n^d n^{-\e(d)} + \delta n^d \al_d r^d + c(d) \delta n^d \al_d r^d  \,.
\label{eq:8_vol_5}
\end{align}
We complete the proof taking $n$ larger if necessary, and using \eqref{eq:8_vol_5} with \eqref{eq:8_vol_3}. \end{proof}

\begin{rmk} We can now bound $|\slice_j^\pm|$ by combining \eqref{eq:8_coro_trunc_cons} of Corollary \ref{trunc_cons} and \eqref{eq:8_volume_3} of Proposition \ref{volume}. As Figure \ref{fig:final_4} suggests, the $\slice_j^\pm$ together with the edges of $\pa^\om G_n$ lying in the thin cylinder $n \cyl(D(x,r'), hr')$ form a cutset separating the faces of this cylinder. We leverage this in the next subsection.
\label{rmk:cut_hints}
\end{rmk}

Motivated by Remark \ref{rmk:cut_hints}, introduce the following edge set depending on $G_n \in \cal{G}_n$, $h>0$ and the ball $B(x,r)$.
\begin{align}
E_n := \Big( \pa^\om G_n \cap n B(x,r) \Big) \cup \left( \bigcup_{j=1}^{\ell^+} \slice_j^+ \right) \cup \left(  \bigcup_{j=1}^{\ell^-} \slice_j^- \right) \,.
\label{eq:8_E_cut}
\end{align}

\subsection{Lower bounds on $|\pa^\om G_n|$} Given a collection of balls $\{ B(x_i, r_i) \}_{i=1}^m$ from Lemma \ref{covering}, a polytope $F$, $h >0$ and $G_n \in \cal{G}_n$, repeat the construction of the previous subsection within each $B(x_i,r_i)$. For $G_n \in \cal{G}_n$, $h >0$ and each $B(x_i,r_i)$, define the edge set $E_n^{(i)}$ as in \eqref{eq:8_E_cut}.

For these objects and the parameters $\delta >0$ and $s \in (0,1/2)$, define an event whose purpose is described in Remark \ref{rmk:F_purpose}:
\begin{align}
\cal{F}(G_n,i;h) := \left\{ \left|E_n^{(i)}\right| \leq \left( 1-s +  c(p,d) \frac{\delta}{h} \right) n^{d-1} \al_{d-1} (r_i)^{d-1} \beta_{p,d}(v_i) \right\} \,.
\label{eq:8_event_2}
\end{align}
The constant $c(p,d) >0$ is not specified here, as it arises naturally in the proof of Corollary~\ref{good_bound} below. It comes from the constants in \eqref{eq:8_volume_3} and \eqref{eq:8_coro_trunc_cons}, the ratio $\al_{d-1} / \al_d$ and the extreme values of $\beta_{p,d}$ over the unit sphere.

\begin{rmk} The bounds on $|\slice^\pm|$ from the previous section also control each $|E_n^{(i)}|$. We use these bounds with concentration estimates when the $E_n^{(i)}$ form a cutset to show $\cal{F}(G_n,i;h)$ is rare when $\delta,h$ and $s$ are chosen appropriately. Corollary \ref{good_bound} below relates the events $\cal{E}(G_n,i)$ introduced at the beginning of the section to the $\cal{F}(G_n,i;h)$. By Lemma \ref{9.1}, knowing each $\cal{F}(G_n,i;h)$ is a low-probability event tells us it is also rare for $|\pa^\om G_n|$ to be too small.  
\label{rmk:F_purpose}
\end{rmk}

\begin{coro} Let $G_n \in \cal{G}_n$, let $F \subset [-1,1]^d$ be a polytope and let $\{ B(x_i ,r_i) \}_{i=1}^m$ be a collection of balls as in Lemma \ref{covering} for $F$ and the paramters $\delta >0 , s \in (0,1/2)$. Let $h >0$, and form the edge sets $E_n^{(i)}$. For $n$ sufficiently large depending on $d, \e_F$ and $F$,
\begin{align}
\cal{E}(G_n,i) \cap  \Big\{ \left\| \1_{\B{C}_\infty \cap nP_n} - \1_{\B{C}_\infty \cap nF} \right\|_{\ell^1} \leq \e_F n^d \Big\} \cap \Big\{ \max_{G_n \in \cal{G}_n} | \pa^\om G_n | \leq \eta_3 n^{d-1} \Big\} \subset \cal{F}(G_n,i;h) \,,
\label{eq:8_good_bound}
\end{align}
where $\cal{E}(G_n,i)$ and $\cal{F}(G_n,i;h)$ are events respectively defined in \eqref{eq:8_event} and \eqref{eq:8_event_2}, and where $\e_F = \delta \min_{i=1}^m (r_i)^d \al_d$ was defined in \eqref{eq:8_epsilon}.
\label{good_bound}
\end{coro}

\begin{proof} For the convenience of the reader, we recall the definition of the event $\cal{E}(G_n,i)$:
\begin{align}
\cal{E}(G_n,i) = \Big\{ | \pa^\om G_n \cap nB(x_i, r_i) | \leq (1-s) n^{d-1} \al_{d-1} (r_i)^{d-1} \beta_{p,d}(v_i) \Big\}\,.
\end{align}
Working within the event on the left-hand side of \eqref{eq:8_good_bound}, and as each $B(x_i,r_i)$ from Lemma~\ref{covering} satisfies 
\begin{align}
\cal{L}^d( (B(x_i,r_i) \cap F) \, \Delta \, B_-(x_i,r_i) ) \leq \delta \al_d (r_i)^d \,,
\end{align}
we apply Proposition \ref{volume} within each ball (with $n$ taken sufficiently large), obtaining 
\begin{align}
\sum_{j=1}^{\ell^+} |\Lambda_j^+(i)| + \sum_{j=1}^{\ell^-} |\Lambda_j^-(i)|  \leq c(d) \delta n^d \al_d (r_i)^d \,,
\end{align}
where the $\Lambda_j^\pm(i)$ are the inward and outward components corresponding to $G_n, h >0$ and the ball $B(x_i,r_i)$. Form the edge sets $\slice_j^\pm(i)$ are defined as in \eqref{eq:8_slice} and apply Corollary \ref{trunc_cons}:
\begin{align}
\sum_{j=1}^{\ell^+} |\slice_j^+(i)| + \sum_{j=1}^{\ell^-} |\slice_j^-(i)| &\leq \frac{c(d)}{nhr_i } \Big( c(d) \delta n^d \al_d (r_i)^d\Big) \,, \\
&\leq c(d) \frac{\delta}{h} n^{d-1} \al_{d-1} (r_i)^{d-1} \,,
\end{align}
Use the definitions of $\cal{E}(G_n,i)$ and the $E_n^{(i)}$: within the event on the left-hand side of \eqref{eq:8_good_bound}, 
\begin{align}
|E_n^{(i)} | \leq (1-s) n^{d-1} \al_{d-1} (r_i)^{d-1} \beta_{p,d}(v_i) + c(d) \frac{\delta}{h} n^{d-1} \al_{d-1} (r_i)^{d-1} \,.
\end{align}
The proof is complete upon defining $\cal{F}(G_n,i ;h)$ appropriately in \eqref{eq:8_event_2}.\end{proof}

The next lemma tells us each $E_n^{(i)}$ forms an open cutset with high probability.

\begin{lem} Let $F$ be a polytope, let $h >0$ and let $\{B(x_i, r_i)\}_{i=1}^m$ be a collection of balls as in Lemma \ref{covering} for $F$, $\delta >0$ and $s \in (0,1/2)$. Let $\cal{E}_1$ be the event that for each $G_n \in \cal{G}_n$ and all $i \in \{1,\dots, m\}$, any open path in $\dcyl( D(x,r_i'), hr_i', n)$ joining the faces $\dface^\pm ( D(x,r_i'), hr_i', n)$ uses an edge of $E_n^{(i)}$. There are positive constants $c_1, c_2$ depending on $p,d,F, \delta,s,h$ so that 
\begin{align}
\prob_p( \cal{E}_1) \geq 1 - c_1 \exp \Big( - c_2 n^{(d-1)/d} \Big) \,,
\end{align}
where we recall $r_i' := (1-h^2) r_i^2$.
\label{E_cuts}
\end{lem}

\begin{proof} Our primary tool is Theorem \ref{kzgm}. We drop the indexing for the sake of clarity and work with generic objects: balls $B(x,r)$, discs $D(x,r')$ and edge sets $E_n$. 

From the careful construction of $\slice_j^\pm$, any open path in $\B{C}_\infty$ between faces $\dface^\pm ( D(x,r'), hr', n)$ in $\dcyl( D(x,r'), hr', n)$ uses an edge of $E_n$. In the almost sure event that there is a unique infinite cluster, $\dface^\pm ( D(x,r'), hr', n)$ can only be joined by an open path in $\dcyl( D(x,r'), hr', n)$ if this path lies in a finite open cluster. Such a path uses at least $2r'hn$ edges, and the cluster containing this path must have volume at least $2r'hn$. A union bound with Theorem \ref{kzgm} applied to each point in $[-n,n]^d \cap \Z^d$ gives the desired result.\end{proof}

\begin{rmk} Let $\cal{E}_1$ be the event from Lemma \ref{E_cuts}. For each $\om \in \cal{E}_1$, completing each $E_n^{(i)}$ to a full cutset in $\dcyl( D(x,r'), hr', n)$ implies $| E_n^{(i)} | \geq \face( D(x_i, r_i'), hr_i',n)$ in $\om$. The next proposition aggregates all work done in this section. 
\label{rmk:8_observation}
\end{rmk}

\begin{prop} Let $F \subset [-1,1]^d$ be a polytope, and for $s \in (0,1/2)$, let $\lambda_F(s) = (1-2s) \cal{I}_{p,d}(F)$. There are positive constants $c_1(p,d,s,F)$, $c_2(p,d,s,F)$ and $\wt{\e}_F(p,d,s,F)$ so that 
\begin{align}
\prob_p \Big( \Big\{ \exists G_n \in \cal{G}_n \text{ such that } | \pa^\om G_n| \leq \lambda_F(s) n^{d-1} \text{ and } \dw( \mu_n, \nu_F) \leq \wt{\e}_F \Big\} \Big) \leq c_1 \exp \left( -c_2 n^{1/2d}  \right) \,.
\end{align}
\label{prob_conversion}
\end{prop}

\begin{proof} Let $\{B(x_i,r_i)\}_{i=1}^m$ be a collection of balls as in Lemma \ref{covering} for $F$, $\delta >0$ and $s \in (0,1/2)$. The parameter $\delta$ and a height parameter $h >0$ will be fixed as functions of $p,d$ and $s$ later. 

Let $\cal{E}_0$ be the event from Proposition \ref{prop:closeness} for parameter $\wt{\e}_F$ to be determined later. Let $\cal{E}_1$ be the event from Lemma \ref{E_cuts}, and for constants $\eta_3$ and $\gamma$ from Lemma~\ref{surface_boundary} and Corollary~\ref{finite_per_2}, define
\begin{align}
\cal{E}_2 := \Big\{ \max_{G_n \in \cal{G}_n} \dw(\mu_n, \nu_n) \leq \wt{\e}_F \Big\} \,, \hspace{5mm}\cal{E}_3 := \Big\{ \max_{G_n \in \cal{G}_n} | \pa^\om G_n | \leq \eta_3 n^{d-1} \Big\} \,,
\end{align}
\begin{align}
\cal{E}_4 &:= \Big\{ \max_{G_n \in \cal{G}_n} \per(P_n) \leq \gamma \Big \} \,,
\end{align}
Let $\cal{E}^*$ be the intersection of $\cal{E}_0$ through $\cal{E}_4$. Apply Lemma \ref{9.1} to conclude
\begin{align}
\Big\{ \exists G_n \in \cal{G}_n \text{ such that } | \pa^\om G_n| \leq \lambda_F(s) n^{d-1} &\text{ and } \dw( \mu_n, \nu_F) \leq \wt{\e}_F \Big\} \cap \cal{E}^* \\
&\subset \bigcup_{G_n \in \cal{G}_n} \bigcup_{i=1}^m \cal{E}(G_n, i) \cap \Big\{ \dw(\mu_n, \nu_F ) \leq \wt{\e}_F \Big\} \cap \cal{E}^* \,,\\
& \subset \bigcup_{G_n \in \cal{G}_n} \bigcup_{i=1}^m \cal{E}(G_n, i) \cap \Big\{ \dw(\nu_n, \nu_F ) \leq 2\wt{\e}_F \Big\} \cap \cal{E}^* \,,
\label{eq:9.67}
\end{align}
where \eqref{eq:9.67} follows from $\cal{E}^* \subset \cal{E}_2$. Use that $\cal{E}^*$ is contained in $\cal{E}_0$ and in $\cal{E}_4$ with Proposition~\ref{prop:closeness}, choosing $\wt{\e}_F$ small depending on $\e_F, d$ and $F$ so that 
\begin{align}
\Big\{ \exists G_n \in \cal{G}_n &\text{ such that } | \pa^\om G_n| \leq \lambda_F(s) n^{d-1} \text{ and } \dw( \mu_n, \nu_F) \leq \wt{\e}_F \Big\} \cap \cal{E}^* \\
& \subset \bigcup_{G_n \in \cal{G}_n} \bigcup_{i=1}^m \cal{E}(G_n, i) \cap \Big\{ \big\| \1_{\B{C}_\infty \cap n P_n} - \1_{\B{C}_\infty \cap nF} \big\|_{\ell^1} \leq \e_F n^d \Big\} \cap \cal{E}^* \,,\\
&\subset \bigcup_{G_n \in \cal{G}_n} \bigcup_{i=1}^m \cal{F}(G_n,i ;h) \cap \cal{E}^* \,,
\end{align}
where we have used Corollary \ref{good_bound} and taken $n$ large depending on $d,F$ and $\e_F$, using $\cal{E}^* \subset \cal{E}_3$. 

Finally, as $\cal{E}^*$ contains $\cal{E}_1$, we use Remark~\ref{rmk:8_observation} to conclude 
\begin{align}
\Big\{ &\exists G_n \in \cal{G}_n \text{ such that } | \pa^\om G_n| \leq \lambda_F(s) n^{d-1} \text{ and } \dw( \mu_n, \nu_F) \leq \wt{\e}_F \Big\} \cap \cal{E}^* \\
&\subset \bigcup_{G_n \in \cal{G}_n} \bigcup_{i=1}^m \Big\{  \face( D(x_i, r_i'), hr_i',n) \leq \Big(1-s + c(p,d) \frac{\delta}{h} \Big) n^{d-1} \al_{d-1} (r_i)^{d-1} \beta_{p,d}(v_i) \Big\}  \,, \\
&\subset \bigcup_{i=1}^m \Big\{  \face( D(x_i, r_i'), hr_i',n) \leq \Big(1-s + c(p,d) \frac{\delta}{h} \Big) n^{d-1} \al_{d-1} (r_i)^{d-1} \beta_{p,d}(v_i) \Big\}  \,, \\
&\subset \bigcup_{i=1}^m \Big\{  \face( D(x_i, r_i'), hr_i',n) \leq \Big(1-s + c(p,d) \frac{\delta}{h} \Big) \frac{1}{(1-h^2)^{(d-1)/2}} n^{d-1} \al_{d-1} (r_i')^{d-1} \beta_{p,d}(v_i) \Big\}  \,.
\end{align}

We now calibrate parameters. Choose $h$ small depending on $p,d$ and $s/2$ so that the concentration estimates of Proposition \ref{disc} are applicable when $n$ is taken large depending on $h$ and $\e_F$. Next, choose $\delta$ depending on $s, c(p,d)$ and $h$ so that 
\begin{align}
\Big\{ \exists G_n \in \cal{G}_n &\text{ such that } | \pa^\om G_n| \leq \lambda_F(s) n^{d-1} \text{ and } \dw( \mu_n, \nu_F) \leq \wt{\e}_F \Big\} \cap \cal{E}^* \\
&\subset \bigcup_{i=1}^m \Big\{  \face( D(x_i, r_i'), hr_i',n) \leq \Big(1-s/2) n^{d-1} \al_{d-1} (r_i')^{d-1} \beta_{p,d}(v_i) \Big\}  \,,
\end{align}
noting that the number $m$ of events in the above union now depends only on $p,d,s$ and $F$. By using Proposition \ref{disc}, we find
\begin{align}
\prob_p \Big(\exists G_n \in \cal{G}_n &\text{ such that } | \pa^\om G_n| \leq \lambda_F(s) n^{d-1} \text{ and } \dw( \mu_n, \nu_F) \leq \wt{\e}_F \Big) \\
&\leq \sum_{i=1}^m c_1 \exp\Big(-c_2 n^{(d-1)/3}\Big) + \prob_p( (\cal{E}^*)^c) \,,\\
&\leq c_1 \exp \Big(-c_2 n^{(d-1)/3}\Big) + \prob_p( (\cal{E}^*)^c)\,,
\end{align}
where $c_1$ and $c_2$ are positive constants depending on $p,d,s,h,\delta$ and $F$. As $\wt{\e}_F$, $\delta$ and $h$ all depend only on $p,d,s$ and $F$, these constants have the correct dependencies, and $\prob_p( (\cal{E}^*)^c)$ is also bounded satisfactorily. Use Proposition~\ref{prop:closeness} and Lemma~\ref{E_cuts} to bound $\prob_p(\cal{E}_0^c)$ and $\prob_p(\cal{E}_1^c)$. We further use Theorem~\ref{emp_close}, Lemma~\ref{surface_boundary} and Corollary~\ref{finite_per_2} to bound $\prob_p( \cal{E}_2^c), \prob_p(\cal{E}_3^c)$ and $\prob_p( \cal{E}_4^c)$ respectively, concluding
\begin{align}
\prob_p \Big(\exists G_n \in \cal{G}_n &\text{ such that } | \pa^\om G_n| \leq \lambda_F(s) n^{d-1} \text{ and } \dw( \mu_n, \nu_F) \leq \wt{\e}_F \Big) \leq \sum_{i=1}^m c_1 \exp\Big(-c_2 n^{(1)/2d}\Big) \,,
\end{align}
for $c_1$ and $c_2$ positive constants depending on $p,d,s,F$.  \end{proof}

\begin{rmk} We assert that Proposition \ref{prob_conversion} holds also when $F$ is a translate of the Wulff crystal $W_{p,d}$; this follows from Theorem \ref{poly_approx} for instance. 
\label{rmk:8_observation_2}
\end{rmk}

\begin{rmk} Proposition \ref{prob_conversion} is the result we have been aiming for since the beginning of the section: when $G_n \in \cal{G}_n$ is such that $\dw(\mu_n,\nu_F)$ is small, we have high probability lower bounds on $| \pa^\om G_n|$ in terms of $\cal{I}_{p,d}(F)$. We use Proposition~\ref{prob_conversion} with a compactness argument to prove the main results of the paper.
\end{rmk}

 \subsection{Proof of main results}
 
 We first prove Theorem \ref{main}, from which we deduce Theorem~\ref{main_benj} and Theorem~\ref{main_L1}. A quantitative version of the isoperimetric inequality for $\cal{I}_{p,d}$ is central. Given $F \subset \R^d$ a set of finite perimeter, define the \emph{asymmetry index} of $F$ as
 \begin{align}
 A(F) := \inf \left\{ \frac{ \cal{L}^d( F\, \Delta\, (x + r W_{p,d}) ) }{ \cal{L}^d(F) } : x \in \R^d, \cal{L}^d( rW_{p,d}) = \cal{L}^d(F) \right\} \,.
 \label{eq:asymmetry}
 \end{align}
For $r >0$ chosen to make $rW_{p,d}$ and $F$ equal in volume, define the \emph{isoperimetric deficit} of $F$ as
\begin{align}
D(F) := \frac{\cal{I}_{p,d} (F) - \cal{I}_{p,d}(rW_{p,d} ) }{\cal{I}_{p,d}(rW_{p,d}) } \,.
\end{align}
The isoperimetric inequality implies $D(F) \geq 0$ for all sets $F$ of finite perimeter, while Taylor's theorem (Theorem \ref{wulff_theorem}) implies $D(F) = 0$ if and only if $A(F) =0$. The next result  quantifies this. 

\begin{thm} (Figalli-Maggi-Pratelli \cite{Figalli_Maggi_Pratelli})\, Let $F \subset \R^d$ be a set of finite perimeter with finite volume. There is $c(d)>0$ so that 
\begin{align}
A(F) \leq c(d) D(F)^{1/2} \,.
\end{align}
\label{fmp}
\end{thm}

\begin{rmk} It follows from Theorem \ref{fmp} that whenever $rW_{p,d}$ is a dilate of the Wulff crystal, and whenever $F^r$ is a set of finite perimeter such that $\cal{L}^d(F^r) = \cal{L}^d(rW_{p,d})$, we have
\begin{align}
\frac{ \cal{I}_{p,d} (F^r) }{ \cal{I}_{p,d} (rW_{p,d} )} \geq 1 + c(d) (A(F^r))^2  \,.
\label{eq:crucial}
\end{align}
\end{rmk}

The next theorem boosts results for polytopes to results for sets of finite perimeter. 

\begin{thm} (\cite{stflour}, Proposition 14.9)\, Let $F \subset [-1,1]^d$ be a set of finite perimeter. There is a sequence of polytopes $\{ F_n \}_{n=1}^\infty$, each contained within $[-1,1]^d$, so that $\cal{L}^d( F \Delta F_n) \to 0$ and $| \cal{I}_{p,d}(F_n) - \cal{I}_{p,d}(F) | \to 0$ as $n \to \infty$.  
\label{poly_approx}
\end{thm}

\n \B{\emph{Proof of Theorem \ref{main} (Precursor to shape theorem).}} Throughout the proof, write $\theta$ for $\theta_p(d)$. Let $\xi > 0$, define $\eta = \eta(\xi)$ via the relation
\begin{align}
(1- \eta) = \frac{1}{1+\xi} \,,
\label{eq:eta}
\end{align}
and use $\xi$ and $\eta$ to define the following collection of measures:
\begin{align}
\cal{W}_{\xi} := \left\{ \nu_{W+x} : \begin{matrix}  x \in \R^d, (W+x) \subset [-1,1]^d \text{ and } W \text{ is a dilate of } W_{p,d} \\ \text{ such that } \cal{L}^d( (1-\eta) W_{p,d} ) \leq \cal{L}^d(W) \leq \cal{L}^d( (1+ 2\xi) W_{p,d} )\end{matrix} \right\} \,.
\end{align}
Let $\zeta >0$, and choose $\xi = \xi(\zeta) >0$ and $\e = \e(\zeta, \xi) >0$ so that the following relations hold:
\begin{align}
\frac{1}{1 + \xi} = 1- 2\zeta \hspace{20mm} \frac{1 + 2\xi}{  1 + c(p,d) \e^2 } = 1-2\zeta \,,
 \label{eq:xi_dependence}
\end{align}
where $c(p,d)$ is specified later. We remark that $\e$ is distinct from $\e(d)$ defined in (\ref{eq:epsilon_d}); this latter fixed value only appears in exponents of various upper bounds, and we always explicate the dependence on $d$ in this case.  Our principal aim is to show the probabilities
\begin{align}
\prob_p \Big( \exists G_n \in \cal{G}_n \text{ such that } \dw(\mu_n , \cal{W}_{\xi} ) \geq \e \Big)
\label{eq:target}
\end{align}
decay rapidly with $n$. Recall that 
\begin{align}
\cal{P}_{\gamma,\,\xi} = \left\{ \nu_F : F \subset [-1,1]^d,\, \per(F) \leq \gamma,\, \cal{L}^d(F) \leq \cal{L}^d( (1+ \xi) W_{p,d})\right\} \,,
\end{align}
let $\cal{V}_\e( \cal{W}_{\xi} )$ be the open $\e$-neighborhood of $\cal{W}_{\xi}$ in the metric $\dw$, and let $\cal{K}_{\gamma,\xi}(\e)$ be the complement of this neighborhood in $\cal{P}_{\gamma,\,\xi}$. By Lemma \ref{compact}, $\cal{K}_{\gamma,\xi}(\e)$ is compact. Define
\begin{align}
\cal{P}oly_{\gamma,\xi} := \left\{ \nu_F : F \subset [-1,1]^d \text{ is a polytope},\, \per(F) \leq 2\gamma,\, \cal{L}^d(F) \leq \cal{L}^d( (1+ 2\xi) W_{p,d})\right\} \,,
\end{align}
so that using the definition of $\dw$ and Theorem \ref{poly_approx}, the $\dw$-balls
\begin{align}
\Big\{ \cal{B}(\nu_F, \wt{\e}_F/2 ) \Big\}_{F \in \cal{P}oly_{\gamma,\xi}}
\label{eq:8_final_cover}
\end{align}
form an open cover of $\cal{K}_{\gamma,\xi}(\e)$, where $\wt{\e}_F$ is chosen as in Proposition \ref{prob_conversion} for $F$ and $s= \zeta/2$. For a parameter $\delta' >0$ to be used shortly, we lose no generality choosing $\wt{\e}_F$ smaller if necessary so that
\begin{align}
\label{eq:delta_prime}
(1 + \wt{\e}_F /\theta ) &\leq (1+ \delta') \,, \\
\wt{\e}_F &\leq \e/2
\label{eq:ep_comparison}
\end{align}
hold for each $F$. Given $F \in\cal{P}oly_{\gamma,\xi}$, define $\lambda_F(\zeta) := (1- \zeta) \cal{I}_{p,d}(F)$ and use the compactness of $\cal{K}_{\gamma,\xi}(\e)$ to extract a finite subcover from \eqref{eq:8_final_cover}: there are polytopes $F_1, \dots, F_m$ such that 
\begin{align}
\Big\{ \cal{B}(\nu_{F_j}, \wt{\e}_{F_j}/2 ) \Big\}_{j=1}^m
\label{eq:8_final_cover_2}
\end{align}
covers $\cal{K}_{\gamma,\xi}(\e)$. We now begin to estimate (\ref{eq:target}).
\begin{align}
\prob_p \Big( \exists &G_n \in \cal{G}_n \text{ such that } \dw(\mu_n , \cal{W}_{\xi} ) \geq \e \Big) \\
&\leq\prob_p \left( \max_{G_n \in \cal{G}_n} \dw(\mu_n, \cal{W}_{\xi}) \geq \e \text{ and } n \Chee \leq (1 + \delta') \vp_{W_{p,d} }\right) + \prob_p\left( n \Chee > (1 + \delta') \vp_{W_{p,d}} \right) \\
&\leq \prob_p \left( \max_{G_n \in \cal{G}_n} \dw(\mu_n, \cal{W}_{\xi}) \geq \e \text{ and } n \Chee \leq (1 + \delta') \vp_{W_{p,d} }\right) + c_1 \exp \left(-c_2 n^{(d-1)/3} \right)
\end{align}
Where we have used bounds from the proof of Corollary \ref{upper_bound_2}, and we recall that $\vp_{W_{p,d}}$ is the conductance (defined at the very end of Section \ref{sec:norm}) of the Wulff crystal. Choose $\delta >0$ so that
\begin{align}
\delta \leq \min_{j=1}^m \wt{\e}_{F_j} /2 \,,
\label{eq:delta}
\end{align}
and invoke Corollary \ref{close_finite} for $\delta$ to further deduce
\begin{align}
\prob_p \Big( \exists &G_n \in \cal{G}_n \text{ such that } \dw(\mu_n , \cal{W}_{\xi} ) \geq \e \Big) \\
&\leq \prob_p \left( \max_{G_n \in \cal{G}_n} \dw(\mu_n, \cal{W}_{\xi}) \geq \e \text{ and } \max_{G_n \in \cal{G}_n} \dw(\mu_n, \cal{P}_{\gamma,\,\xi}) < \delta \text{ and } n \Chee \leq (1 + \delta') \vp_{W_{p,d} }\right)\\
&\hspace{20mm} + c_1 \exp \left(-c_2 n^{(1-\e(d))/2d} \right) \,.
\end{align}
Use the finite open cover \eqref{eq:8_final_cover_2}, the choice of $\delta$ and a union bound:
\begin{align}
\prob_p \Big( \exists &G_n \in \cal{G}_n \text{ such that } \dw(\mu_n , \cal{W}_{\xi} ) \geq \e \Big) \\
&\leq \sum_{i=1}^m \prob_p \left( \exists G_n \in \cal{G}_n \text{ such that } \dw(\mu_n, \nu_{F_j}) \leq \e_{F_j} \text{ and } n \Chee \leq (1 + \delta') \vp_{W_{p,d} }\right)\\
&\hspace{20mm} + c_1 \exp \left(-c_2 n^{(1-\e(d))/2d} \right) \,.
\label{eq:final_union}
\end{align}
We focus on bounding each summand of the form $\prob_p(\cal{F}_j)$ above, where 
\begin{align}
\cal{F}_j : =\left\{ \exists G_n \in \cal{G}_n \text{ such that } \dw(\mu_n, \nu_{F_j}) \leq \e_{F_j} \text{ and } n \Chee \leq (1 + \delta') \vp_{W_{p,d} }\right\}  \,.
\end{align}
We begin by unravelling the Cheeger constant and using \eqref{eq:delta_prime}.
\begin{align}
\prob_p( \cal{F}_j) &=\prob_p \left( \begin{matrix} \exists G_n \in \cal{G}_n \text{ such that } \dw(\mu_n, \nu_{F_j}) \leq \wt{\e}_{F_j} \\  \text{ and } \\ n |\pa^\om G_n|  \leq (1 + \delta') | G_n |\vp_{W_{p,d} } \end{matrix}\right) \, \\
\label{eq:9.107}
&\leq \prob_p \left( \begin{matrix} \exists G_n \in \cal{G}_n \text{ such that } \dw(\mu_n, \nu_{F_j}) \leq \wt{\e}_{F_j}\\ \text{ and } \\ n |\pa^\om G_n|  \leq (1 + \delta') n^d (\theta\cal{L}^d(F_j) + \wt{\e}_{F_j} ) \vp_{W_{p,d} } \end{matrix}\right) \,\\
&\leq \prob_p \left(\begin{matrix} \exists G_n \in \cal{G}_n \text{ such that } \dw(\mu_n, \nu_{F_j}) \leq \wt{\e}_{F_j}\\ \text{ and }\\  |\pa^\om G_n|  \leq (1 + \delta')^2 n^{d-1} \theta\cal{L}^d(F_j) \vp_{W_{p,d} }\end{matrix}\right)  \,.
\label{eq:9.108}
\end{align}
To obtain \eqref{eq:9.107}, we used the definition of $\dw$, and to obtain \eqref{eq:9.108} we used \eqref{eq:delta_prime}. Observe that
\begin{align}
\prob_p( \cal{F}_j) &\leq \prob_p \left( \begin{matrix}\exists G_n \in \cal{G}_n \text{ such that } \dw(\mu_n, \nu_{F_j}) \leq \wt{\e}_{F_j}\\ \text{ and }\\  |\pa^\om G_n|  \leq (1 + \delta')^2 n^{d-1} \cal{I}_{p,d} (F_j)( \vp_{F_j})^{-1} \vp_{W_{p,d} } \end{matrix} \right) \, \\
&\leq \prob_p \left(\begin{matrix} \exists G_n \in \cal{G}_n \text{ such that } \dw(\mu_n, \nu_{F_j}) \leq \wt{\e}_{F_j}\\ \text{ and }\\  |\pa^\om G_n|  \leq (1 + \delta')^2 \frac{\cal{I}_{p,d}(rW_{p,d})}{\cal{I}_{p,d}(F_j)} r n^{d-1} \cal{I}_{p,d} (F_j) \end{matrix}\right) \,,
\label{eq:cases_final}
\end{align}
where $r >0$ is chosen so that $\cal{L}^d(F_j) = \cal{L}^d(rW_{p,d})$. Form two cases. In \emph{Case (1)}, $r \leq (1-\eta)$, and in \emph{Case (2)}, $r \in (1-\eta, 1+ 2\xi]$. Focusing on the first case for now, use Theorem \ref{wulff_theorem} and the relation \eqref{eq:eta} between $\xi$ and $\eta$:
\begin{align}
\prob_p( \cal{F}_j) &\leq \prob_p \left( \begin{matrix}\exists G_n \in \cal{G}_n \text{ such that } \dw(\mu_n, \nu_{F_j}) \leq \wt{\e}_{F_j}\\ \text{ and }\\  |\pa^\om G_n|  \leq \frac{(1 + \delta')^2}{1+\xi} n^{d-1} \cal{I}_{p,d} (F_j) \end{matrix} \right) 
\end{align}
As $\xi$ was chosen as in (\ref{eq:xi_dependence}), we choose $\delta'$ small enough depending on $\xi$ and $\zeta$ so that 
\begin{align}
\prob_p ( \cal{F}_j) &\leq \prob_p \Big( \exists G_n \in \cal{G}_n \text{ such that } \dw(\mu_n, \nu_{F_j}) \leq \wt{\e}_{F_j} \text{ and }  |\pa^\om G_n| \leq \lambda_{F_j}(\zeta) n^{d-1} \Big) \,
\label{eq:amenable_conversion}
\end{align}
holds whenever we are in \emph{Case (1)}. 

We maneuver into a similar position in \emph{Case (2)}. From \eqref{eq:cases_final}, we deduce
\begin{align}
\prob_p( \cal{F}_j) &\leq \prob_p \left( \begin{matrix}\exists G_n \in \cal{G}_n \text{ such that } \dw(\mu_n, \nu_{F_j}) \leq \wt{\e}_{F_j}\\ \text{ and }\\  |\pa^\om G_n|  \leq \frac{(1 + \delta')^2}{1 +c(d) (A(F_j))^2} (1+2\xi) n^{d-1} \cal{I}_{p,d} (F_j) \end{matrix} \right) \,,
\label{eq:final_case_2}
\end{align}
where $A(F_j)$ is the asymmetry index of $F_j$ introduced in \eqref{eq:asymmetry}, and where we have used the observation in \eqref{eq:crucial}. In \eqref{eq:ep_comparison}, we chose each $\wt{\e}_{F_j}$ to be at most $\e/2$. The finite open cover \eqref{eq:8_final_cover_2} may be assumed to have no redundancies, so by the construction of $\cal{K}_{\gamma,\xi}(\e)$, 
\begin{align}
\dw(\nu_{F_j} \cal{W}_\xi) \geq \e/2
\end{align}
for each $F_j$. Using the definition \eqref{eq:2_metric} of $\dw$, we have the following lower-bound on the asymmetry index of each $F_j$:
\begin{align}
A(F_j) \geq (\e/4)  \cal{L}^d(F_j) \geq (\e/4)(1-\eta) \cal{L}^d(W_{p,d}) \,.
\label{eq:asm_lower}
\end{align}
As $\xi$ and hence $\eta$ will be taken to zero, we lose no generality supposing $\eta < 1/2$. Thus, \eqref{eq:asm_lower} and \eqref{eq:final_case_2} together yield 
\begin{align}
\prob_p( \cal{F}_j) &\leq \prob_p \left( \begin{matrix}\exists G_n \in \cal{G}_n \text{ such that } \dw(\mu_n, \nu_{F_j}) \leq \wt{\e}_{F_j}\\ \text{ and }\\  |\pa^\om G_n|  \leq \frac{(1 + \delta')^2}{1 +c(p,d) \e^2} (1+2\xi) n^{d-1} \cal{I}_{p,d} (F_j) \end{matrix} \right) \,.
\end{align}
Now use our choice of $\e$ in \eqref{eq:xi_dependence}, taking $\delta'$ sufficiently small depending on $\xi$ and $\zeta$ so that 
\begin{align}
\prob_p ( \cal{F}_j) &\leq \prob_p \Big( \exists G_n \in \cal{G}_n \text{ such that } \dw(\mu_n, \nu_{F_j}) \leq \wt{\e}_{F_j} \text{ and }  |\pa^\om G_n| \leq \lambda_{F_j}(\zeta) n^{d-1} \Big) \,
\label{eq:amenable_conversion_2}
\end{align}
holds in \emph{Case (2)} also.

Return to (\ref{eq:final_union}) and apply the bounds \eqref{eq:amenable_conversion} and \eqref{eq:amenable_conversion_2} to each summand:
\begin{align}
\prob_p \Big( \exists &G_n \in \cal{G}_n \text{ such that } \dw(\mu_n , \cal{W}_{\xi} ) \geq \e \Big) \\
&\leq \sum_{i=1}^m \prob_p \left( \exists G_n \in \cal{G}_n \text{ such that } \dw(\mu_n, \nu_{F_j}) \leq \wt{\e}_{F_j} \text{ and }  |\pa^\om G_n| \leq \lambda_{F_j}(\zeta) n^{d-1} \right)\\
&\hspace{20mm} + c_1 \exp \left(-c_2 n^{(1-\e(d))/2d} \right) \,.
\end{align}
Thus,
\begin{align}
\prob_p \Big( \exists &G_n \in \cal{G}_n \text{ such that } \dw(\mu_n , \cal{W}_{\xi} ) \geq \e\Big) \leq mc_1\exp\left( -c_2 n^{1/2d}\right) + c_1 \exp \left(-c_2 n^{(1-\e(d))/2d}\right)\,.
\label{eq:use_in_final}
\end{align}
We have used the hard-earned bounds from Proposition \ref{prob_conversion} directly above. By Borel-Cantelli,\begin{align}
\prob_p \left( \max_{G_n \in \cal{G}_n} \dw (\mu_n , \cal{W}_{\xi} ) \leq \e(\zeta,\xi) \text{ for all but finitely many $n$} \right) = 1 \,.
\end{align}
Observe that 
\begin{align}
\dw( \cal{W}_{\xi}, \cal{W} ) &\leq c(p,d) \max \left( \cal{L}^d ( W_{p,d} \setminus (1-\eta)W_{p,d} ) , \cal{L}^d ((1+2\xi)W_{p,d} \setminus W_{p,d} ) \right) \,,\\
& \leq c(p,d,\xi) \,,
\label{eq:use_in_final_2}
\end{align}
where $c(p,d,\xi)$ tends to $0$ as $\xi \to 0$. From (\ref{eq:xi_dependence}), we have that $\xi \equiv \xi(\zeta) \to 0$ as $\zeta \to 0$ and also that $\e \equiv \e(\zeta, \xi) \to 0$ as $\zeta,\xi \to 0$. Thus, 
\begin{align}
\prob_p \left( \max_{G_n \in \cal{G}_n} \dw (\mu_n , \cal{W} ) \leq c(p,d,\zeta) \text{ for all but finitely many $n$} \right) = 1
\label{eq:coup}
\end{align}
where $c(p,d,\zeta) \to 0$ as $\zeta \to 0$. This completes the proof of Theorem \ref{main}. $\hfill\qed$\newline

\n\B{\emph{Proof of Theorem \ref{main_benj} (Cheeger asymptotics).}} We first show $\cal{W}$ is compact in $\dw$ by appealing to the proof of Lemma \ref{compact}. It suffices to show that whenever $\{W_n\}_{n=1}^\infty$ is a sequence with $\nu_{W_n} \in \cal{W}$ and $\1_{W_n}$ tending to some $\1_F$ in $L^1$-sense, $F$ is a translate of $W_{p,d}$. By dominated convergence, $\cal{L}^d(F) = \cal{L}^d(W_{p,d})$. Lemma \ref{lsc} implies $\cal{I}_{p,d}(F) \leq \cal{I}_{p,d}(W_{p,d})$, and it follows from Theorem \ref{wulff_theorem} that $F$ is a translate of ~$W_{p,d}$. 

Let $\e, \zeta' >0$. For each $\nu_W \in \cal{W}$, choose $\wt{\e}_W$ as in Proposition \ref{prob_conversion} for $\zeta' = 2s$ (see Remark~\ref{rmk:8_observation_2}).  The $\dw$-balls $\cal{B}(\nu_W, \wt{\e}_W/2 )$ indexed by $\cal{W}$ are an open cover of $\cal{W}$; extract a finite collection of translates of $W_{p,d}$, enumerated $W_1, \dots, W_m$, so that 
\begin{align}
\Big\{ \cal{B}( \nu_{W_i}, \wt{\e}_{W_i} /2 ) \Big\}_{i=1}^m
\end{align}
covers $\cal{W}$. Choose $\zeta > 0$ small so that $c(p,d,\zeta)$ in (\ref{eq:coup}) is at most $\min_{i=1}^m \wt{\e}_{W_i} /2$, and work in the almost sure event from (\ref{eq:coup}). Use Proposition \ref{prob_conversion}, Remark \ref{rmk:8_observation_2} and Borel-Cantelli to conclude
\begin{align}
\prob_p \left( \liminf_{n \to \infty} \min_{G_n \in \cal{G}_n} \frac{ |\pa^\om G_n |} { n^{d-1} } \geq (1- \zeta ') \cal{I}_{p,d} (W_{p,d})  \right) = 1\,.
\label{eq:final_surface}
\end{align}
Because we are within the event from \eqref{eq:coup}, we take $\zeta$ smaller if necessary in a way depending on $p, d$ and $\zeta'$ so that 
\begin{align}
\prob_p \left( \limsup_{n \to \infty} \max_{G_n \in \cal{G}_n} \frac{ |G_n| }{n^d} \leq (1+\zeta') \theta_p(d)\cal{L}^d(W_{p,d}) \right) = 1\,.
\label{eq:final_volume}
\end{align}
Choose $\zeta'$ small depending on $\e$ so that by (\ref{eq:final_surface}) and (\ref{eq:final_volume}), 
\begin{align}
\prob_p \left( \liminf_{n \to \infty} n\Chee \geq (1-\e) \frac{ \cal{I}_{p,d}(W_{p,d} ) }{ \theta_p(d) \cal{L}^d(W_{p,d}) } \right) = 1 \,.
\end{align}
The complementary upper bound on $\Chee$ was shown in Corollary \ref{upper_bound_2}, completing the proof. $\hfill\qed$\newline
 
\n\B{\emph{Proof of Theorem \ref{main_L1} (Shape theorem).}} In \eqref{eq:5_convex_empirical}, we defined the empirical measure of a translate $W \subset [-1,1]^d$ of the Wulff crystal as:
\begin{align}
\ov{\nu}_W(n) := \frac{1}{n^d} \sum_{ x\, \in\, \B{C}_\infty \cap nW} \delta_{x/n}\,.
\end{align}
Let $\e, \e' >0$. Define $M_n := n^{-1} \Z^d \cap [-1,1]^d$, so that $| M_n | \leq (3n)^d$. By Corollary \ref{facilitate}, there are positive constants $c_1(p,d,\e')$ and $c_2(p,d,\e')$ so that 
\begin{align}
\prob_p \left( \max_{x \in M_n,\, (W_{p,d} +x) \subset [-1,1]^d}\dw\left( \ov{\nu}_{W_{p,d} +x}(n) , \nu_{W_{p,d} +x} \right) \leq \e' \right) \geq 1- c_1 \exp \Big(-c_2 n^{d-1} \Big)\,,
\label{eq:main_1}
\end{align}
and by Borel-Cantelli, the event
\begin{align}
\cal{E}_1 := \left\{ \limsup_{n \to \infty} \max_{x \in M_n,\, (W_{p,d} + x) \subset [-1,1]^d } \dw \left( \ov{\nu}_{W_{p,d} +x}(n) , \, \nu_{W_{p,d} +x} \right) \leq \e' \right\}
\end{align}
occurs almost surely. 

Choose $\zeta$ in \eqref{eq:coup} small depending on $\e'$, so that the event
\begin{align}
\cal{E}_2 := \left\{ \max_{G_n \in \cal{G}_n} \dw (\mu_n , \cal{W} ) \leq \e' \text{ for all but finitely many $n$} \right\}
\end{align}
also occurs almost surely. Take $n$ large depending on $\e'$ so that for any translate $W \subset [-1,1]^d$ of the Wulff crystal, there is $x \in M_n$ with $\dw(\nu_W , \nu_{W_{p,d} +x} ) \leq \e'$. For any $\om \in \cal{E}_1 \cap \cal{E}_2$, there is $N(\om) \in \N$ so that $n \geq N(\om)$ implies
\begin{align}
\max_{G_n \in \cal{G}_n} \min_{x \in M_n} \dw( \mu_n, \ov{\nu}_{W_{p,d} +x} (n) ) \leq 3 \e' \,.
\end{align}
Suppose the following holds for some $G_n \in \cal{G}_n$ and some $x \in M_n$:
\begin{align}
 \dw( \mu_n, \ov{\nu}_{W_{p,d} +x} (n) ) \leq 3 \e'\,.
 \label{eq:perchance_to_dream}
\end{align}
Let $k \in \N$, and let $\Delta^k \equiv \Delta^{k,d}$ denote the dyadic cubes at scale $k$ contained in $[-1,1]^d$. Define
\begin{align}
\textsf{Q} := \Big\{ Q \in \Delta^k : Q \cap \pa (W_{p,d} +x) \neq \emptyset \Big\}\,,
\end{align} 
and observe that
\begin{align}
\big\| \1_{G_n} - &\1_{ n(W_{p,d} +x ) \cap \giant } \big\|_{\ell^1}  \leq \sum_{Q\, \in\, \Delta^k} \big\| \1_{G_n \cap nQ} - \1_{ n(W_{p,d} +x ) \cap \giant \cap nQ } \big\|_{\ell^1} \,, \\
&\leq \sum_{Q\, \in\, \Delta^k \setminus \textsf{Q}} \big\| \1_{G_n \cap nQ} - \1_{ n(W_{p,d} +x ) \cap \giant \cap nQ } \big\|_{\ell^1} + \sum_{Q\, \in\, \textsf{Q}} \big\| \1_{G_n \cap nQ} - \1_{ n(W_{p,d} +x ) \cap \giant \cap nQ } \big\|_{\ell^1} \,,\\
&\leq \sum_{Q\, \in\, \Delta^k \setminus \textsf{Q}} \big\| \1_{G_n \cap nQ} - \1_{ n(W_{p,d} +x ) \cap \giant \cap nQ } \big\|_{\ell^1} + c(p,d) 2^{-dk} n^d \,,
\end{align}
where $c(p,d) >0$ accounts for the perimeter of $W_{p,d}$. For each $Q \in \Delta^k \setminus \textsf{Q}$, either $n(W_{p,d} +x )  \cap nQ =  nQ$ or $n(W_{p,d} +x ) \cap nQ = \emptyset$. From the definition \eqref{eq:2_metric} of $\dw$, 
\begin{align}
\big\| \1_{G_n} - \1_{ n(W_{p,d} +x ) \cap \giant } \big\|_{\ell^1}  &\leq 2^k |\Delta^k| n^d \dw( \mu_n, \ov{\nu}_{W_{p,d} +x} (n) ) + c(p,d) 2^{-dk} n^d \,.
\end{align}
Choose $k$ large depending on $\e$, and then $\e'$ small depending on $\e$ and $k$, using \eqref{eq:perchance_to_dream}, to conclude
\begin{align}
n^{-d} \big\| \1_{G_n} - \1_{ n(W_{p,d} +x ) \cap \giant } \big\|_{\ell^1} \leq \e \,.
\end{align} 
The choice of $k$ and $\e'$ do not depend on $G_n \in \cal{G}_n$, on $x \in M_n$ or on $\om \in \cal{E}_1 \cap \cal{E}_2$. For $\e'$ chosen this way according to $k$ and $\e$, for any $\om \in \cal{E}_1 \cap \cal{E}_2$, and $n \geq N(\om)$, 
\begin{align} 
\max_{G_n \in \cal{G}_n} \min_{x \in M_n} \left(n^{-d} \big\| \1_{G_n} - \1_{ n(W_{p,d} +x ) \cap \giant } \big\|_{\ell^1} \right) \leq \e \,.
\end{align}
We conclude that for any $\e >0$, 
\begin{align}
\prob_p \left( \limsup_{n \to \infty} \max_{G_n \in \cal{G}_n} \inf_{x \in \R^d} n^{-d} \big \| \1_{G_n} - \1_{\giant \cap (x + nW_{p,d})} \big \|_{\ell^1} \leq \e \right) =1 \,,
\end{align}
completing the proof. $\hfill\qed$\newline



\appendix

{\large\section{\B{Tools from percolation, graph theory and geometry}}\label{sec:perco}}

\subsection{Tools from percolation}

We present tools from percolation used throughout the paper, introducing the notation $\Lambda(n) := [-n,n]^d \cap \Z^d$. Proposition \ref{prop_BBHK} and its corollary control the size of open edge boundaries of large subgraphs of $\giant$. 

\begin{prop} (Berger-Biskup-Hoffman-Kozma \cite{BBHK}, Proposition A.2)\, Let $d \geq 2$ and $p > p_c(d)$. There are positive constants $c_1(p,d), c_2(p,d)$ and $c_3(p,d)$ so that for all $t > 0$, 
\begin{align}
\prob_p\left( \exists \Lambda \ni 0,\, \om\text{\rm-connected},\, |\Lambda| \geq t^{d / (d-1)},\, | \pa^\om \Lambda | < c_3 | \Lambda |^{(d-1) /d} \right) \leq c_1 \exp(- c_2 t) \,.
\end{align}
\label{prop_BBHK}
\end{prop}

The next corollary is similar to Proposition A.1 in \cite{BBHK}; we include the proof because it is short.

\begin{coro} Let $d \geq 2$, $p > p_c(d)$ and $\al \in (0,1)$. There are positive constants $c_1(p,d)$, $c_2(p,d)$, $c_3(p,d)$ and an almost surely finite random variable $R = R(\om)$ such that whenever $n \geq R$, we have the following lower bound on $|\pa^\om \Lambda|$ for each $\om$-connected $\Lambda$ satisfying $\Lambda \subset \textbf{\rm \textbf{C}}_{2n}$ and $|\Lambda | \geq n^\al$:
\begin{align}
| \pa^\om \Lambda | \geq c_3 | \Lambda |^{ (d-1) / d} \,.
\end{align}
Moreover, we have the following tail bounds on $R$:
\begin{align}
\prob_p( R > n) \leq c_1 n^d \exp \left( - c_2 n^{\al(d-1)/d} \right) \,.
\end{align}

\label{coro_BBHK}
\end{coro}

\begin{proof} Let $c_3$ be as in Proposition \ref{prop_BBHK}, and for $\al \in (01,)$ let $\cal{E}_n$ be the following event:
\begin{align}
\Big \{\text{$\exists \Lambda, \om$-connected with $\Lambda \subset \textbf{\rm \textbf{C}}_{2n}$,\, $|\Lambda| \geq n^{\al}$ but $| \pa^\om \Lambda | < c_3|\Lambda |^{(d-1) /d}$}  \Big\} \,. 
\end{align}
Apply Proposition \ref{prop_BBHK} to every point in the box $\Lambda(2n)$ with $t = n^{\al(d-1)/d}$ to obtain
\begin{align}
\prob_p(\cal{E}_n) \leq c_1 n^d \exp \left( - c_2 n^{\al(d-1)/d} \right) \,.
\end{align}
These probabilities are summable in $n$. Let $R$ be the (random) smallest natural number such that that $n \geq R$ implies $\cal{E}_n^c$ occurs. As $\{ R > n \} \subset \cal{E}_n$, the proof is complete. \end{proof}

We now introduce a tool for controlling the density of the infinite cluster within a large box. 

\begin{prop} (Durrett-Schonmann \cite{Durrett_Schonmann}, Gandolfi~\cite{Gandolfi})\, Let $d \geq 2$ and $p > p_c(d)$. Recall that  $\theta_p(d) = \prob_p( 0 \in \text{\rm{\B{C}}}_\infty)$ is the density of the infinite cluster. For any $\e >0$, there are positive constants $c_1(p,d,\e)$ and $c_2(p,d,\e)$ so that 
\begin{align}
\prob_p\left( \frac{| \giant |}{ |\Lambda(n) |} \notin (\theta_p(d) - \e, \theta_p(d) + \e) \right) \leq c_1 \exp \Big( - c_2 n^{d-1} \Big)\,.
\end{align}
\label{Gandolfi}
\end{prop}

Proposition \ref{Gandolfi} was later refined by Pisztora \cite{P}; the following is an immediate corollary. 

\begin{coro} Let $d \geq 2$ and $p > p_c(d)$. Let $r > 0$, let $Q \subset \R^d$ be a translate of the cube $[-r,r]^d$ and let $\e > 0$. There are positive constants $c_1(p,d,\e),c_2(p,d,\e)$ so that 
\begin{align}
\prob \left(  \frac{ | \text{\rm{\B{C}}}_\infty \cap Q |}{\cal{L}^d(Q) } \notin (\theta_p(d) - \e, \theta_p(d) + \e ) \right) \leq c_1 \exp \left(-c_2r^{d-1} \right) \,.
\end{align}
\label{density_control}
\end{coro}

The next result is fundamental, it is used in Section \ref{sec:final}.

\begin{thm} (Grimmett-Marstrand~\cite{GrMa})\, Let $d\geq 2$ and $p > p_c(d)$, and let ${\text{\rm \B{C}}}(0)$ denote the open cluster containing the origin. There is a positive constant $c(p)$ so that
\begin{align}
\prob_p( | \text{\rm {\B{C}}}(0)| = n) \leq \exp\left(- c n^{(d-1)/d} \right) \,.
\end{align}
\label{kzgm}
\end{thm}

We now apply these tools to the $G_n \in \cal{G}_n$, and begin with a basic observation.

\begin{lem} For all $n$, if $G_n \in \cal{G}_n$ is disconnected, then $G_n$ is a finite disjoint union of connected optimal subgraphs. 
\label{opt_conn}
\end{lem}

\begin{proof}The proof follows from the identity that for $a,b,c,d > 0$,
\begin{align}
\frac{a +b}{c +d} \geq \min\left( \frac{a}{c}, \frac{b}{d} \right) \,.
\end{align}
Remark \ref{rem:barry_white} implies the connected components of any $G_n \in \cal{G}_n$ have disjoint open edge boundaries. If $G_n$ is optimal and disconnected, decompose $G$ into two disjoint subgraphs $G_n'$ and $G_n''$ and we must have $\vp_{G_n} = \vp_{G_n'} = \vp_{G_n''}$. \end{proof}

We now use Corollary \ref{density_control} to obtain a high probability upper bound on $\Chee$.

\begin{lem} Let $d \geq 2$ and $p > p_c(d)$. There are positive constants $c_1(p,d), c_2(p,d)$ and $c_3'(p,d)$ so that 
\begin{align}
\prob_p ( \Chee > c_3' n^{-1} ) \leq c_1 \exp \left(-c_2 n^{d-1} \right) \,.
\end{align}
\label{chee_asym}
\end{lem}

\begin{proof} Abbreviate $\theta_p(d)$ as $\theta$ and work in the high probability event from Corollary \ref{density_control} for the box $[-r,r]^d$ with $r := n / 2(d!)^{1/d}$ for some $\e >0$. Also work in the corresponding high probability event for the box $[-n,n]^d$ with the same $\e$. Write $\e' \equiv \e / \theta$ and let $H_n$ be $\B{C}_\infty \cap [-r,r]^d$, so that
\begin{align}
(1-\e') \theta(2r)^d \leq |H_n| \leq (1 - \e') \theta (2n)^d \,
\end{align}
holds when $\e$ is taken small depending on $d$. Thus, $H_n$ is valid with volume on the order of $n^d$. The size of $\pa^\om H_n$ is at most the $\cal{H}^{d-1}$-measure of $\pa [-r,r]^d$, which completes the proof.  \end{proof}

We use Lemma \ref{chee_asym} with Corollary \ref{coro_BBHK} to bound the volume of any $G_n \in \cal{G}_n$ from below.

\begin{lem} Let $d\geq 2$ and $p > p_c(d)$. There are positive constants $c_1(p,d), c_2(p,d)$ and $\eta_1(p,d)$ so that,
\begin{align}
\prob_p(\exists G_n \in \cal{G}_n\, \text{\rm such that } |G_n| < \eta_1 n^d) \leq c_1 \exp \left(-c_2 n^{(d-1)/2d}\right)\,. 
\end{align}
\label{opt_macro}
\end{lem}

\begin{proof} Set $\al = 1/2$, and work in the intersection of the high probability events
\begin{align}
\{ \Chee \leq c_3' n^{-1} \} \cap \{ R \leq n \}
\end{align}
respectively from Lemma \ref{chee_asym} and Corollary \ref{coro_BBHK}. For $G_n \in \cal{G}_n$, use Lemma \ref{opt_conn} to extract from $G_n$ a connected subgraph $H_n \subset G_n$ with $H_n \in \cal{G}_n$. If $|H_n| \leq n^{1/2}$, that $H_n \subset \giant$ implies $\pa^\om H_n$ is non-empty, and hence that $\vp_{H_n} > n^{-1/2}$. This is impossible when $\Chee \leq c_3' n^{-1}$ and $n$ is large.

Suppose $|H_n| \geq n^{1/2}$, and use the event from Corollary \ref{coro_BBHK}:
\begin{align}
| \pa^\om H_n | \geq c_3 |H_n|^{(d-1)/d} \,.
\end{align}
Thus,
\begin{align}
c_3 |H_n|^{-1/d} \leq \vp_{H_n} \leq c_3' n^{-1} \,,
\end{align}
and the claim holds with $\eta_1 = (c_3 / c_3')^d$. \end{proof}

Using Lemma \ref{opt_macro} with Lemma \ref{opt_conn} and the bound $|\Lambda(n) | \leq (3 n)^d$, we deduce the following.

\begin{coro} Let $d \geq 2$ and $p > p_c(d)$. There are positive constants $c_1(p,d), c_2(p,d)$ and $\eta_4(p,d)$ so that 
\begin{align}
\prob_p\left(
\begin{matrix}  \text{\rm$\exists G_n \in \cal{G}_n$ such that the number of }\\
\text{\rm connected components of $G_n$ exceeds $\eta_4$} 
\end{matrix} 
\right) \leq c_1 \exp \left(-c_2 n^{(d-1)/2d} \right)  \,.
\end{align}
\label{conn_finite}
\end{coro}

Having established that Cheeger optimizers are usually volume order, we now exhibit control on the open edge boundary of each Cheeger optimizer. 

\begin{lem} Let $d \geq 2$ and $p > p_c(d)$. There are positive constants $c_1(p,d)$, $c_2(p,d)$ and $\eta_2(p,d)$, $\eta_3(p,d) $ so that 
\begin{align}
\prob_p\left( \exists \text{$G_n \in \cal{G}_n$ \rm so that } | \pa^\om G_n| < \eta_2 n^{d-1}\, \text{\rm or } |\pa^\om G_n | > \eta_3 n^{d-1} \right) \leq c_1 \exp \left(-c_2 n^{(d-1)/2d} \right) \,.
\end{align}
\label{surface_boundary}
\end{lem}

\begin{proof} Work in the high probability event $\{ \Chee \leq c_3' n^{-1} \}$ from Lemma \ref{chee_asym} and consider $G_n \in \cal{G}_n$. Set $\eta_3 = (c_3')3^d$. As $\vp_{G_n} \leq c_3'n^{-1}$ and $|G_n| \leq (3n)^d$,
\begin{align}
| \pa^\om G_n | \leq \eta_3 n^{d-1} \,.
\end{align}
To prove the second half of this lemma, set $\al = 1/2$ and work in the intersection of the events
\begin{align}
\Big\{R \leq n \Big\} \cap \Big\{\forall G_n \in \cal{G}_n, \text{ we have } |G_n| \geq \eta_1 n^d \Big\}
\end{align}
from Corollary \ref{coro_BBHK} and Lemma \ref{opt_macro}. Given $G_n \in \cal{G}_n$, extract through Lemma \ref{opt_conn} a subgraph $H_n \subset G_n$ which is connected and optimal. Within $\{ R \leq n \}$, 
\begin{align}
|\pa^\om G_n | \geq |\pa^\om H_n | \geq c_3 \left( \eta_1 n^d \right)^{(d-1)/d} \,,
\end{align}
and we set $\eta_2 = c_3 (\eta_1)^{(d-1)/d}$.  \end{proof}

\subsection{Tools from graph theory, approximation, miscellany.}

We begin by stating Tur{\'a}n's theorem, used in the proof of Proposition \ref{finite_per}. For a graph $(\rm{V},\rm{E})$, an \emph{independent} set of vertices $A \subset \rm{V}$ is a collection of vertices such that no two elements of $A$ are joined by an edge in $\rm{E}$. For a finite graph $(\rm{V},\rm{E})$, the \emph{independence number} of $(\rm{V},\rm{E})$ is 
\begin{align}
\al(\rm{V},\rm{E}) := \max \Big\{ |A| : A \text{ is an independent subset of } \rm{V} \Big\} \,.
\end{align}

\begin{thm} (\cite{Zh_Steele}, Lemma 6)\, For $(\rm{V},\rm{E})$ a finite graph with maximal degree $\delta$,
\begin{align}
\al(\rm{V},\rm{E}) \geq \frac{ |\rm{V}|}{\delta +1 } \,.
\end{align}
\label{turan}
\end{thm}

The next tool is standard, it gives bounds on the number of $*$-connected subsets of $\Z^d$ containing the origin. 

\begin{prop} (\cite{Grimmett}, Equation (4.24))\,There is a positive constant $c(d)$ so that the number of $*$-connected subsets of $\Z^d$ of size $s$ containing the origin is at most $[c(d)]^s$. 

\label{span_tree}
\end{prop}

Moving back to the continuum, we make a short remark.

\begin{rmk} We use the phrase isometric image throughout the paper. An isometry $i : \R^{d-1} \to \R^d$ is a function preserving Euclidean distances. Given $F \subset \R^d$, say $F$ is the \emph{isometric image} of a set $E \subset \R^{d-1}$ if there is an isometry $i : \R^{d-1} \to \R^d$ so that $i(E) = F$. 
\label{rem:isometric}
\end{rmk}

We now discuss the surface energy functional defined in \eqref{eq:new_surface_energy}.

\begin{lem} (\cite{stflour}, Section 14.2)\, Let $\tau$ be a norm on $\R^d$. The associated surface energy functional $\cal{I}_\tau$ is \emph{lower semicontinuous}: if $E_n$ is a sequence of Borel sets in $\R^d$ such that $\1_{E_n} \to \1_E$ in $L^1$-sense, then
\begin{align}
\cal{I}_\tau (E) \leq \liminf_{n\to \infty} \cal{I}_\tau (E_n)
\end{align}
\label{lsc} 
\end{lem}

A consequence of lower semicontinuity is the following approximation result.

\begin{prop} Consider the Wulff crystal $W_{p,d}$ of Theorem \ref{main_L1}. Given $\e >0$, there is a polytope $P_\e \subset W_{p,d}$ so that 

(i) $| \cal{I}_{p,d}(P_\e) - \cal{I}_{p,d}(W_{p,d}) | \leq \e$

(ii) $\cal{L}^d ( W_{p,d} \setminus P_\e) \leq \e $

\label{poly_limit}
\end{prop}

The last object we deal with is $\sfS$ from Section~\ref{sec:norm}, specifically the nicely varying property defined in Section \ref{sec:norm_3}. Let $\mathbb{S}_+^{d-1}$ denote the closed, upper hemisphere of the unit $(d-1)$-sphere $\mathbb{S}^{d-1}$. Write $T\mathbb{S}^{d-1}$ for the tangent bundle of $\mathbb{S}^{d-1}$, and let $T \mathbb{S}_+^{d-1}$ be the restriction of this bundle to the upper hemisphere. An \emph{orthonormal $k$-frame} on $\mathbb{S}_+^{d-1}$ is an assignment taking each $x \in \mathbb{S}_+^{d-1}$ to an ordered collection of $k$ orthonormal vectors in $T_x\mathbb{S}_+^{d-1}$, the tangent plane to $\mathbb{S}_+^{d-1}$ at $x$. This may be written in Euclidean coordinates due to natural embeddings of the sphere and tangent spaces into $\R^d$, and may thus be written as a function
\begin{align}
f : x \mapsto ( v_1(x) , \dots, v_k(x) ) \,,
\end{align}
with $x \in \mathbb{S}_+^{d-1} \subset \R^d$, and with each $v_i(x) \in T_x\mathbb{S}_+^{d-1} \subset \R^d$. 

\begin{prop} There is an orthonormal $(d-1)$-frame $f$ and a constant $C >0$ so that for $\e > 0$, whenever $x,y \in \mathbb{S}_+^{d-1}$ satisfy $|x - y|_2 \leq \e$, we have
\begin{align}
\big|(f(x))_i  - (f(y))_i \big|_2 \leq C\e
\end{align}
for all $i \in \{1, \dots, d-1\}$. 

\label{chosen}
\end{prop}

\begin{proof} Let $s \in \mathbb{S}^{d-1}$ denote the south pole, with coordinate representation $(0, \dots, 0, -1)$ in $\mathbb{R}^d$. Consider the standard stereographic projection $\pi : \mathbb{S}^{d-1} \setminus \{s \} \to \R^{d-1}$. The image of $\mathbb{S}_+^{d-1}$ under $\pi$ is a closed disc $D \subset \R^{d-1}$ centered at the origin. The disc $D$ is \emph{parallelizable}, that is, we may construct a smooth $(d-1)$-frame $g$ on $D$. Indeed, one can take the standard basis for $\R^{d-1}$at each tangent space $T_y D$. Define $f$ as the pullback $\pi^* g$. As $f$ varies smoothly over a compact domain, each of its coordinate functions is Lipschitz, which completes the proof. \end{proof}



\bibliographystyle{abbrv}
\bibliography{iso_sources}
\nocite{*}

\end{document}